\documentclass[12pt,a4paper,final]{article}
\usepackage[margin=25mm]{geometry}
\usepackage[T1]{fontenc}
\usepackage[utf8]{inputenc}

\usepackage{amsmath,amsfonts,amssymb,bef_alex,%
	datetime,todonotes,bm,relsize,tikz}
\usepackage[normalem]{ulem}
\usepackage{enumerate}
\usepackage[shortlabels]{enumitem}
\usepackage{comment,cancel}
\usepackage{mathrsfs}
\usepackage{mathtools}

\usepackage{hyperref}
\renewcommand{\ti}{{\times}}
\def\FG{\bm}
\newcommand{\sfC}{\mathsf C}
\newcommand{\oti}{{\otimes}}

\def\dd{\,\rmd}

\newcommand{\sfD}{\mathsf D}
\newcommand{\vp}{\varphi}
\newcommand{\Om}{\Omega}
\newcommand{\ve}{\varepsilon}
\newcommand{\la}{\langle}
\newcommand{\ra}{\rangle}
\newcommand{\pa}{\partial}
\newcommand{\dn}{\mathrm{d}}
\newcommand{\TS}{\textsf{V}}
\newcommand{\weakstar}{\overset{*}{\rightharpoonup}}
\newcommand{\wk}{\rightharpoonup}
\newcommand{\whh}{\widehat}
\newcommand{\psie}{\psi_{\text{ext}}}

\newcommand{\psiext}{\psi_\mathrm{ext}}
\newcommand{\pp}{\rmp}
\newcommand{\qq}{\rmq}
\newcommand{\mfHS}{\mathfrak H^*}
\newcommand{\grmR}{g_\rmR}
\newcommand{\FF}{\mathscr F}
\newcommand{\HH}{\mathscr H}

\newcommand{\coleq}{\mathrel{\mathop:}=}
\newcommand{\eqcol}{=\mathrel{\mathop:}}
\newcommand{\mquad}{\hspace{-1em}}

\newcommand{\ul}{\underline}

\newtheorem{hypotheses}[theorem]{Hypotheses}

\numberwithin{equation}{section}
\numberwithin{figure}{section}

\usepackage[notref,notcite,color]{showkeys}

\begin{document}
	
	\title{On the equilibrium solutions of \\ 
		electro--energy--reaction--diffusion systems}
	
	\author{ 
		Katharina Hopf\thanks{Weierstra\ss-Institut f\"ur Angewandte 
			Analysis und Stochastik, Anton-Wilhelm-Amo-Stra\ss e 39, 10117 Berlin, Germany} 
		\and
		Michael Kniely\footnotemark[1] \thanks{Current address: Universit\"at Graz, 
			Institut f\"ur Mathematik und Wissenschaftliches Rechnen, Heinrichstra\ss e 36, 8010 Graz, Austria, michael.kniely@uni-graz.at (corresponding author)}
		\and
		Alexander Mielke\footnotemark[1] \thanks{Humboldt-Universit\"at zu Berlin, Institut f\"ur Mathematik, Rudower Chaussee 25, 12489 Berlin, Germany}
	}
	
	\date{13 February 2026}
	
	\maketitle
	
	\begin{abstract}
        Electro--energy--reaction--diffusion systems are thermodynamically 
        consistent continuum models for reaction--diffusion processes that 
        account for temperature and electrostatic effects in a way that total 
        charge and energy are conserved. 
        The question of the long-time asymptotic behavior of 
        electro--energy--reaction--diffusion systems motivates the characterization 
        of their equilibrium solutions, which leads to a maximization problem of 
        the entropy on the manifold of states with fixed values for the 
        linear charge and the nonlinear convex energy functional. 
        As the main result, we establish the existence, uniqueness, and 
        regularity of solutions to this constrained optimization problem. 
        We give two conceptually different proofs, which are related to 
        different perspectives on the constrained maximization problem. 
        The first one is based on the method of Lagrange multipliers, while the  
        second one employs the direct method of the calculus of variations.
	\end{abstract}
	
	
	\normalsize
	
\section{Introduction}
\label{se:Intro}

Non-isothermal reaction--diffusion systems for electrically charged 
constituents appear in numerous situations both in natural sciences and 
engineering applications. For instance, describing chemical reactions of 
ions generally demands for a model taking also energetic effects into 
account. In particular, this is necessary in 
the case of endothermic and exothermic reactions, which, respectively, ``consume'' 
and ``produce'' heat during the reaction process. Another example, where reactive, 
diffusive, energetic, and electrostatic effects are essential ingredients, are the 
dynamics of charge carriers in high-energy semiconductor devices. 
	
In this contribution, we establish the existence of a unique 
global equilibrium for a system of $I \in \bbN$ electrically charged 
constituents, which are continuously distributed in a bounded Lipschitz domain 
$\Omega \subset \bbR^d$ with $d \geq 1$. 
We use $\bfc = (c_i)_i \in [0, \infty)^I$ to denote the 
corresponding vector of concentrations, and we employ the internal energy density 
$u \in [0, \infty)$ as the main thermodynamic variable of the system. 
One could also choose the temperature $\theta \in [0, \infty)$ instead of the 
internal energy, but this would lead to technical difficulties in our approach, 
which will be discussed below in more detail. 
Thus, our theory starts essentially from the modeling paper 
\cite{AlGaHu02TDEM} for thermodynamically consistent 
electro--energy--reaction--diffusion systems (EERDS),
but similar to \cite{Miel11GSRD,Miel13TMER,MieMit18CEER}, we
exploit the theory of gradient systems. The wording ``thermodynamically 
consistent'' refers to the fact that fundamental laws of physics and 
especially thermodynamics are satisfied -- in our situation, the conservation
of charge and energy as well as the production of entropy. 
	
Using the state variable $\bfZ =\bfZ(x)$, $\bfZ\coleq (\bfc, u)\in 
[0,\infty)^{I+1}$ and the dual entropy-production potential 
$\calP^*(\bfZ, \bfW)$ involving 
dual variables $\bfW=\bfW(x)$, $\bfW\coleq (\bfy, v) \in \bbR^{I+1}$ 
defined in \eqref{eq:dual-variables}, the evolution equation reads
\begin{align}
 \label{eq:general-system}
	\dot \bfZ = \pl_{\bfW} \calP^*\big(\bfZ; \rmD \calS(\bfZ)\big). 
\end{align}
Here, the entropy functional 
\begin{align}
 \label{eq:entropy-functional}
	\calS (\bfZ) = \int_\Omega S\big( \bfc(x), u(x) \big) \dd x
\end{align}
plays the role of the so-called driving functional for the gradient-flow
system \eqref{eq:general-system}. Fundamental thermodynamics 
\cite{LieYng99PMSL} demand that the entropy density 
$S : (0, \infty)^{I+1} \rightarrow \bbR$ is concave (and, hence, 
continuous); this concavity property is generally not satisfied if 
the entropy is defined as a 
function of the concentrations $\bfc$ and the temperature $\theta$. 

For this introduction, it suffices to note that the dual entropy-production potential 
\begin{align}
\label{eq:dual-entropy-production}
	\calP^*(\bfZ, \bfW) = \calP^*_\mathrm{diff}(\bfZ, \bfW) + \calP^*_\mathrm{reac}(\bfZ, \bfW)
\end{align}
is the sum of diffusive and reactive contributions, which are non-negative 
and consistent with the conservation of total energy $\calE$ and 
total charge $\calQ$ given by 
\[
  \calE(\bfc,u) \coleq \int_\Omega \Big( u + \frac\eps2|\nabla 
  \psi_\bfc|^2 \Big) \dd x \quad \text{and} \quad \calQ(\bfc)
  \coleq \int_\Omega \bfq {\cdot} \bfc \dd x, 
\]
where $\psi_\bfc$ is the solution of the Poisson equation 
\eqref{eq:poisson}, and where $\bfq \in \bbR^I$ denotes the vector 
of charge numbers for the individual species, e.g., $q_1 = -1$ and 
$q_2 = 1$ for electrons and holes in a semiconductor. We refer to 
Appendix \ref{app:TimeDepModel} for more details on the time-dependent 
EERDS \eqref{eq:general-system}. In particular, along positive and 
sufficiently smooth solutions of \eqref{eq:general-system}, one can easily show 
\[
  \frac{\rmd}{\rmd t} \calS(\bfZ) \geq 0, \quad
  \frac{\rmd}{\rmd t} \calE(\bfZ) = 0, \quad
  \frac{\rmd}{\rmd t} \calQ(\bfZ) = 0, 
\]
see \cite{Miel11GSRD,Miel13TMER,MieMit18CEER}. Thus, one expects 
that most solutions converge to a local maximizer of the entropy $\calS$ 
under the given constraints $\calE(\bfZ)=E_0$ and 
$\bfQ(\bfZ)=Q_0$, where $E_0$ and $Q_0$ are determined by the initial
conditions. If there is a unique global maximizer under the given 
constraints, then all solutions are expected to converge to it.  
This relates to the \emph{maximum-entropy principle}, 
see \cite[Sec.~3]{AlGaHu02TDEM}, which states that all steady states of 
\eqref{eq:general-system} are given as maximizers of $\calS$ under given 
values of the constraints $\calE$ and $\calQ$. In this paper, we make
this principle rigorous and show, under suitable technical assumptions, that
for each suitable pair $(E_0,Q_0)$ there exists a unique maximizer of $\calS$. 
It turns out that the choice $\bfZ=(\bfc,u)$ is crucial here, because $\calS$ 
is concave in these variables but, in general, not with respect to $(\bfc,\theta)$. 

In \cite[Thm.~5.2]{GliHun05SEMS}, the existence of thermodynamic equilibria
was established for a more specific two-dimensional model in a completely different way, 
namely by prescribing the (constant) temperature $\theta_0$ and one of the two 
(constant) electrochemical potentials $\zeta_1$. 
These two constants correspond to our $\eta$ and $\kappa$ used in Section 
\ref{se:proofs-LagrangianApproach} (see, e.g., \eqref{eq:k-functional}), which are dual Lagrange parameters to 
$E_0$ and $Q_0$, respectively, see Remark \ref{re:ConstantPotentials} for more explanation. 
Starting from $(\theta_0,\zeta_1)$ makes the construction of equilibria easier 
but does not uncover the relation to the maximum-entropy principle. 
However, \cite{GliHun05SEMS}  goes much further in another direction, 
by showing existence of steady-state solutions out of thermodynamic equilibrium
if the boundary conditions are close to those of a thermodynamic equilibrium solution. 
    
For a prescribed doping profile $D \in \rmL^{p_\Omega}(\Omega)$, 
$p_\Omega > \max\{1,d/2\}$, we define the electrostatic potential 
$\Psi : \Omega \rightarrow \bbR$ as the solution to Poisson's equation 
\begin{subequations}
    \label{eq:I.PoissonAll}
\begin{align}
	\label{eq:poisson}
	- \DIV \big( \eps \nabla \Psi \big) = \bfq {\cdot} \bfc + D, 
\end{align}
where the charge distribution on the right-hand side involves the charge vector 
$\bfq \in \bbR^I$. The function $\eps \in \rmL^\infty(\Omega)$, $\eps \geq \ul\eps > 0$, 
models the (uniformly positive) permittivity of the material. 
We endow \eqref{eq:poisson} with mixed Dirichlet--Robin boundary conditions 
\begin{align}
	\label{eq:poisson-bc}
	\Psi = 0 \ \ \text{on} \ \Gamma_\rmD \quad \text{and} \quad 
	\varepsilon \nu {\cdot} \nabla \Psi + \omega \Psi = \grmR 
	\ \ \text{on} \ \Gamma_\rmR
\end{align}
\end{subequations}
with boundary data $\grmR \in \rmL^{p_{\rmR}}(\Gamma_\rmR)$, 
$p_{\rmR} > \max\{1,d{-}1\}$ modeling a given surface charge density. 
We emphasize that our theory needs homogeneous Dirichlet data to allow for 
the variational approaches induced by the thermodynamic maximum-entropy principle, see also \cite[Thm.~5.2]{GliHun05SEMS} (where $\varphi_D=\Psi|_{\Gamma_\rmD}$ is assumed to be constant). 
The Dirichlet part of the boundary $\Gamma_\rmD \subset \partial \Omega$ 
	is supposed to be a measurable subset of $\pl \Omega$ and to either have a positive 
	measure $\calH^{d-1}(\Gamma_\rmD) > 0$ or to be empty. 
	The complement $\Gamma_\rmR \coleq \partial \Omega \backslash \Gamma_\rmD$ 
	denotes the Robin part of the boundary, and 
	the coefficient function $\omega \in \rmL^\infty(\Gamma_\rmR)$, 
	$\omega \geq 0$ even allows us to include pure Neumann boundary conditions 
	via $\omega \equiv 0$. 
	This kind of non-homogeneous mixed boundary conditions has already been 
	employed in \cite{AlGaHu02TDEM}. 
	Finally, we suppose that the following compatibility conditions hold: 
	\begin{align}
		\int_\Omega \bfq {\cdot} \bfc \dd x = 0 \quad \text{and} \quad 
		\int_\Omega D \dd x + \int_{\pl \Omega} \grmR \dd a = 0 
		&&&\text{if} \ \Gamma_\rmD = \emptyset \ \text{and} \ \omega \equiv 0. 
		\label{eq:compatibility-neumann}
	\end{align}
	The identities \eqref{eq:compatibility-neumann} are necessary to guarantee 
	the existence of solutions to \eqref{eq:I.PoissonAll} in the pure Neumann case.

	The interest in reaction--diffusion models taking temperature and electrostatic 
	effects into account increased over the last couple of years. The first 
	contribution towards a rigorous non-isothermal reaction--diffusion model for 
	electrons and holes in a semiconductor device appeared in \cite{Wach90RTTH}; 
	both the stationary and the time-dependent case are discussed therein. The same 
	model equations were subsequently derived in \cite{AlGaHu02TDEM} from a free 
	energy functional relying only on fundamental thermodynamical principles. 
	The generalization to an arbitrary number of charged species was performed in 
	\cite{Miel11GSRD} (see also \cite{Miel15TCQM}). 
	All these results focus on modeling issues and formal considerations. 
    A first existence result for stationary solutions was obtained in \cite{GliHun05SEMS}. To our 
	knowledge, the first global existence result for weak solutions the time-dependent setting
	was provided fairly recently in \cite{BuPoZa17EAIF}. In that paper, the authors 
	consider a thermodynamically and mechanically consistent fluid mixture of charged 
	constituents subject to reaction, drift-diffusion, and heat conduction. On the 
	one hand, the result covers a remarkably large class of models and electro-thermal 
	effects, e.g., the Soret and Dufour effects. On the other hand, the existence 
	analysis heavily relies on the assumption that the constituents are dissolved in a solvent, 
	which a priori leads to bounded concentrations and bounded reaction terms. 
	We also mention the recent note \cite{KanKop20NISG} on a thermodynamically consistent 
	generalization of the model in \cite{AlGaHu02TDEM} to arbitrary particle statistics.
	
	Our motivation for studying EERDS originates from the previous work \cite{FHKM22GEAE} 
	on global solutions to pure energy--reaction--diffusion systems and related publications 
	\cite{Miel11GSRD,HHMM18DEER,MieMit18CEER,Hopf_2022} on the large-time behavior and 
    weak--strong uniqueness of these systems. 
	Moreover, reaction--diffusion systems coupled to Poisson's equation were investigated in 
	\cite{FelKni18EMEC,FelKni21UCEF}, where exponential convergence to the equilibrium was shown. 

    The main result of this article on the existence of equilibrium states 
    is relevant for a large class of volume reactions and volume drift--diffusion processes which can be 
    expressed in the form \eqref{eq:general-system} with dual entropy-production potentials 
    $\calP^*_\mathrm{reac}(\bfZ, \bfW)$ and $\calP^*_\mathrm{diff}(\bfZ, \bfW)$ as in 
    \eqref{eq:dual-entropy-production}. We briefly discuss in Appendix \ref{app:TimeDepModel} 
    that reversible mass-action reactions and linear diffusion with electrostatic drift are 
    covered by our framework. In this setting, we also argue that stationary states of 
    the EERDS are (constrained) critical points of the total entropy functional $\calS$. 
    And as we prove below that any (constrained) critical point of $\calS$ coincides with 
    the unique global equilibrium of \eqref{eq:general-system}, we expect that any global 
    solution to \eqref{eq:general-system} tends to the equilibrium in the large-time limit. 
    There are several points which are beyond the scope of this paper and which require further research: 
    (i) the existence of global solutions to the evolutionary problem \eqref{eq:general-system}, 
    (ii) the large-time equilibration of global solutions, as well as (iii) reactions and 
    diffusion on interfaces and the boundary.
    
	We approach our goal of proving the existence of a unique 
    global equilibrium to \eqref{eq:general-system} 
    with two mostly independent strategies. The first one relies on the technique 
    of Lagrange multipliers proving that a unique critical point $(\bfc^\ast, u^\ast)$ 
    of $\calS$ subject to the constraints \eqref{eq:constraints} exists 
    (cf.\ Definition \ref{def:equilibrium}) in the class of continuous and 
    uniformly positive functions. Using the structure of the constraints, 
    we further show that the critical point is globally optimal within this class.
    The second strategy builds on the direct method of the calculus of variations, 
    which allows us to work in a setting of low regularity. 
    Assuming that the set of states $(\bfc, u)$ 
    satisfying the charge and energy constraints is non-empty, we establish
    the existence (and uniqueness) of a constrained maximizer of $\calS$ 
    in a rather straightforward manner by relying on the (strict) concavity of the 
    entropy and its monotonicity with respect to the internal energy component. 
    The continuity and uniform positivity of this optimizer is subsequently deduced
    by a regularization and limiting procedure. 
    
	A key functional appearing in both optimization approaches mentioned above 
    is the Legendre transform $H^*:=\mathfrak L(H)$ of the convex function $(\bfc,u)\mapsto H(\bfc,u) \coleq -S(\bfc,u)$, given by
    \begin{align}
		\label{eq:def-dual-entropy}
		H^\ast(\bfy, v) \coleq \sup \big\{\bfy {\cdot} \bfc + v u + S(\bfc, u) 
		\ \big| \ (\bfc, u) \in (0, \infty)^{I+1} \big\}.
	\end{align}
	The function $H^\ast$ is the so-called \emph{dual entropy} 
    defined for $\bfy \in \bbR^I$ and $v < 0$. We will later see that $H^\ast(\bfy, v)$ 
	appears, for instance, when looking for critical points of the Lagrange functional 
	$\calL$ related to the constrained maximization of $\calS$. The actual quantity 
	of interest, in this context, is the restriction of the dual entropy 
    $H^\ast$ to a two-dimensional subspace of $\bbR^{I+1}$ involving 
    $\mu \in \bbR$ and $\eta > 0$, the \emph{reduced dual entropy} 
	\begin{align}
		\label{eq:def-fraks-ast}
		\ol{H^*}(\mu, \eta) \coleq H^*({-}\mu \bfq, -\eta).
	\end{align}
    The relevant entropic quantities are summarized in Figure \ref{fig:entropies}. 
\begin{figure}[t]
		\centering
		\begin{tikzpicture}
           \draw[thick] (-3.5,-1) rectangle (7.5,1);
			\node (P1) at (-1.15,0) {$H(\boldsymbol c,u) = -S(\boldsymbol c,u)\ $} ;
			\node (P2) at (3,0) {$\ H^\ast(\boldsymbol y, v)\ $};
			\node (P4) at (6,0) {$\ \ol{H^*}(\mu, \eta)$};
			\draw[->] (P1) -- node [above,text width=1.5cm,midway,align=center]{\small $\mathfrak L$} (P2);
			\draw[->] (P2) -- (P4);
		\end{tikzpicture}
		\caption{Relation between the different types of convex entropy functions: 
        the negative entropy $-S = H : (0, \infty)^I \times (0, \infty) \rightarrow \bbR$, 
        the dual entropy $H^\ast : \bbR^I \times (-\infty, 0) \rightarrow \bbR$, and 
        the reduced dual entropy $\ol{H^*} : \bbR \times (0, \infty) \rightarrow \bbR$. 
        See \eqref{eq:def-dual-entropy} and \eqref{eq:def-fraks-ast} 
        for the consecutive definitions of $H^\ast$ and $\ol{H^*}$.}
		\label{fig:entropies}
    \end{figure}

    To make the discussion more concrete, we give in Examples \ref{ex:entropies} 
    and \ref{ex:second-entropies} below two methods to 
    derive admissible classes of entropy densities 
    $S(\bfc, u)$, which have been employed recently in studies of 
    energy--reaction--diffusion systems. Our main results in 
    Sections \ref{su:LagrangianApproach} and \ref{su:DirectMethod}, 
    however, hold for general entropy functions subject to 
    appropriate hypotheses formulated therein.
    Typical entropy densities $S(\bfc, u)$ consist of 
    Boltzmann-type functions $c_i \log c_i$ modeling the 
    concentration-based entropy. The purely thermal part $\sigma(u)$ of the entropy 
    is usually taken to be a sublinear function of $u$ to account for the desired 
    concavity of $S(\bfc, \cdot)$. 
    Moreover, equilibrium concentrations $\bfw(u) = (w_1(u), \dotsc, w_I(u))$ 
	enter the definition of $S(\bfc, u)$ in such a way that 
	\[
	\bfw(u) = \mathop{\mathrm{argmax}}_{\bfc \in (0, \infty)^I}\; S( {\cdot} ,u), 
    \quad \text{i.e.,} \quad \rmD_\bfc S(\bfw(u),u)=0 
	\]  
	by the concavity of $S(\cdot, u)$. 
    The interaction part of the entropy coupling concentrations 
    $c_i$ and the internal energy $u$ is chosen in such a way that the 
    condition above is satisfied. 
    The authors of \cite{MieMit18CEER} propose, for instance, the consistent 
    choice 
	\begin{align}
		S(\bfc,u) = \sigma(u) - \sum_{i=1}^{I} \big( c_i\log c_i - c_i - c_i\log
		w_i(u)\big),
	\label{eq:S.MiMi}\end{align}
	where the equilibrium densities $w_i:{[0,\infty)}\to {(0,\infty)}$ 
    are increasing and concave functions. 
    This class of entropy functions was subsequently employed in \cite{FHKM22GEAE,Hopf_2022} 
    in the context of renormalized solutions to energy--reaction--diffusion systems. 
    We will give explicit representations for $H^\ast$ and $\ol{H^*}$ in a 
    slightly simplified setting in Example \ref{ex:entropies}. 
	Let us now recall that 
	$
	\theta = (\rmD_u S(\bfc,u))^{-1} \eqcol \Theta(\bfc,u)
	$
	gives an expression for the temperature in terms of $\bfc$ and $u$, 
    which can be equivalently converted into $u=U(\bfc,\theta)$
	because $\rmD_u\Theta(\bfc, u)>0$ by the strict concavity of $S$. 
    At the (local) equilibrium, we expect 
    \[
        \Theta(\bfw(u), u) \ \text{constant in }\Omega, \quad \text{i.e.,} \quad 
        \rmD_u S(\bfw(u),u) \ \text{constant in }\Omega.
    \]
    This condition is obviously satisfied for any constant function 
    $u : \Omega \rightarrow (0, \infty)$. But already in this quite elementary 
    situation, it is not clear whether one can fulfill the charge and energy 
    constraints \eqref{eq:constraints} by adapting $u$ appropriately.

    The remainder of the paper is organized as follows. In Section \ref{se:electro-energy}, 
    we rigorously introduce the electrostatic potential and the electrostatic energy, 
    and we give some auxiliary results on the minimal value of the electrostatic energy. 
    More elaborate proofs are shifted to Section \ref{se:proofs-electro-energy} at the 
    end of the paper. The main results on the existence, uniqueness, and regularity 
    of equilibrium states are presented in Section \ref{se:results} -- 
    in Subsection \ref{su:LagrangianApproach} in the framework of the Lagrangian method, 
    and in Subsection \ref{su:DirectMethod} in the context of 
    the direct method of the calculus of variations. 
    We collect examples on admissible entropy densities in Subsection \ref{su:examples} 
    both in the case of Boltzmann statistics as well as in the case of bounded concentrations.
    The key proofs are carried out in Section \ref{se:proofs-LagrangianApproach} 
    and \ref{se:proofs-DirectMethod}, respectively. Some technical and more 
    involved proofs are contained in Section \ref{se:completing}, 
    while Appendix \ref{app:TimeDepModel} is devoted to a brief discussion 
    of the evolutionary EERDS and its connection to the equilibrium problem.

\section{Electrostatic and energetic principles}
\label{se:electro-energy}
  
\paragraph{Electrostatic potential.}
Following~\cite{AlGaHu02TDEM}, 
we decompose the total electrostatic potential $\Psi$ into 
the internal (or unbiased) and external (or induced) potential 
$\Psi = \psi_\bfc + \psi_\mathrm{ext}$,
where the internal potential $\psi_\bfc$  is determined by 
\begin{subequations}
\label{eq:poisson-parts}
\begin{align}
	 \nonumber &&&&-\DIV \big( \varepsilon \nabla \psi_\bfc \big) &= \bfq {\cdot} \bfc &&\text{in} \ \Omega,
 \\	\label{eq:poisson-parts.1} &&&& \psi_\bfc &= 0 &&\text{on} \ \Gamma_\rmD,
 \\	\nonumber &&&&\varepsilon \nu {\cdot} \nabla \psi_\bfc + \omega \psi_\bfc &= 0&&\text{on} \ \Gamma_\rmR, &&&& 
\end{align}
while the external potential $\psi_\mathrm{ext}$ satisfies
\begin{align}
	&&&& -\DIV \big( \varepsilon \nabla \psi_\mathrm{ext} \big) &= D &&\text{in} \ \Omega, &&&& \nonumber \\ 
	&&&& \psi_\mathrm{ext} &= 0 &&\text{on} \ \Gamma_\rmD, &&&& \label{eq:poisson-parts.2} \\ 
 	&&&&\varepsilon \nu {\cdot} \nabla \psi_\mathrm{ext} + \omega \psi_\mathrm{ext} &= \grmR 
 	&&\text{on} \ \Gamma_\rmR. &&&&  \nonumber
\end{align}
\end{subequations}
Recall that $\Omega \subset \mathbb R^d$, $d \geq 1$, is a bounded Lipschitz domain. 
In addition, the Dirichlet boundary $\Gamma_\rmD \subset \partial \Omega$ is either empty 
or has a positive measure $\calH^{d-1}(\Gamma_\rmD) > 0$, while the Robin boundary equals 
$\Gamma_\rmR \coleq \partial \Omega \backslash \Gamma_\rmD$. 
We impose the following hypotheses on the coefficient functions:
\begin{enumerate}[label={\rm(D\arabic*)}]
 	\item\label{hp:eps} $\eps \in \rmL^\infty(\Omega)$, $\eps \geq \ul\eps$ for some 
    $\ul\ve\in\mathbb R_{>0}$. 
 	\item\label{hp:varrho}   $\omega \in \rmL^\infty(\Gamma_\rmR)$ with
 	$\omega \geq 0$. 
\end{enumerate}
The subsequent definition introduces the natural function space for 
the internal potential $\psi_\bfc$ and the external potential $\psiext$.
	
\begin{definition}[Space $\HH$]
	\label{def:h1-gammad}
	Assume~\ref{hp:eps}, \ref{hp:varrho}.
    We define $\HH \subseteq \rmH^1(\Omega)$ as the subspace
		\begin{align}
			\label{eq:def-h1-gammad-rho}
			\HH \coleq 
			\begin{cases}
				\big\{ \psi \in \rmH^1(\Omega) \ \big| \ \psi|_{\Gamma_\rmD} = 0 \}, 
				\quad &\text{if} \ \calH^{d-1}(\Gamma_\rmD) > 0, \\ 
				\big\{ \psi \in \rmH^1(\Omega) \ \big| \ \int_\Omega \psi \dd x = 0 \}, 
				\quad &\text{if} \ \Gamma_\rmD = \emptyset \ \text{and} \ \omega \equiv 0, \\ 
				\rmH^1(\Omega), 
				\quad &\text{if} \ \Gamma_\rmD = \emptyset \ \text{and} \ \omega \not\equiv 0. 
			\end{cases}
		\end{align}
\end{definition} 
Note that the subspace $\HH \subseteq \rmH^1(\Omega)$ is closed and hence itself a Hilbert space. 
\begin{lemma}[Poisson: existence]
	\label{lemma:poisson.gen}Suppose~\ref{hp:eps}, \ref{hp:varrho}.
    For $\psi, \vp \in \HH$, we define 
\begin{align}
\label{eq:def-l}
   \la L\psi,\vp\ra \coleq  \int_\Om\ve\nabla\psi\cdot\nabla\vp \, \dd x
   +\int_{\Gamma_\rmR} \omega \psi\vp \, \dd a.
\end{align}
Then, the induced bounded linear operator $ L: \HH \to \HH^*$ is invertible.
\end{lemma}
\begin{proof}
    This follows from the classical Riesz representation/Lax--Milgram theorem 
    by observing that the definition of the space $\HH$ ensures that the 
    continuous bilinear form $\la L\cdot,\cdot\ra$ is coercive. 
    More precisely, there exists $\alpha_0>0$ such that for all $\HH$, we have  
    \begin{align}\label{eq:Lcoerc}
    \la L\psi,\psi\ra\ge\alpha_0\|\psi\|_{\rmH^1(\Om)}^2.
    \end{align}
    The existence of such a constant $\alpha_0>0$ can be shown by a contradiction argument.
\end{proof}
As a consequence of Lemma~\ref{lemma:poisson.gen}, whenever $\bfq{\cdot} \bfc\in \HH^*$, 
the linear Poisson problem~\eqref{eq:poisson-parts.1} for the internal potential possesses 
a unique weak solution $\psi_\bfc\in \HH$, i.e., there exists a unique $\psi_\bfc\in \HH$ 
such that $\la L\psi_\bfc,\vp\ra =\la \bfq{\cdot}\bfc,\vp\ra$ for all $\vp\in \HH$.

The external potential $\psi_\mathrm{ext}$ will, for simplicity, 
be considered in a more regular setting.  
\begin{enumerate}[label={\rm(D\arabic*)},resume]
	\item\label{hp:data} $D \in \rmL^{p_\Omega}(\Omega)$ and 
    $\grmR \in \rmL^{p_\rmR}(\Gamma_\rmR)$ for certain $p_\Omega > \max\{1,d/2\}$,  
    $p_\rmR > \max\{1,d-1\}$. If $\Gamma_\rmD = \emptyset$ and $\omega \equiv 0$, 
    then let further $\int_\Omega D \dd x + \int_{\pl \Omega} \grmR \dd a = 0$ hold. 
\end{enumerate}

\begin{lemma}[Poisson: regularity]
	\label{lemma:poisson}Suppose \ref{hp:eps}--\ref{hp:data}.
	Then, the Poisson equation \eqref{eq:poisson-parts.2} has a 
    unique solution $\psi_\mathrm{ext} \in \HH \cap \rmC(\ol \Omega)$. 
	\end{lemma}
\begin{proof}
	The assertions follow from the presentation in \cite[Chap.\,4]{Trol10OCPD}, which is 
	carried out in the pure Robin case but which also applies to our 
	situation with mixed boundary conditions. 
\end{proof}

\paragraph{Electrostatic energy.} 
We define the bilinear form $\bbB : \HH \times \HH 
\rightarrow \bbR$, 
\begin{align}
\label{eq:def-b-func}
\bbB(\psi,\phi)\coleq \int_\Omega \eps\nabla \psi \cdot \nabla \phi \dd x
+ \int_{\Gamma_\rmR} \omega \psi \phi \dd a
\qquad \text{and set}\quad  \calB(\psi)\coleq \bbB (\psi,\psi). 
\end{align}
Note that $\bbB : \HH \times \HH \rightarrow \bbR$ 
is coercive and related to the linear operator $L:\HH\to \HH^*$ via 
\[
\bbB(\psi,\phi) = \langle L\psi,\phi\rangle_\HH,
\]
where $\langle \,\cdot\,,\,\cdot\,\rangle_\HH$ is the duality pairing in
$\HH^*\ti \HH$. For a given charge distribution
$\rho \in \rmL^1(\Omega)\cap \HH^*$, the internal potential
$\psi_\rho$ is the solution of the weak problem
\begin{align}
\label{eq:weak-poisson}
\bbB(\psi_\rho,\phi) = \int_\Omega \rho\phi\dd x =\langle \rho,\phi\rangle_\HH 
\quad \text{ for all } \phi\in \HH.
\end{align}
The electrostatic energy for the data $(\rho,\psiext)$ is then introduced as 
\begin{equation}
  \label{eq:E.rho.psiext}
  \rmE(\rho,\psiext) \coleq \frac12\calB(\psi_\rho{+} \psiext) 
  = \frac12\calB(\psi_\rho) + \langle \rho, \psiext \rangle_\HH + 
  \frac12\calB(\psiext) \geq 0. 
\end{equation}
	The first two terms of the total electrostatic energy \eqref{eq:E.rho.psiext} account 
	for the energy of the charge density $\rho$ in the 
	self-consistent potential 
    $\psi_\rho = L^{-1}\rho$ 
    and the $\rho$-independent 
	potential $\psi_\mathrm{ext}$, where $\psi_{\rho}$ and $\psi_\mathrm{ext}$ 
	are introduced in \eqref{eq:weak-poisson} and \eqref{eq:poisson-parts.2}. 
    The last term in \eqref{eq:E.rho.psiext} can be interpreted as the ``energy'' 
    of the external potential $\psiext$. 
	Specifying $\rho \coleq \bfq {\cdot} \bfc$ and adding the integral over $u$, 
    i.e., the total internal energy, leads to the total energy $\calE(\bfc, u)$. 
    The energy and charge functionals are thus defined as 
	\begin{subequations}
		\label{eq:constraints}
		\begin{align}
			\calE(\bfc, u) &\coleq \frac12\calB(\psi_{\bfc} {+} \psiext) + \int_\Omega u \dd x = E_0, 
			\label{eq:constraint-energy} \\
			\calQ(\bfc, u) &\coleq \int_\Omega \bfq {\cdot} \bfc \, \dd x = Q_0,  
			\label{eq:constraint-charge} 
		\end{align}
	\end{subequations}
    where we recall the identities $\psi_\bfc = \psi_{\bfq{\cdot}\bfc} = L^{-1}(\bfq{\cdot}\bfc)$ 
    according to \eqref{eq:poisson-parts.1}, \eqref{eq:weak-poisson}, and \eqref{eq:def-l}.
	The total charge $\calQ(\bfc, u)$ only depends on the dynamic charge 
	density $\bfq {\cdot} \bfc$. If there are further stoichiometric constraints, 
    additional conservation laws arise, which have to be included as well; 
    for the sake of a simplified presentation, we shall assume that no 
	further conservation laws are present. 

Invoking the homogeneous Dirichlet condition $\psiext = 0$ on $\Gamma_\rmD$, 
the Fr\'echet differentials of $\calE(\bfc,u)$ and $\calQ(\bfc,u)$ are easily seen to equal 
	\begin{align}
		\label{eq:derivatives-energy-charge}
		\rmD\calE(\bfc,u) = \binom{(\psi_{\bfc} {+} \psi_\mathrm{ext}) \bfq}{1} 
		\quad \text{and} \quad 
		\rmD\calQ(\bfc,u) = \binom{\bfq}{0}.
	\end{align}

\paragraph{Optimal charge distribution.}
In this part, we derive a sharp 
lower bound for the electrostatic energy
$\rmE(\rho,\psiext)$ depending on the given external field $\psiext$ 
and the total charge $Q_0$ as introduced in \eqref{eq:E.rho.psiext}. 

For finding the infimum of $\rmE(\rho,\psiext)$ under the constraint
$\int_\Omega \rho\dd x = Q_0$, it is helpful to relax the problem by admitting
all $\rho \in \HH^*$. However, the latter charge constraint is only well defined
on $\HH^*$ if the constant function $1_\Omega$ lies in $\HH$, which is neither the case
if we have a (homogeneous) Dirichlet boundary condition on a nontrival part
$\Gamma_\rmD$ nor in the case of pure Neumann boundary conditions. 
But we will see that the charge constraint is only relevant in the remaining 
pure Robin case. The following lemma will turn out to be useful subsequently. 
For later usage, we provide a stronger version of the density statement. 

\begin{lemma}
\label{lemma:l1-density}
Let $\HH$ be the space introduced in 
Definition \ref{def:h1-gammad}. Then, the following facts hold true. 
\begin{enumerate}[label={\rm(F\arabic*)}]
\item
$\rmC(\ol \Omega)$ is dense in $\HH^*$, 
in particular, $\rmL^1(\Omega) \cap \HH^*$ is dense in $\HH^*$. 
  \label{Fact1}
\item
If $\calH^{d-1}(\Gamma_\rmD)>0$, then there exists 
a sequence $\rho_n \in \rmL^1(\Omega)\cap \HH^*$ with $\int_\Omega \rho_n \dd x = 1$ 
for all $n\in \N$ and $\| \rho_n\|_{\HH^*}\to 0$ for $n \to \infty  $.  
  \label{Fact2}
\end{enumerate}
\end{lemma}

A physical interpretation of the sequence $\rho_n$ in \ref{Fact2} is the accumulation of all
charge near the Dirichlet boundary. The point is that the relaxation allows for
surface charges, whereas the original problem is restricted to volume charges. 
More precisely, in the proof of \ref{Fact2}, we will show that the sequence can be chosen such that 
\begin{align}
\label{eq:def-rhon-an}
\rho_n(x) \coleq \frac{1}{|A_n|} \, \bm1_{A_n}(x) \ \text{ with } \ 
|A_n|= \frac{C}{n} \ \text{ and } \ \mafo{dist}(A_n, \Gamma_\rmD  ) \leq \frac1n. 
\end{align}

Detailed proofs of Lemma~\ref{lemma:l1-density} and Proposition~\ref{pr:MinElectroEner} below are given in Section~\ref{se:proofs-electro-energy}.

\begin{proposition}[Minimal electrostatic energy]
\label{pr:MinElectroEner} For $\rmE$ given by \eqref{eq:E.rho.psiext}, we
define on $\R\ti \HH$ the functional 
\begin{align}
    \label{eq:def-rmV}
\rmV(Q_0,\psiext) \coleq \inf\Bigset{ \rmE(\rho,\psiext)}{ 
\rho \in \rmL^1(\Omega)\cap \HH^*,\ \int_\Omega \rho \dd x = Q_0 } .
\end{align}
\\[0.2em]
(A) (Some Dirichlet) For $\calH^{d-1}(\Gamma_\rmD)>0$, we have $\rmV(Q_0,\psiext)=0$.
\\[0.2em]
(B) (Pure Neumann) For $\Gamma_\rmR = \pl\Omega$ and $\omega \equiv 0$, we have
$\rmV(0,\psiext) = 0 $. 
\\[0.2em]
(C) (Pure Robin) For $\Gamma_\rmR = \pl\Omega$ and $\int_\Gamma\omega \dd a>0$, we have
\begin{equation}
  \label{eq:V.Q.psiext}
  \rmV(Q_0,\psiext) = \frac1{2\int_\Gamma \omega \dd a} \Big( Q_0 + \int_\Gamma
  \omega \psiext \dd a\Big)^2.
\end{equation}

Denote by $N:\rmL^2(\Omega)\ti \rmL^2(\pl\Omega) \to \HH^*$ 
the mapping defined via 
\[
\big\langle N(D,\grmR), \phi\big\rangle_\HH 
\coleq \int_\Omega D\phi \dd x + \int_{\pl\Omega} \grmR \phi \dd a \quad 
\text{for all } \phi\in \HH
\]
and let $M\coleq L^{-1} N: \rmL^2(\Omega)\ti \rmL^2(\pl\Omega) \to \HH$. Then, for
$\psiext \in \HH$ and $\psiext= M(D,\grmR)$, respectively, the minimum is
attained at the relaxed $\wt\rho_*\in \HH^*$ given as follows:
\\[0.2em]
(A) (Some Dirichlet) $\wt\rho_*=-L\psiext$ and $\wt\rho_*=N({-}D,{-}\grmR)$,
respectively.
\\[0.2em]
(B) (Pure Neumann) $\wt\rho_*=-L\psiext$ and $\wt\rho_*=N({-}D,{-}\grmR)$,
respectively.
\\[0.2em]
(C) (Pure Robin) $\wt\rho_*=L(\kappa_*{-}\psiext)$ and
$\wt\rho_*=N({-}D,\kappa_*\omega{-}\grmR)$, respectively, \\
\hspace*{2em} where
$\kappa_*= \big( Q_0 + \int_\Gamma \omega \psiext \dd a\big) / 
\int_\Gamma \omega \dd a$.
\end{proposition}

\begin{remark}[Surface charges and attainment]
\label{rem:Attainment}
The question of attainment of the infimum in the definition of $\rmV(Q_0,\psiext)$
by a $\rho_\mafo{att}\in \rmL^1(\Omega) \cap \HH^*$ can be read off from the last
attainment statements.

We have attainment if and only if $\wt\rho_*=N(\rho_\mafo{att},0)$ and
$\int_\Omega \rho_\mafo{att} \dd x = Q_0$. Hence, in the three cases we have 
attainment if and only if:
\\[0.2em]
(A) (Some Dirichlet) $\grmR\equiv 0$ and $\int_\Omega D \dd x=-Q_0$.
\\[0.2em]
(B) (Pure Neumann) $\grmR\equiv 0$ and $\int_\Omega D \dd x=0=Q_0$.
\\[0.2em]
(C) (Pure Robin) $\int_{\pl\Omega}\omega \dd a \:\grmR = \big(Q_0 +
\int_{\pl\Omega} \omega \psiext \dd a \big) \omega$. \medskip

In all three cases, we have the necessary attainment condition that there are
no surface charges, i.e., $\wt\rho_*=N(-D,0)$. In (A), there is the additional
condition that the total doping charge compensates the given charge. 
\end{remark}

\begin{remark}[Restriction to positive charge]
\label{rem:rmV-geq}
    We emphasize that the statements of Proposition \ref{pr:MinElectroEner} 
    and Remark \ref{rem:Attainment} on the minimal electrostatic energy heavily rely on 
    $\rho = \bfq {\cdot} \bfc$ being allowed to take positive \emph{and} negative values. 
    In other words, we suppose that positive \emph{and} negative charge carriers are present, 
    i.e., there exist $1 \leq i, j \leq I$ such that $q_i < 0 < q_j$.
    
    If all species are non-negatively charged, i.e., if $q_i \geq 0$ for all $1 \leq i \leq I$, 
    we have to replace $\rmL^1(\Omega)$ in \eqref{eq:def-rmV} by $\rmL^1_+(\Omega)$, 
    and we conjecture that 
    \begin{align*} 
        \rmV_\geq (Q_0, \psiext) 
        \coleq \inf\Bigset{ \rmE(\rho,\psiext)}
        { \rho \in \rmL^1_+(\Omega) \cap \HH^*,\ \int_\Omega \rho \dd x = Q_0 }
        \gneqq \rmV(Q_0, \psiext).
    \end{align*}
\end{remark}

\section{Main results}
\label{se:results}
    \subsection{Local optimization via the Lagrange multiplier method}
    \label{su:LagrangianApproach}

    Before stating our main results on the existence, uniqueness, and 
    regularity of a (local) maximizer of the total entropy $\calS$ subject to 
    \eqref{eq:constraints} in Theorem \ref{thm:existence-uniqueness} 
    and Corollary \ref{cor:equilibrium}, 
    we introduce the underlying hypotheses and the precise definitions 
    of local equilibria and critical points.     

    The following conditions are the main properties, which are supposed to hold for 
    the entropy function $S(\bfc, u)$ and the dual entropy density $H^\ast(\bfy, v)$. 
    Here and throughout the article, we use $\rmD$ to denote the Gateaux 
    resp.\ Fr\'echet derivative.

    \begin{hypotheses}
    \label{hypo:entropy}
    Assume that the entropy density 
	$S : (0, \infty)^{I+1} \rightarrow \bbR$ and the dual entropy density 
	$H^\ast : \bbR^I \times (-\infty, 0) \rightarrow \bbR$ defined in 
    \eqref{eq:def-dual-entropy} fulfil the subsequent requirements:
    \begin{enumerate}
        \item $-S$ and $H^\ast$ are continuously differentiable and strictly convex functions. 
        \item $-\rmD S : (0, \infty)^{I+1} \rightarrow \bbR^I \times (-\infty, 0)$ 
        is a homeomorphism. 
    \end{enumerate}
    \end{hypotheses}
	Under Hypotheses \ref{hypo:entropy}, $-\rmD S$ generates a 
	transformation between the primal variables 
    $(\bfc, u) = \rmD H^\ast(\bfy, v) \in (0, \infty)^{I+1}$ 
	and the dual variables 
    \begin{align}
    \label{eq:dual-variables}
        (\bfy, v) = -\rmD S(\bfc, u) \in \bbR^I \times (-\infty, 0), 
    \end{align}
    where $-\rmD S(\bfc, u)$ and $\rmD H^\ast(\bfy, v)$ are strictly monotone in each argument. 
	The fact that $\rmD H^\ast$ is the inverse function to $-\rmD S = \rmD H$ is a 
    consequence of the Fenchel equivalences known from convex analysis. 
    The components of $\bfy = (y_1, \dotsc, y_I)$ 
	are related to the \emph{chemical potentials} $\zeta_i$ via 
    $y_i = - \zeta_i /\theta$ (cf.\ \cite[Def.~3.1]{AlGaHu02TDEM}), 
	whereas $v = -\rmD_u S(\bfc, u) < 0$ is the negative \emph{inverse temperature}. 
	Obviously, $S(\bfc, \cdot)$ is strictly increasing for all $\bfc \in (0, \infty)^I$. 
    Note that the class of entropy densities in Example \ref{ex:entropies} 
    satisfies the assumptions in Hypotheses \ref{hypo:entropy}. 

    The following hypotheses collect assumptions on the prescribed total energy 
    $E_0 \geq 0$ and total charge $Q_0 \in \bbR$ according to necessary 
    constraints stated in Proposition \ref{pr:MinElectroEner}. 

    \begin{hypotheses}
    \label{hypo:bc}
        Let $Q_0 \in \bbR$ and $E_0 > \rmV(Q_0, \psiext)$ 
        as defined in Proposition \ref{pr:MinElectroEner}. 
        If $\Gamma_\rmD = \emptyset$ and $\omega \equiv 0$, 
        i.e., in the pure Neumann case, we additionally set $Q_0 \coleq 0$. 
    \end{hypotheses}

    In the present context of the approach with Lagrange multipliers, 
    we consider continuous and uniformly positive concentrations 
    and internal energy densities 
    \begin{align}
		\label{eq:function-space}
        (\bfc, u) \in \FF \coleq \rmC \big(\ol \Omega, (0, \infty)^{I+1} \big). 
	\end{align}
    The main reason for this choice is that it ensures Fr\'echet differentiability 
    of the entropy functional $\calS=\calS(\bfc, u)$.
    At the same time, as we will show in Theorem \ref{thm:existence-uniqueness}, 
    there exists a unique constrained critical point of $\calS$ in $\FF$. 
    Let us also note that the constraints \eqref{eq:constraints} are 
    Fr\'echet differentiable in $\FF$. 

    We refer to \cite[Lem.~6.1]{AlGaHu02TDEM} for the first mathematically precise 
	formulation of the following equilibrium principle. 
	
	\begin{definition}[Local/global equilibria, critical points]
		\label{def:equilibrium} 
        Let $(\bfc^\ast, u^\ast) \in \FF$ satisfy \eqref{eq:constraints}. 
		We call $(\bfc^\ast, u^\ast)$ a \emph{local equilibrium} of 
        \eqref{eq:general-system} compatible with total energy 
		$E_0 > 0$ and total charge $Q_0 \in \bbR$ if the entropy $\calS$ 
		attains its supremum in a $\rmC^0$ neighborhood of $(\bfc^\ast, u^\ast)$ 
        in $\FF$ under the constraints 
		\eqref{eq:constraints} at $(\bfc^\ast, u^\ast)$. More precisely, we demand 
        that a constant $\delta > 0$ exists such that 
		\begin{align*}
			\calS(\bfc^\ast, u^\ast) = \sup
            \big\{ \calS(\bfc, u) \ \big| \ 
			&\calE(\bfc, u) = E_0, \ \calQ(\bfc, u) = Q_0, \\ 
			&(\bfc, u) \in \FF, \ 
            \| (\bfc {-} \bfc^\ast, u {-} u^\ast) \|_{\rmC(\ol\Omega)} < \delta \big\}. 
		\end{align*}
        We call $(\bfc^\ast, u^\ast)$ a \emph{global equilibrium} of 
        \eqref{eq:general-system} compatible with total energy 
		$E_0 > 0$ and total charge $Q_0 \in \bbR$ if the entropy $\calS$ 
		attains its supremum in $\FF$ under the constraints 
		\eqref{eq:constraints} at $(\bfc^\ast, u^\ast)$. 
        More precisely, we demand that 
		\begin{align*}
			\calS(\bfc^\ast, u^\ast) = \sup
            \big\{ \calS(\bfc, u) \ \big| \ \calE(\bfc, u) = E_0, \ 
            \calQ(\bfc, u) = Q_0, \ (\bfc, u) \in \FF \big\}. 
		\end{align*}
        We say that $(\bfc^\ast, u^\ast)$ is a 
		\emph{critical point} of the total entropy $\calS$ subject to 
		\eqref{eq:constraints} if and only if 
        there exist constants $\eta, \kappa \in \bbR$ such that 
		\begin{align}\label{eq:def.critical.pt}
			\rmD \calS(\bfc^\ast, u^\ast) = \eta \rmD \calE(\bfc^\ast, u^\ast) 
			+ \kappa \rmD \calQ(\bfc^\ast, u^\ast). 
		\end{align}
	\end{definition}

Recall that $\calS (\bfc, u) \in \bbR$ is well-defined for all 
$(\bfc, u) \in \FF$. Besides, since the energy constraint is 
nonlinear (note the quadratic electrostatic energy term $|\nabla \psi_\bfc|^2$), 
it is not immediately clear whether a unique constrained 
local maximizer of the total entropy $\calS$ exists. 
Studying the large-time behavior of \eqref{eq:general-system} 
(which will be the subject of a future work), 
we are particularly interested in proving convergence 
into a \emph{unique} local equilibrium for time $t \to \infty$. 
Indeed, we will show that there even exists a \emph{unique} 
critical point $(\bfc^\ast, u^\ast) \in \FF$ 
of $\calS$ under the constraints \eqref{eq:constraints}.

\begin{remark}[Constant temperature and electrochemical potentials]
\label{re:ConstantPotentials}
In \cite[Thm.~5.2]{GliHun05SEMS}, thermodynamical equilibria are defined by 
the assumption of constant temperature and constant electrochemical potentials. 
This property is equivalent to our characterization in \eqref{eq:def.critical.pt}. 
Indeed using \eqref{eq:derivatives-energy-charge}, the last component of 
\eqref{eq:def.critical.pt} gives 
\[
\frac{1}{\theta} = \rmD_u S(\bfc,u)= \eta, \quad \text{namely } \theta \equiv \frac{1}{\eta},
\]
which means that the temperature is equal to the constant $\theta_0 \coleq 1/\eta$. 

The vector of electrochemical potentials $\bfzeta$ is defined as 
$\bfzeta \coleq \theta \rmD_\bfc S(\bfc,u) - \psi \bfq$, and \eqref{eq:def.critical.pt} yields 
\[
\bfzeta = \frac1\eta \;\! \rmD_\bfc S(\bfc,u) - \psi \bfq 
= \frac1\eta \big( \rmD_\bfc S(\bfc,u) - \eta \psi \bfq\big) \equiv \frac\kappa\eta \bfq .
\]
This means that all electrochemical potentials 
$\zeta_i=\theta \rmD_{c_i} S(\bfc,u) -q_i \psi$ are constant, 
and all constants are the same up to a multiplication with the charge numbers 
$q_i$, namely $\zeta_i=\zeta_1 q_i/q_1$. 
\end{remark}

\begin{proposition}[Critical point versus local maximizer]\label{prop:crit=equil}
The following holds true:
\begin{enumerate}
    \item Any local equilibrium solution of \eqref{eq:general-system} with respect to fixed total energy $E_0$ and total charge $Q_0$ is a constrained critical point of $\calS$. 
    \item If $S\in \rmC^2((0,\infty)^{I+1})$ is locally strongly concave in the sense that $-\rmD^2S(\bfc,u)$ is positive definite for all $(\bfc,u)\in (0,\infty)^{I+1}$, then any constrained critical point of $\calS$, i.e., any solution of~\eqref{eq:def.critical.pt}, is a strict local maximizer of $\calS$ with respect to the energy and charge constraints.
\end{enumerate}
\end{proposition}

This result strongly relies on the structure and regularity of the constraints. In particular, we take advantage of the fact that the set of admissible states satisfying the energy and charge constraints forms a smooth manifold in the Banach space $\rmC(\ol\Omega)^{I+1}$ and that an extended version thereof, where states are allowed to take values outside $(0,\infty)^{I+1}$, forms a smooth manifold in $\rmL^2(\Omega)^{I+1}$.
Its proof will be postponed to the end of this manuscript, see page~\pageref{page:proofProp}.

We now state the key result of the paper on the existence, uniqueness, 
regularity, and positivity of a constrained critical point 
$(\bfc^\ast, u^\ast) \in \FF$ of the entropy functional $\calS$ 
under the charge and energy constraints \eqref{eq:constraints}. 
A crucial ingredient to the subsequent existence proof 
is the estimate \eqref{eq:existence-bound-mfs*}, which requires at least a 
slightly superlinear growth of $\mu \mapsto H^\ast(-\mu \bfq, -\eta)$. 
The following condition \eqref{eq:Assum.psiexst} on $\psie$ states that 
the regularity of $\psiext$ is better than simply $\rmH^1(\Omega)$. For this, 
our standing assumptions \ref{hp:eps}--\ref{hp:data}
are not sufficient, but the condition is certainly true, if 
$\psiext \in \rmH^2(\Omega)$ and $\eps \in \rmC^1(\ol\Omega)$ giving 
$D\in \rmL^2(\Omega)$ and $\grmR \in \rmH^{1/2}(\Gamma_\rmR)$.

\begin{theorem}[Constrained critical points of $\calS$]
	\label{thm:existence-uniqueness}
Assume that Hypotheses \ref{hypo:entropy} and \ref{hypo:bc} hold true. 
        Moreover, there shall exist positive constants 
        $c_\ast, \pp, \qq > 0$ such that 
        \begin{align}
        \label{eq:existence-bound-mfs*}
          \ol{H^*}(\mu,\eta) = H^*(-\mu\bfq , - \eta) \geq  
           c_*\,\frac{1+ |\mu|^\qq} {\eta^\pp} 
          \quad \text{and} \quad 
          \frac\qq{1{+}\pp} > 1 
        \end{align}
	    holds for all $\mu \in \mathbb R$ and $\eta > 0$. Finally, suppose that 
        constants $C > 0$ and $\vartheta \in (0, 1)$ exist such that for all 
        $\phi \in \rmH^1(\Omega)$, we have 
        \begin{align}
          \label{eq:Assum.psiexst}
          | \bbB(\phi,\psiext) | \leq C \| \phi\|_{\rmL^1}^{1-\vartheta} \|
          \phi\|^\vartheta_{\rmH^1}. 
        \end{align}
        Then, there exists a unique critical point  
		$(\bfc^\ast, u^\ast) \in \FF$ of the total entropy $\calS$ 
        subject to \eqref{eq:constraints}. The corresponding electrostatic potential, 
		which is the solution to \eqref{eq:poisson}--\eqref{eq:poisson-bc} 
        with right-hand side $\bfq{\cdot}\bfc^\ast + D$, satisfies 
        $\Psi^\ast \in \rmC(\ol \Omega)$ and the associated temperature is the 
		positive constant $\theta^\ast = (\rmD_u S(\bfc^\ast, u^\ast))^{-1} > 0$. 
	\end{theorem}

    \begin{remark}[Weaker growth assumptions]
        \label{rem:uniqueness}
        The assertion of Theorem \ref{thm:existence-uniqueness} on the 
        uniqueness of critical points $(\bfc^\ast, u^\ast)$ also holds 
        for all $E_0 \geq 0$ and without the lower bound 
        \eqref{eq:existence-bound-mfs*} on $\ol{H^*}$. The uniqueness 
        is already a consequence of the strict convexity of $\ol{H^*}$. 

        The lower bound \eqref{eq:existence-bound-mfs*} is trivially satisfied in the 
        situation of Example \ref{ex:entropies} provided that positively \emph{and} 
        negatively charged species are present, i.e., there exist $1 \leq i,j \leq I$ 
        such that $q_i < 0 < q_j$ holds. 

        If we have, for instance, $q_i \geq 0$ for all $1 \leq i \leq I$, then 
        the reduced dual entropy from Example \ref{ex:entropies} only satisfies 
        \begin{align}
        \label{eq:existence-bound-alternative} 
            \ol{H^*}(\mu, \eta) \geq c_* \frac{1 + \max\{0, -\mu\}^\qq}{\eta^\pp}
            \quad \text{and} \quad
            \frac\qq{1{+}\pp} > 1
        \end{align}
        with positive constants $c_*, \pp, \qq > 0$. 
        Proving the existence result of Theorem \ref{thm:existence-uniqueness} under 
        the relaxed condition \eqref{eq:existence-bound-alternative} for $E_0 > \mathrm V_\geq (Q_0, \psi_\mathrm{ext})$ (cf.\ Remark \ref{rem:rmV-geq}) is beyond the 
        scope of the present article, see Subsection \ref{su:calK.coercive} for further details. 
    \end{remark}

    The first main result of this article 
    ensures that the unique constrained critical point of $\calS$ from 
    Theorem \ref{thm:existence-uniqueness} is the unique constrained
    local maximizer of $\calS$ and, thus, the unique local equilibrium of 
    \eqref{eq:general-system} in the sense of Definition \ref{def:equilibrium}. 
    We formulate this result as a corollary, which is obtained by combining Theorem \ref{thm:existence-uniqueness} and Proposition~\ref{prop:crit=equil}.

    \begin{corollary}[Existence and uniqueness of a local equilibrium]
    \label{cor:equilibrium}
        Let the hypotheses of Theorem~\ref{thm:existence-uniqueness} be in place. 
        Further suppose that $S \in \rmC^2((0,\infty)^{I+1})$ is locally strongly concave. 
        Then, there exists a unique local equilibrium of \eqref{eq:general-system} 
        with respect to fixed total energy $E_0 > \rmV(Q_0, \psiext)$ 
        and total charge $Q_0 \in \bbR$.
    \end{corollary}

Our second main result establishes the existence and uniqueness of a 
constrained \emph{global} maximizer of $\calS$ and, hence, of a \emph{global} equilibrium 
of \eqref{eq:general-system} according to Definition \ref{def:equilibrium}. 
Important ingredients are the concavity of the total entropy $\calS$, 
the quadratic nature of the nonlinear part of the energy $\calE$, 
and the fact that the derivatives 
of $\calS$ and $\calE$ w.r.t.\ $u$ are both positive, which allows us 
to turn local into global statements. 
  
\begin{theorem}[Existence and uniqueness of a global equilibrium]
\label{th:GlobEquilibrium} 
Let the hypotheses of Theorem~\ref{thm:existence-uniqueness} be in place.
Further suppose that $S \in \rmC^2((0,\infty)^{I+1})$ is locally strongly
concave.  Then, the critical point $(\bfc^*,u^*)$ obtained in
Theorem~\ref{thm:existence-uniqueness} is the unique global maximizer of $\calS$
subject to the constraints of fixed total energy $E_0 > \rmV(Q_0, \psiext)$ 
and total charge $Q_0 \in \bbR$, i.e.,
\begin{equation}
  \label{eq:GlobMaxi}
  \calS(\bfc,u) \leq \calS(\bfc^*,u^*) \quad \text{ for all }(\bfc,u)\in  
  \FF \text{ with } \calE(\bfc,u)=E_0 \text{ and }
  \calQ(\bfc, u)=Q_0. 
\end{equation}
\end{theorem}

\subsection{Global optimization by the direct method}
\label{su:DirectMethod}

Let us now summarize our results on the existence, uniqueness, and regularity of a global optimizer based on the direct method.
We will mostly work with the negative entropy $H(\bfc,u)=-S(\bfc,u)$, which is convex. All results assume Hypotheses~\ref{hp:eps}--\ref{hp:data} to be in force. As is standard in the context of the calculus of variation, we extend $H$ by $+\infty$ and consider it as a function defined on the whole space $\mathbb R^{I+1}$, cf.~\ref{hp:h.pr.cx.lsc}.

The main advantage of the direct method is that it allows us to separate existence and uniqueness results from regularity. In particular, it does not require differentiability of the entropy functional.
A technical complication that arises in the direct method stems from the fact that the physically relevant entropies are sublinear in $u$ as $u\to\infty$. Thus, the primary control of the internal energy densities $(u_k)_k$  of any optimizing sequence $(\bfc_k,u_k)_k$ comes from the energy constraint $\mathcal{E}(\bfc_k,u_k)=E_0$ leading to a bound for $(u_k)_k$ in the non-reflexive Banach space $\rmL^1(\Omega)$. It is therefore natural to first formulate the optimization problem in a larger functional setting allowing for finite measures as internal energy distributions. Let us anticipate that the monotonicity of the entropy in the internal energy component enables us to show, in a separate step, that any optimizer must be absolutely continuous with respect to the Lebesgue measure, so that concentrations can eventually be ruled out. Under additional mild differentiability assumptions on the entropy function and some extra hypotheses on the data, the optimizer is shown to be continuous and uniformly positive on $\Om$ componentwise, see Section~\ref{ssec:reg.opt}. This follows from a monotonicity argument and elliptic regularity, since at the optimizer the right-hand side $\bfq{\cdot}\bfc$ of
the Poisson equation~\eqref{eq:poisson-parts.1} can be expressed as a monotonic function of $\psi_\bfc$.

In the subsequent results, negative entropy functions $H(\bfc,u)=-S(\bfc,u)$ are admissible if they fulfil the following conditions:
\paragraph{Hypotheses.}
Let $\sfD\coleq (0,\infty)^{I+1}$.
\begin{enumerate}[label={\rm(H\arabic*)}]
	\item\label{hp:h.pr.cx.lsc}  $H:\mathbb{R}^{I+1}\to\mathbb R\cup\{\infty\}$ is proper, convex, and lower semicontinuous. 
	\item\label{hp:h.dom} $\text{dom}\,H\coleq \{H<\infty\}$ is given by $\ol\sfD=[0,\infty)^{I+1}$, 
	and $H:\text{dom}\,H\to\mathbb{R}$ is strictly convex. 
	\item\label{hp:mon.h} \textit{Monotonicity in $u$:} the map $(0,\infty)\ni u\mapsto H(\bfc,u)$ is strictly decreasing for all $\bfc\in [0,\infty)^I$. 
	\item Behavior at infinity:
	\begin{itemize}
		\item\label{hp:S.infinity} \textit{Sublinearity in $u$:} for all $(\bfc,u)\in\ol\sfD$, $\lim_{t\to\infty }t^{-1}H(\bfc,tu)=0$. 
		\item \textit{Superlinearity in $\bfc$:} there exists a continuous convex function $\gamma:\mathbb{R}^I\to\mathbb{R}_+$ such that for all $\bfc\in [0,\infty)^I\setminus\{0\}$
		\begin{align}
		\lim_{t\to\infty}t^{-1}\gamma(t\bfc)=+\infty,
		\end{align}
		a continuous function $\sigma:\mathbb{R}\to\mathbb R$ with $\lim_{t\to\infty}t^{-1}\sigma(tu)=0$ for $u>0$, and constants $K_0,K_1\ge0$ such that for all $(\bfc,u)\in \mathbb{R}^{I+1}$ 
		\begin{align}\label{eq:hp.coerc.S}
			H(\bfc,u)\ge \gamma(\bfc) - K_1\sigma(u)-K_0+\iota_{\ol\sfD}(\bfz).
		\end{align}
		Here, $\iota_{\ol\sfD}$ denotes the indicator function of the set $\ol\sfD$ satisfying $\iota_{\ol\sfD}(\bfz)=0$ if $\bfz\in\ol\sfD$ and $\iota_{\ol\sfD}(\bfz)=+\infty$ if $\bfz\not\in\ol\sfD$.
	\end{itemize}
\end{enumerate}

\begin{remark}
    The entropy functions considered in~\cite{FHKM22GEAE,Hopf_2022} satisfy these properties. For the coercivity estimate, see~\cite[eq.~(2.7)]{FHKM22GEAE} and~\cite[eq.~(6.7)]{Hopf_2022}.
\end{remark}

We denote by $M_+(\ol\Omega)$ the set of non-negative finite Radon measures $\nu$ on $\ol\Omega$ and abbreviate $U\coleq  \rmL^1_+(\Omega)^{I}\times M_+(\ol\Omega)$. 
Every $\nu\in M_+(\ol\Omega)$ can be uniquely decomposed in an absolutely continuous part $u\dd x$ and a singular part $\nu^s$ w.r.t.\ the Lebesgue measure. 
To extend the entropy functional to the set $U$, we recall that, by~\ref{hp:S.infinity}, $H$ is sublinear in $u$ as $u\to\infty$. Therefore, the natural extension $\mathcal{\wt H}$ of $\mathcal{H}\coleq -\mathcal{S}$ to $U$ is given by (cf.~\cite{DT_1984})
\begin{align}
\label{eq:negative-entropy-functional}
\mathcal{\wt H}(\bfc,\nu) \coleq  \mathcal{H}(\bfc,u),\qquad \nu=u\dd x+\nu^s.
\end{align}
The extended functional $\mathcal{\wt H}$ is clearly convex, and its restriction to $\rmL^1_+(\Omega)^{I+1}$ is strictly convex.
An appropriate weak lower semicontinuity property of $\mathcal{\wt H}$ will be shown in Subsection~\ref{ssec:ext-meas}.

We next define the extended total energy functional: for $(\bfc,\nu)\in U$, we let
\begin{align}\label{eq:energy.ext}
\mathcal{\wt E}(\bfc,\nu) \coleq \begin{cases}
	 \frac12\calB(\psi_{\bfc} {+} \psiext) + \nu(\ol\Omega)
    &\text{ if }\bfq{\cdot}\bfc\in \HH^*,
    \\+\infty&\text{ if }\bfq{\cdot}\bfc\not\in \HH^*.
\end{cases} 
\end{align}
The original energy functional $\mathcal{E}=\mathcal{E}(\bfc,u)$ 
will be understood as a map $\mathcal{E}: \rmL^1_+(\Om)^{I+1}\to\mathbb{R}\cup\{\infty\}$,  
(resp.\ as~$\mathcal{E}: \rmL^1(\Om)^{I+1}\to\mathbb{R}\cup\{\infty\}$ in Section~\ref{ssec:reg.opt}), 
where the value $\mathcal{E}(\bfc,u)$  is defined as in~\eqref{eq:constraint-energy}.
Finally, we let $\mathcal{\wt Q}(\bfc,\nu) \coleq \mathcal{Q}(\bfc,u) 
= \int_\Om \bfq{\cdot}\bfc\dd x$ for $(\bfc,\nu)\in U$, 
where as above $\nu=u \, \dn x+\nu^s$. 
\paragraph{Constrained sets.}
Given $(E_0,Q_0)\in\mathbb{R}^2$, we define
\begin{align*}
	\wt M_{E_0,Q_0} \coleq \bigset{ \bfmu\in U }{ \mathcal{\wt E}(\bfmu)=E_0, \mathcal{\wt Q}(\bfmu)=Q_0 }.
\end{align*}
We further let
\begin{align*}
	M_{E_0,Q_0} \coleq \bigset{ \bfz\in \rmL^1_+(\Omega)^{I+1} } {\mathcal{E}(\bfz)=E_0, \mathcal{Q}(\bfz)=Q_0 },
\end{align*}
which can be identified with a subset of $\wt M_{E_0,Q_0}$ via $\bfz=(\bfc,u)\mapsto (\bfc,u \, \dn x)$.

We recall the extended negative entropy functional $\mathcal{\wt H}:U\to\mathbb{R}\cup\{\infty\}$ and let $\mathcal{\wt S}\coleq -\mathcal{\wt H}$. In the following, we show existence and uniqueness of a solution to the constrained entropy maximization problem.

\begin{theorem}[Existence and uniqueness]\label{thm:ex-uniq.direct}
Let $(E_0,Q_0)\in\mathbb{R}^2$ satisfy $\wt M_{E_0,Q_0}\not=\emptyset$ and
$\sup_{\wt M_{E_0,Q_0}}\mathcal{\wt S}>-\infty$.
Then, there exists a unique solution $\bfmu_*=(\bfc_*,\nu_*)\in \wt M_{E_0,Q_0}$ of the optimization problem
  \begin{align*}
        \mathcal{\wt S}(\bfmu_*) = \sup_{\wt M_{E_0,Q_0}}\mathcal{\wt S}.
  \end{align*}
  The measure part $\nu_*$ of the optimizer is absolutely continuous with respect to the Lebesgue measure, i.e., it has the form $\nu_*=u_*\,\dn x$ for some $u_*\in \rmL^1_+(\Omega)$.
  In particular, it holds that 
\begin{align}\label{eq:strictOpt.dens}
    \mathcal{S}(\bfc_*,u_*)> \mathcal{S}(\bfc,u) 
 \quad\text{ for all }(\bfc,u)\in M_{E_0,Q_0} \setminus\{(\bfc_*,u_*)\}.
\end{align}
Furthermore, any maximizing sequence $(\bfc_k,\nu_k)_k\subseteq \wt M_{E_0,Q_0}$ converges (weakly, weakly-star) in $U$ to the optimizer $\bfmu_*$, and 
the sequence of electrostatic potentials $(\psi_{\bfc_k})_k$ converges strongly in $\HH$ to the electrostatic potential $\psi_{\bfc_*}$ of the optimizer.
\end{theorem}
	Sufficient criteria for the conditions $\wt M_{E_0,Q_0}\not=\emptyset$ and
	$\sup_{\wt M_{E_0,Q_0}}\mathcal{\wt S}>-\infty$ to hold true are provided in Lemma~\ref{l:feasible.set}.

Uniform positivity and continuity of the global optimizers can be shown by means of a  regularization of the entropy. 
In addition to hypotheses \ref{hp:h.pr.cx.lsc}--\ref{hp:S.infinity}, we now further impose the differentiability of the entropy function:
\begin{enumerate}[resume*]
    \item $H$ is continuously differentiable on $(0,\infty)^{I+1}$, i.e., $H\in \rmC^1((0,\infty)^{I+1})$.  \label{hp:h.C1}
\end{enumerate}

\begin{theorem}[Regularity]\label{thm:reg-globalopt.v1}
In addition to~\ref{hp:h.pr.cx.lsc}--\ref{hp:S.infinity}, assume hypothesis~\ref{hp:h.C1}.	Let $(E_0,Q_0)\in\mathbb{R}^2$, and 
	suppose that the functional $\calK$ given by~\eqref{eq:k-functional} satisfies the coercivity property~\eqref{eq:coerc.K}.
	Then, $M_{E_0,Q_0}\not=\emptyset$, $\sup_{\wt M_{E_0,Q_0}}\mathcal{\wt S}>-\infty$, and the 
	unique optimizer $\bfz_*\in \rmL^1_+(\Om)^{I+1}$ constructed 
	in Theorem~\ref{thm:ex-uniq.direct} satisfies $\bfz_{*}\in \rmC(\ol\Om,(0,\infty)^{I+1})$.
\end{theorem}

\subsection{Examples for suitable entropies}
\label{su:examples}

In this subsection, we provide two families of entropy density functions $(\bfc,u)\mapsto S(\bfc,u)$ which satisfy all the assumptions of our theory. 
Moreover,  we comment on examples where the concentrations are bounded 
from above by a size exclusion bound $ \sum_{i=1}^I c_i \leq w(u)$ or 
by a Fermi--Dirac (or Blakemore) type restriction $ c_i \leq w_i(u)$. 

The simplest entropy, which may not be very interesting from the point of physics, 
is given in the separated form 
\begin{align}
\label{eq:S.separated}
H(\bfc,u)=-S(\bfc,u)= H_0(u) + G(\bfc) \ \Longrightarrow \ 
 H^*(\bfy,v) = H_0^*(v) + G^*(\bfy).
\end{align}
Even simpler, one can assume $G(\bfc)= \sum_{i=1}^I G_i(c_i)$ giving 
$\bfG^*(\bfy)= \sum_{i=1}^I G_i^*(y_i)$. We emphasize that even in this 
additive setting the problem remains non-trivial because the energy constraint 
$\calE(\bfc,u)=E_0$ couples the variables nonlinearly.  

\begin{example}[Entropy functions $S$, $H^*$, and $\ol{H^*}$]
\label{ex:entropies}
We consider a special case of the entropy $S(\bfc,u)$ 
proposed in \cite{MieMit18CEER} and considered above, namely
\[
S(\bfc,u) \coleq \beta_0 w(u) - \sum_{i=1}^{I}  \Big( 
  c_i \log c_i - c_i - c_i \log \big( \beta_i (w(u) {+} w_0) \big) \Big),
\]
where $\beta_0, \beta_i, w_0 > 0$ and $w:{[0,\infty)}\to {[0,\infty)}$ is increasing 
and concave. Hence, the equilibrium concentrations are given by 
$ w_i(u) = \beta_i (w(u) + w_0)$. 
From $\bfy \coleq -\rmD_\bfc S(\bfc,u)$, we find 
\[
		c_i = (w(u) {+} w_0) \beta_i \ee^{y_i}.
\]
Using this, we derive for $v \coleq -\rmD_u S(\bfc,u)$ the relation
\[
v =- \beta_0 w'(u) - \sum_{i=1}^{I}
\frac{c_i}{w(u) {+} w_0} w'(u) = - w'(u) B(\bfy) \ \text{ with } \ 
B(\bfy) \coleq \beta_0 + \sum_{i=1}^{I} \beta_i \ee^{y_i}.
\]
Because of $\sum_{i=1}^{I} c_i = (w(u)+w_0) \big(B(\bfy){-}\beta_0\big)$, 
we also find for the temperature 
\[
\theta = \Theta(\bfc,u) = \frac{w(u)+w_0}{w'(u)\big(\beta_0 (w(u)+w_0) 
	+ \sum_{i=1}^{I} c_i\big)} = \frac{1}{w'(u) B(\bfy)}.
\]
The Legendre transform $H^* \coleq \mathfrak L(-S)$ can be calculated
explicitly by using $w^* \coleq \mathfrak L(-w)$ in the form 
\[
	H^*(\bfy,v) = B(\bfy) w^* \Big( \frac{v}{B(\bfy)} \Big) 
    + \big(  B(\bfy){-}  \beta_0 \big) w_0 
       \quad \text{for}\ \ v < 0
\]
and $H^*(\bfy,v) = \infty$ for $v \geq 0$. 
In the case $w(u) \coleq u^\alpha/\alpha$ with $\alpha \in {(0,1)}$, we find 
\[
w^*(v)= \frac{1{-}\alpha}\alpha (-v)^{-\alpha/(1-\alpha)} 
	\quad \text{and} \quad 
		H^*(\bfy,v) = \frac{1{-}\alpha}\alpha 
        \frac{B(\bfy)^{1/(1-\alpha)}}{(-v)^{\alpha/(1-\alpha)}} 
        + \big( B(\bfy) - \beta_0 \big) w_0.
\]
Defining the convex function $b:\R\to \R$ via $b(\mu) \coleq B({-}\mu \bfq)$, 
we find the convex reduced dual entropy 
\[
\ol{H^*}(\mu,\eta) = H^*(-\mu \bfq, -\eta) 
   = \frac{1{-}\alpha}\alpha \: 
   \frac{b(\mu)^{1/(1-\alpha)}}{\eta^{\alpha/(1-\alpha)}} 
        + \big( b(\mu) - \beta_0 \big) w_0.
\]
The dependence of the various entropy functions on each other 
is depicted in Figure \ref{fig:entropies}. 
\end{example}

\begin{example}[Entropies with additive free energies]
\label{ex:second-entropies}
As in \cite{AlGaHu02TDEM}, we can generate concave entropies $(\bfc,u)\mapsto
S(\bfc,u)$ by starting from a free energy $F(\bfc,\theta)$, where $\theta$ is
the temperature. The main assumption is that the function $\bfc \mapsto F(\bfc,\theta)$ 
is convex for fixed $\theta$ and that the mapping $\theta \mapsto F(\bfc, \theta)$ 
is strictly concave for fixed $\bfc$, which is the same as
saying that the heat capacity $\rmC_\text{heat}(\bfc,\theta) \coleq - 
\theta \pl_\theta^2 F(\bfc,\theta)$ is positive.  

The free entropy $-F(\bfc,\theta)/\theta$ can be
written in terms of the variable $v = -1/\theta < 0$, which is thermodynamically
conjugate to the internal energy $u$. We see that the function 
\[
\wt H(\bfc, v) \coleq - v \, F \Big( \bfc,-\frac{1}{v} \Big) \quad\text{for }v<0 
\] 
is still concave in $v$ and convex in $\bfc$. Doing a partial
Legendre transform in $\bfc$, we see that 
\[
H^*(\bfy,v) = \sup\bigset{ \bfy{\cdot}\bfc - \wt H(\bfc,v)}{ \bfc \in
  [0,\infty)^I }
\]
is a convex functional by the standard fact that the partial Legendre transform of
a convex--concave functional is jointly convex. 
Setting $H^*(\bfy,v) \coleq \infty$ for $v>0$, 
we find the corresponding entropy density function 
\[
S(\bfc,u) = -H(\bfc,u) \ \text{ with } \ H(\bfc,u)= \sup\bigset{
  \bfy{\cdot}\bfc + vu - H^*(\bfy,v)}{ (\bfy, v) \in \R^{I+1} }.
\] 
By performing suitable partial Legendre transforms, one sees that this definition is
consistent with the standard definition 
$\wt S(\bfc,\theta) \coleq -\pl_\theta F(\bfc,\theta)$. 

In electrochemistry, it is often assumed that $F$ has the form 
\begin{align}
\label{eq:additive-free-energy}
F(\bfc,\theta) = F_0(\theta) + \sum_{i=1}^I F_i(c_i,\theta),
\end{align}
see also \cite{AlGaHu02TDEM}. In that case, our main function 
$\ol{H^*}(\mu,\eta) = H^*(-\mu\bfq, -\eta)$ again has an additive form, namely 
\begin{align}
\label{eq:additive-dual-entropy}
H^*(\bfy,v) = v F_0 \Big( {-}\frac{1}{v} \Big) + \sum_{i=1}^I 
(-v) \wh F_i \Big( {-}\frac{y_i}{v}, -\frac1v \Big), \ \text{ where } \
\wh F_i(z,\theta) \coleq \mathfrak L \big[ F_i(\cdot, \theta)\big](z). 
\end{align}
In the simplest case, one assumes $F_i(c_i,\theta)=\theta G_i(c_i)$, then 
one obtains the separated case in \eqref{eq:S.separated}. For Fermi 
statistics (as in \cite{AlGaHu02TDEM}),
one has $\partial_z \wh F_i(z,\theta) = a_i(\theta) \calF_\alpha 
(b_i(\theta) z)$ where $\calF_\alpha$ is the Fermi--Dirac integral satisfying 
$   \calF_\alpha(z) \sim z^{\alpha+1}$ for $z\to \infty$. 

With this, it is easy to derive lower bounds for $\ol{H^*}(\mu,\eta) = H^*(-\mu \bfq, - \eta)$ 
from  upper bounds for $F_i$ or lower bounds for $\wh F_i$.  
\end{example}

We now comment on a situation that is not covered by our assumption, 
because we ask $H(\bfc,u) = -S(\bfc,u) \in \R$ for all $(\bfc,u) \in \ol\sfD =
[0,\infty)^{I+1}$. In several applications, one is interested in the case 
that the concentrations are bounded from above, which means that $H(\bfc,u) 
=+\infty$ for the values that are not allowed. We consider the case of 
exclusion processes where $\sum_{i=1}^I c_i$ is restricted and the case 
where each species satisfies an individual restriction. Without going into 
detail, we see that it should be possible to generalize the direct method 
to these cases without to much effort, whereas the application of the 
dual Lagrangian approach seems to be much more difficult, because the 
reduced density $\ol{H^*}$ has linear growth only, such that the coercivity 
of the convex functional $\calK$ introduced in Section \ref{eq:k-functional} is no 
longer guaranteed.  The main reason for the additional technicalities arises from 
the fact that the total charge $Q_0$ is no longer allowed to vary in all of $ \R$, 
but needs to be restricted by the bounds arising from the bounds for $\bfc$. In the 
direct method this easily translates into the condition that the admissible set
$ M_{E_0,Q_0}$ is nonempty, whereas in the Lagrangian case it is clear that
coercivity fails for large $|Q_0|$ and it needs some effort 
to establish coercivity for small $| Q_0| $.  

\begin{example}[Size exclusion processes]
\label{ex:size-exclusion-direct}
Size exclusion models (see, e.g., \cite{OttE97TDIF}) can be considered in the global approach outlined in
Subsection~\ref{su:DirectMethod} by slightly adjusting the set-up. Specifically
for Boltzmann exclusion models, one can consider the negative entropy function 
$H_{\textrm{\rm excl}}(\bfc,u)$ defined via 
\begin{align}\label{eq:h.size-excl}
   H_{\textrm{\rm excl}}(\bfc,u) \coleq -w(u)+\sum_{i=0}^I w(u)\mathfrak{b} 
   \Big( \frac{c_i}{w(u)} \Big),\quad\text{where }c_0:=w(u)-\sum_{i=1}^Ic_i,
\end{align}
for $(\bfc, u) \in \sfD \coleq \{(\boldsymbol{c},u)\in(0,\infty)^{I+1} \mid \sum_{i=1}^I 
c_i < w(u)\}$ and $H_{\textrm{\rm excl}}(\bfc,u):=+\infty$ 
for $(\bfc,u)\not\in\overline{\sfD}$.
Here, we employ $\mathfrak{b}(s) \coleq s\log s-s+1$ and $w\in C^2$ as in 
Example~\ref{ex:entropies}. 
Notice that the closed set $\ol\sfD$ 
is convex in $(\bfc,u)=(c_1,\dots,c_n,u)$, since $w$ is concave, and, using 
the monotonicity $w'\ge0$ and imposing {\em strict} concavity $w''<0$,  
a calculation similar to that in~\cite{MieMit18CEER} shows that the entropy
function $H_{\textrm{\rm excl}}$ satisfies 
the required convexity property. Therefore, the Boltzmann exclusion 
entropy satisfies a set of hypotheses analogous to
\ref{hp:h.pr.cx.lsc}--\ref{hp:S.infinity} with the adjusted form of $\sfD$. 
Thus, the global optimization by the direct method (cf.\ Theorem \ref{thm:ex-uniq.direct}) should equally apply 
to entropies of the form~\eqref{eq:h.size-excl}.

As in Example \ref{ex:entropies}, we can calculate the dual entropy 
$H_\mathrm{excl}^*$ 
by applying the Legendre transform to $H_\mathrm{excl}$. This entails 
\begin{align*}
   H_\mathrm{excl}^*(\bfy, v) = B_\mathrm{excl}(\bfy) w^* 
   \Big( \frac{v}{B_\mathrm{excl}(\bfy)} \Big) 
    \ \text{ with } \ B_\mathrm{excl}(\bfy) \coleq \beta_0 - I + \log \Big( 1 + \sum_{j=1}^I \ee^{y_j} \Big).
\end{align*}
Choosing $w(u) = u^\alpha/\alpha$ with $\alpha \in {(0,1)}$ as in 
Example \ref{ex:entropies},  we see that the convex reduced dual entropy 
	\[
		\ol{H_\mathrm{excl}^*}(\mu,\eta) = H_\mathrm{excl}^*(-\mu \bfq, -\eta) 
        = \frac{1{-}\alpha}\alpha \: \frac{b_\mathrm{excl}(\mu)^{1/(1-\alpha)}}{\eta^{\alpha/(1-\alpha)}} 
	\]
    with $b_\mathrm{excl}:\R\to \R$, $b_\mathrm{excl}(\mu) \coleq B_\mathrm{excl}({-}\mu \bfq)$ as in Example \ref{ex:entropies} 
    does \emph{not} satisfy condition \eqref{eq:existence-bound-mfs*} as we are exactly in the 
    borderline case $\frac\qq{1{+}\pp} = 1$. 
\end{example}

\begin{example}[Individually bounded concentrations]
\label{ex:IndivUpperBd}
    Following the setup in Example \ref{ex:second-entropies}, we end up in a similar situation. 
    If we start from a free energy as in \eqref{eq:additive-free-energy}, then each 
    individual free energy $F_i( \cdot, \theta)$ is defined on a bounded domain for 
    any $\theta > 0$. As a consequence, $\wh F_i(z ,\theta)$ grows linearly for $z \rightarrow \infty$ 
    and fixed $\theta > 0$, which results in a linear growth of $\ol{H^*}(\mu,\eta)$ for 
    $|\mu| \rightarrow \infty$ and fixed $\eta > 0$ according to \eqref{eq:additive-dual-entropy}. 
    Thus, the Lagrangian method is \emph{not} applicable in this case, while we expect that the direct method is still feasible. 

To be more specific, we follow Example \ref{ex:second-entropies} and consider additive free 
energies as in \eqref{eq:additive-free-energy} with $F_i(c_i,\theta) = 
\theta^a \mathfrak f_i( c_i)$ for some $a \geq 1$ and convex functions 
$\mathfrak f_i$ with $\mathfrak f_i(\zeta) \coleq +\infty$ for $\zeta \not\in 
[0,\varsigma_i]$. E.g., the choice $\mathfrak f_i(\zeta) \coleq \mathfrak b(\zeta)
+ \varsigma_i \mathfrak b(1{-}\zeta/\varsigma_i)$ with $\mathfrak b$ as in 
Example \ref{ex:size-exclusion-direct} leads to a family of 
Blakemore statistics, which contains the Fermi--Dirac statistics with 
$\varsigma_i=1$, see \cite{FaKoFu17CACF}. 
Under the assumption that $F_0$ is strictly concave, we can proceed as in Example 
\ref{ex:second-entropies} and obtain the convex functions $H(\bfc,u)=-S(\bfc,u)$ as 
well as 
\[
H^* (\bfy,\eta) = -\eta F_0 \Big( \frac{1}{\eta} \Big) + \sum_{i=1}^I \eta^{a-1}  \, 
\mathfrak f^*_i\big({-} \eta^{a-1} y_i \big).
\]
The case $ a = 1$ is the separated case, and $a=2$ leads to the classical Fermi--Dirac 
statistics with $c_i =(\mathfrak f^*_i)' ( - y_i/\theta)$. The dual functions $\mathfrak f_i^*$ 
have only linear growth, hence $\ol{H^*}$ cannot grow faster in $\mu$ than linear. 

Hence, the dual 
Lagrangian approach is obstructed again; however, as in the previous example, the direct method might still be applicable. 
\end{example}

\section{Lagrangian analysis for critical points}
\label{se:proofs-LagrangianApproach}
	
In this section, we settle the existence, uniqueness, and regularity 
	of critical points and local/global maximizers of the total entropy $\calS$ 
    under the constraints \eqref{eq:constraints} as stated in Theorem 
    \ref{thm:existence-uniqueness}, Corollary \ref{cor:equilibrium}, and 
    Theorem \ref{th:GlobEquilibrium}. 
	We first introduce an appropriate Lagrange functional $\calL$ related to 
	the constrained maximization of $\calS$. 
	The critical points of $\calL$ can be characterized as minimizers 
    of a convex functional $\calK$. We therefore arrive at a convex minimization 
	problem, which admits a unique solution (cf.\ Proposition \ref{prop:k-minimizer}).

\subsection{Lagrangian multiplier equations for critical points}
\label{su:Lagrange}
	
	According to Definition \ref{def:equilibrium}, we first look for a local 
    maximizer of the entropy functional $\calS(\bfc, u)$ in $\FF$ 
	under the constraints of charge and energy conservation. The standard approach 
	for solving this kind of constrained optimization problems is the Lagrangian 
	method (see, e.g., \cite{Trol10OCPD}), which, roughly speaking, relies on the 
	following principle: at the position of any constrained local extremum of the 
	objective functional, its derivative is a linear combination of the derivatives 
	of the constraint functions. The Lagrange function is constructed in such a way 
	that the constrained local extrema coincide with its critical points. 
The constrained maximization problem is the following:
\[
\text{maximize } \ \calS(\bfc,u) \quad \text{subject to} \ \ 
\frac12 \calB(\psi_{\bfc} {+} \psiext) + \int_\Omega u \dd x = E_0 
\ \ \text{and} \ \int_\Omega \bfq {\cdot} \bfc \dd x = Q_0,  
\]
where the functional $\calB$ is defined in \eqref{eq:def-b-func}, and where 
the electrostatic potential $\psi_{\bfc}$ is the solution to \eqref{eq:poisson-parts.1}. 
All together, we have the constraints 
	\begin{subequations}
		\label{eq:lagrange-constraints}
		\begin{align}
			\mathfrak E(u, \psi) &\coleq \frac12\calB(\psi {+} \psiext) + \int_\Omega u \dd x = E_0, 
			\label{eq:lagrange-constraint-energy} \\
			\mathcal Q(\bfc) &\coleq \int_\Omega \bfq {\cdot} \bfc \, \dd x = Q_0, 
			\label{eq:lagrange-constraint-charge} \\ 
			\varPi_\Omega(\bfc, \psi) &\coleq\vphantom{\int_\Omega}\DIV(\eps\nabla \psi) + \bfq{\cdot}\bfc = 0 
			\quad \text{in} \ \Omega, \label{eq:lagrange-constraint-poisson-volume} \\ 
			\varPi_{\Gamma_\rmR}(\bfc, \psi) &\coleq\vphantom{\int_\Omega} 
			\varepsilon \nu {\cdot} \nabla \psi + \omega \psi = 0 
			\quad \text{on} \ \Gamma_\rmR. \label{eq:lagrange-constraint-poisson-surface}
		\end{align}
	\end{subequations}

	Note that we have expressed the charge and energy functionals from 
	\eqref{eq:constraints} in terms of a different set of variables, and that 
	Poisson's equation in $\Omega$ and the boundary condition on $\Gamma_\rmR$ 
	are included as separate constraints on the independent variables $\bfc$ and $\psi$. 
	The four constraints on $(\bfc, u) \in \FF$ and $\psi \in 
	\HH$ in \eqref{eq:lagrange-constraints} give rise to 
    Lagrange multipliers $\eta \in \bbR$ for \eqref{eq:lagrange-constraint-energy}, 
	$\kappa \in \R$ for \eqref{eq:lagrange-constraint-charge}, and 
	$\lambda \in \HH$ for \eqref{eq:lagrange-constraint-poisson-volume} and \eqref{eq:lagrange-constraint-poisson-surface}. 
	The Lagrange functional $\calL : \FF_\calL \rightarrow \bbR$ 
	is now defined as 
	\begin{subequations}
	\begin{align}
		\calL(\bfc,u,\psi,\eta,\kappa,\lambda) &\coleq \calS(\bfc,u) 
		+ \big( E_0 - \mathfrak E(u, \psi) \big) \eta 
		+ \big( Q_0 - \mathcal Q(\bfc) \big) \kappa \nonumber \\ 
		&\qquad + \int_\Omega \big( \varepsilon \nabla \psi \cdot \nabla \lambda 
		- \bfq {\cdot} \bfc \, \lambda \big) \dd x 
		+ \int_{\Gamma_\rmR} \omega \psi \lambda \dd a \label{eq:lagrange-functional}
	\\ 
		\text{with }\label{eq:l-domain}
		\FF_\calL &\coleq \FF \times \HH 
		\times (0,\infty) \times \R \times \HH. 
	\end{align}    
	\end{subequations}
 
	The derivatives with repect to the Lagrange multipliers $\eta$ and $\kappa$ give rise to 
	the energy and charge constraints, respectively, whereas $\rmD_\lambda \calL = 0$ 
	leads to both Poisson's equation and the related Robin boundary condition: 
	\begin{subequations}
		\label{eq:lagrange-eta-kappa-lambda}
		\begin{align}
			\rmD_\eta \calL = 0 \qquad &\Leftrightarrow \qquad \mathfrak E(u, \psi) = E_0, \label{eq:lagrange-eta} \\ 
			\rmD_\kappa \calL = 0 \qquad &\Leftrightarrow \qquad \mathcal Q(\bfc) = Q_0, \label{eq:lagrange-kappa} \\ 
			\rmD_\lambda \calL = 0 \qquad &\Leftrightarrow \qquad 
			\begin{cases}
				\varPi_\Omega(\bfc, \psi) &= 0, \\ 
				\varPi_{\Gamma_\rmR}(\bfc, \psi)\mquad &= 0. 
			\end{cases} 
			\label{eq:lagrange-lambda}
		\end{align}
	\end{subequations}	
	Finally, the derivative of $\calL$ w.r.t.\ $(\bfc, u)$ leads to the representation of 
	$\rmD S(\bfc, u)$ as a linear combination of the derivatives of the constraints, 
	while an elementary calculation shows the validity of the second equivalence: 
	\begin{align}
		\rmD_{(\bfc,u)} \calL = 0 \quad &\Leftrightarrow \quad \rmD S(\bfc, u) = 
		\begin{pmatrix}
			(\kappa {+} \lambda) \bfq \\ 
			\eta
		\end{pmatrix}, \label{eq:lagrange-c-u} \\ 
		\rmD_\psi \calL = 0 \quad &\Leftrightarrow \quad 
        \lambda = \eta (\psi + \psiext). \label{eq:lagrange-psi}
	\end{align}
    The equations represented by $\rmD \calL = 0$ are called the Lagrange equations 
    for maximizing the total entropy $\calS(\bfc, u)$ among 
    $(\bfc, u, \psi) \in \FF \times \HH$ 
    satisfying the charge and energy constraints as well as Poisson's equation. 
	We see that the \emph{positive constant} $\eta$ plays the role of the inverse of a 
	positive constant temperature, i.e., $\eta = 1/\theta>0$ 
	because $1/\theta = \rmD_u S(\bfc,u)>0$ due to Hypotheses \ref{hypo:entropy}.

\subsection{Nonrigorous derivation of \texorpdfstring{$\calK$}{K}}
\label{su:Derive.whK}

We proceed with the formal derivation of an auxiliary functional $\calK$, 
which will be shown to be coercive in Subsection \ref{su:calK.coercive}. 
Exploiting the coercivity and the strict convexity of $\calK$, 
we will subsequently be able to prove 
the existence of a unique minimizer of $\calK$. 

We rewrite the Lagrange functional from \eqref{eq:lagrange-functional} as 
\begin{align*}
\calL(\bfc,u,\psi,\eta,\kappa,\lambda) &\coleq  \calS(\bfc,u) 
+ \eta \Big( E_0 - \int_\Omega u \dd x - \frac12\calB(\psi{+}\psiext)\Big)
+ \kappa\Big(Q_0 - \int_\Omega \bfq {\cdot} \bfc \dd x  \Big) \\
&\quad + \bbB(\psi,\lambda) - \int_\Omega \bfq {\cdot} \bfc \, \lambda \dd x.  
\end{align*}
Recall that $\eta$ is the Lagrange parameter for the energy constraint, $\kappa$ for
the charge constraint, and $\lambda \in \HH$ 
for guaranteeing the constraint $\psi = \psi_{\bfc}$. 

It is easy to see that $\calL$ is concave in $(\bfc,u,\psi)$, while it is
affine, and hence convex, in $(\eta,\kappa,\lambda)$. Taking the infimum over
the latter variables, we obtain the value $-\infty$ if one of the constraints
is not satisfied. If the constraints are satisfied, then we can maximize over
$(\bfc,u,\psi)$ and obtain exactly the desired maximal value for $\calS$,
namely
\begin{align*}
&\sup \Big\{ \calS(\bfc,u) \ \Big| \ \int_\Omega u \dd x 
+ \frac12\calB(\psi_{\bfc}{+}\psiext) 
= E_0, \ \int_\Omega \bfq {\cdot} \bfc \dd x = Q_0 \Big\} \\ 
&\quad = \sup_{(\bfc,u,\psi)} \Big( \inf_{ (\eta,\kappa,\lambda)} 
\calL(\bfc,u,\psi,\eta,\kappa,\lambda) \Big).
\end{align*} 
In qualified cases (see, e.g., \cite[Chap.~VI]{EkeTem76CAVP}), one can
interchange the order of the supremum and the infimum. Here, we do this without
justification, because we will check a posteriori that the minimizers
$(\eta,\kappa,\lambda)$ constructed in Subsection \ref{su:exist-minimizers} lead to the
desired solution of the original equilibrium problem (cf.\ Corollary \ref{cor:equilibrium} 
and Theorem \ref{th:GlobEquilibrium}). 
Thus, we are led to look at 
\[
\inf_{ (\eta,\kappa,\lambda)} \calK(\eta,\kappa,\lambda) \quad \text{with} \quad 
\calK(\eta,\kappa,\lambda) \coleq  \sup_{(\bfc,u,\psi)} \calL(\bfc,u,\psi,\eta,\kappa,\lambda). 
\]
The point is that $\calK$ can be calculated  explicitly by using the Legendre
transform $H^*$ of the convex function $ (\bfc,u)\mapsto H(\bfc,u)=-
S(\bfc,u)$. Indeed, the supremum over $(\bfc,u)$ features $H^*$, and completely
independently, one can maximize the quadratic functional in 
$\psi \in \HH$. This
leads to the functional $\calK : \FF_\calK \rightarrow \bbR\cup\{\infty\}$,
\begin{subequations}
\begin{align}
\label{eq:k-functional}
\calK(\eta, \kappa, \lambda) &\coleq \int_\Omega H^*\big({-}(\kappa{+}\lambda)\bfq, -
\eta\big) \dd x + \kappa \,Q_0 + \eta\,E_0 - \bbB(\lambda,\psiext) +
\frac1{2\eta} \calB(\lambda) \\ 
\text{with }\label{eq:k-domain} 
		\FF_\calK &\coleq \vphantom{\int_\Omega} 
		(0,\infty) \times \R \times \HH. 
\end{align}
\end{subequations}

By showing 
that $\calK$ admits a unique minimizer $(\eta^\ast, \kappa^\ast, \lambda^\ast)$ 
with continuous $\lambda^\ast$ (cf.\ Proposition \ref{prop:k-minimizer}), 
and by establishing a bijection between critical points of $\calK$ and $\calL$ 
(cf.\ Proposition \ref{prop:CharEquil}), 
we arrive at a unique solution of the Lagrange equations 
\eqref{eq:lagrange-eta-kappa-lambda}--\eqref{eq:lagrange-psi}. 
Observe that $\calK$ is convex, 
which follows from the convexity of the dual entropy 
$H^*$ and the convexity of $(\xi_1,\xi_2) \mapsto |\xi_1|^2/\xi_2$ 
as a function from $\mathbb R^m \times (0, \infty)$ into $\mathbb R$ 
for any $m \in \N$.

\subsection{Coercivity of \texorpdfstring{$\calK$}{K}}
\label{su:calK.coercive}

Below, we will show that $\calK$ is strictly convex and even coercive 
on $\FF_\calK$ for $E_0 > \mathrm V(Q_0,\psi_\mathrm{ext})$; hence, 
there is a unique critical point, which is the global minimizer of $\calK$. 
To simplify the notation, we introduce (resp.\ recall) the convex 
and non-negative functions 
\[
\mfHS(\mu,\eta) \coleq \int_\Omega \ol{H^*}(\mu(x),\eta(x))\dd x \qquad \text{and} \qquad \ol{H^*}(\mu,\eta) \coleq H^*(-\mu\bfq,-\eta). 
\]

The following elementary lemma will be useful for showing coercivity of $\calK$. 
It generalizes the often used estimate $|\mu|^2/\eta\geq 2\eps |\mu| - \eps^2 \eta$. 

\begin{lemma}
\label{le:qq.pp.estim} Fix positive  $\pp$ and $\qq$, then
for each $\delta>0$, there exists some $c_\delta>0$ such that 
\[
 \frac{|\mu|^\qq}{\eta^\pp} \geq c_\delta |\mu|^{\qq/(1+\pp)} - \delta \eta \quad
 \text{for all } (\mu,\eta) \in \R\ti \R_{>0} \:. 
\]
\end{lemma}
\begin{proof} We apply Young's inequality in the form 
\[
|\mu|^{\qq/(1+\pp)} = \Big(a\tfrac{\ds |\mu|^\qq}{\ds \eta^\pp}
\Big)^{1/(1+\pp)} \, \Big( \tfrac{\ds\eta}{\ds a^{1/\pp}}\Big) ^{p/(1+\pp)}
\leq \frac a{1{+}\pp}\,\tfrac{\ds |\mu|^\qq}{\ds \eta^\pp} + \frac\pp{1{+}\pp}
\tfrac{\ds\eta}{\ds a^{1/\pp}}. 
\]
Thus, adjusting $a$ according to $\delta$, we find $c_\delta= (1{+}\pp)
\big(\delta/\pp\big)^{\pp/(1+\pp)}>0$. 
\end{proof}

Combining \eqref{eq:existence-bound-mfs*} and Lemma \ref{le:qq.pp.estim} will provide a
bound for the superlinear term $|\mu|^{\qq/\pp}$ where $\mu=\lambda{+}\kappa$. 

We are now ready to state the following coercivity result, where we use the
bound $\rmV(Q_0,\psiext)\geq 0$ for $\rmE(\bfq{\cdot}\bfc,\psiext)$ derived in Proposition
\ref{pr:MinElectroEner}. Note that $\rmV(Q_0,\psiext)=0$ in the case 
$\calH^{d-1}(\Gamma_\rmD)>0$ as well as in the pure Neumann case, 
however, for the pure Robin case, $\rmV(Q_0,\psiext)$ is nontrivial. 
We conjecture that $\calK$ is coercive even 
in the case of purely non-negative charge carriers $q_i \geq 0$, $i = 1, \dotsc, I$, 
where the reduced dual entropy $\ol{H^*}$ 
only fulfills the weaker hypothesis \eqref{eq:existence-bound-alternative}. 
In this situation, the more stringent 
lower bounds on $E_0 > \rmV_\geq(Q_0, \psiext)$ mentioned 
in Remarks \ref{rem:rmV-geq} and \ref{rem:uniqueness} and $Q_0 \geq 0$ should still guarantee the coercivity of $\calK$.

\begin{proposition}[Coercivity of $\calK$] 
\label{prop:Coerc.calK}
Assume that \eqref{eq:existence-bound-mfs*} and
\eqref{eq:Assum.psiexst} hold.  
Then, for all $(Q_0,E_0)$ satisfying $E_0> \rmV(Q_0,\psiext)$ (cf.\ Proposition 
\ref{pr:MinElectroEner}), the functional $\calK$
is coercive, namely, there exist positive constants $c_1$ and $C_1$ such that 
\begin{equation}
\label{eq:coerc.K}
\calK(\eta,\kappa,\lambda) \geq c_1 \Big( \| \lambda\|_{\rmH^1} + \frac1{\eta^\pp} + \eta
+ |\kappa| \Big) -C_1 \quad \text{for all } (\eta,\kappa,\lambda) \in \R_{>0}\ti
\R \ti \HH. 
\end{equation}
\end{proposition} 
\begin{proof}
We define the shifted functional $\calN$ via 
$\calN(\eta,\kappa,\mu)\coleq \calK(\eta,\kappa, \mu{-}\kappa)$. Of course, $\calK$ is
coercive if and only if $\calN$ is coercive. We find the explicit expression
\begin{align*}
\calN(\eta,\kappa,\mu) &= \mfHS(\mu,\eta) - \bbB(\mu,\psiext) + \eta\, E_0 +
\frac1{2\eta} \calB(\mu)
\\
&\quad + \kappa\Big( Q_0 + \bbB(1,\psiext) - \frac1\eta \bbB(1,\mu) \Big) +
\frac{\kappa^2}{2\eta} \calB(1) 
\\
&=\calN_2(\eta,\mu) + \frac{\calB(1)}{2\eta} \big(\kappa
{-}\kappa_*(\mu,\eta)\big)^2 \quad \text{with} 
\\
\kappa_*(\mu,\eta)& =\frac1{\calB(1)} \Big( \bbB(1,\mu)- \eta\big(
Q_0+\bbB(1,\psiext)\big)\Big) \quad \text{and} 
\\
\calN_2(\eta,\mu) &= \mfHS(\mu,\eta) - \bbB(\mu,\psiext) + \eta E_0 +
\frac1{2\eta} \calB(\mu) - \frac{\calB(1)}{2\eta}\, \kappa_*(\mu,\eta)^2. 
\end{align*}
We note that $\calB(1)=\int_{\Gamma_\rmR} \omega \dd a >0$ by assumption and
that $\bbB(1,\phi) = \int_{\Gamma_\rmR} \omega \phi \dd a$. 
Reorganizing the terms in $\calN_2$ associated with $\eta^1$ and using the
definition of $\rmV(Q_0,\psiext)$ in Proposition \ref{pr:MinElectroEner}, we have 
\begin{align}
\calN_2(\eta,\mu) &= \mfHS(\mu,\eta) + 
\big( Q_0+\bbB(1,\psiext) \big) \bbB(1,\mu) - \bbB(\mu,\psiext)
\nonumber \\
&\quad + \eta\big( E_0 - \rmV(Q_0,\psiext)\big) +
\frac1{2\eta}\Big( \calB(\mu) - \frac1{\calB(1)} \bbB(1,\mu)^2\Big). 
\label{eq:n2-identity}
\end{align}

We are now ready to show coercivity. First, we observe that  
$\big( Q_0+\bbB(1,\psiext) \big) \bbB(1,\mu) \leq C \big| \int_{\Gamma_\rmR} \mu \dd a \big|$, 
where subsequently $C > 0$ and $c > 0$ denote sufficiently large and small constants, respectively. 
For $\rma >1$, we can use that the trace mapping from $\rmW^{1/\rma,\rma}(\Omega) \hookrightarrow \rmL^\rma(\pl\Omega) \hookrightarrow \rmL^1(\Gamma_\rmR)$ is continuous. Together with the Gagliardo--Nirenberg estimate, we derive 
\begin{align}
	\label{eq:mu-first-estimate}
	\| \mu \|_{\rmL^1(\Gamma_\rmR)}
    &\leq C \| \mu\|_{\rmL^\rma(\pl\Omega)} \leq C \| \mu\|_{\rmW^{1/\rma, \rma}(\Omega)} \leq C \|\mu\|_{\rmW^{1,\rma}}^{1/\rma} \| \mu\|_{\rmL^\rma}^{1-1/\rma}  \nonumber \\ 
	&\leq 
    \eps \|\mu\|_{\rmW^{1,\rma}} + C_\eps \| \mu\|_{\rmL^\rma}
    \leq  \eps \|\nabla \mu\|_{\rmL^\rma} +  C_\eps \| \mu\|_{\rmL^\rma},
\end{align}
where $\eps>0$ can be chosen arbitrarily, if $C_\eps$ is adjusted accordingly. Choosing $\rma \coleq \min\{ \qq/(1{+}\pp), 2\}$, 
equation \eqref{eq:mu-first-estimate} immediately entails 
\begin{align}
	\Big| \int_{\Gamma_\rmR} \mu \dd a \Big| 
    &\leq \eps \| \nabla \mu \|_{\rmL^2} 
    + C_\eps \|\mu\|_{\rmL^{\qq/(1+\pp)}}. 
	\label{eq:mu-estimate}
\end{align}
By assumption, we have 
\[
\Xi \coleq \frac{E_0 - \rmV(Q_0,\psiext)}{5} > 0.
\]
To derive a lower bound for \eqref{eq:n2-identity}, 
we can combine assumption \eqref{eq:existence-bound-mfs*} and Lemma
\ref{le:qq.pp.estim} for $\mfHS$, assumption \eqref{eq:Assum.psiexst} 
for the third term in \eqref{eq:n2-identity}, 
and estimate \eqref{eq:mu-estimate} to obtain
\begin{align*}
\calN_2(\eta,\mu) &\geq c_*|\Omega| \eta^{-\pp} + c_*C_\Xi \|
\mu\|_{\rmL^{\qq/(1+\pp)}}^{\qq/(1+\pp)} - \Xi \eta 
- C \| \mu\|_{\rmL^1}^{1-\vartheta} \| \mu\|_{\rmH^1}^\vartheta \\ 
&\quad + 5 \Xi \eta + \frac{\underline \eps}{2\eta} \|\nabla \mu\|_{\rmL^2}^2 
- c \| \nabla \mu \|_{\rmL^2} 
- C \|\mu\|_{\rmL^{\qq/(1+\pp)}},
\end{align*}
where we used $\int_{\Gamma_\rmR} \omega \dd a\int_{\Gamma_\rmR}
\omega \mu^2 \dd a \geq \big( \int_{\Gamma_\rmR}
\omega \mu \dd a  \big)^2$ for the last term in \eqref{eq:n2-identity}. 
Invoking Young's inequality  
to estimate $\| \mu\|_{\rmL^1}^{1-\vartheta} \| \mu\|_{\rmH^1}^\vartheta 
\lesssim \Xi \| \mu\|_{\rmH^1} + C_\Xi\|\mu\|_{\rmL^1}$ 
and applying Lemma \ref{le:qq.pp.estim}
(now with $\tilde\qq=2$ and $\tilde\pp=1$) to 
$\frac{\underline \eps}{2\eta} \|\nabla \mu\|_{\rmL^2}^2$ gives 
\begin{align*}
\calN_2(\eta,\mu) &\geq c_*|\Omega| \eta^{-\pp} + c_*C_\Delta \|
\mu\|_{\rmL^{\qq/(1+\pp)}}^{\qq/(1+\pp)} - C \| \mu\|_{\rmL^{\qq/(1+\pp)}} +
3\Xi \eta + c_3 \| \mu\|_{\rmH^1}.
\end{align*}
The crucial ingredient to the previous estimate is a bound of the form 
\begin{align}
\label{eq:n2-poincare}
\| \mu \|_{\rmH^1} \leq C \big( \| \nabla \mu \|_{\rmL^2} + \| \mu \|_{\rmL^1} \big). 
\end{align}
Such a bound follows, e.g., from the Gagliardo--Nirenberg inequality. 
Moreover, we can estimate 
\begin{align*}
\frac{(\kappa{-}\kappa_*)^2} {2\eta} 
&\geq c_4 \big|\kappa{-}\kappa_*(\mu,\eta) \big| -\Xi \eta \\ 
&\geq c_4 |\kappa| -c_4 C\big(\|\mu\|_{\rmH^1} +\eta\big) - \Xi \eta 
\geq c_4|\kappa| - \frac{c_3}2 \|\mu\|_{\rmH^1} -2\Xi \eta,
\end{align*}
if $c_4 $ is chosen sufficiently small. 
Adding this to the lower bound for
$\calN_2$ gives a lower bound for $\calN$. The condition $\qq>1{+}\pp$ gives a
superlinear coercive term that compensates the negative linear term in
$\|\mu\|_{\rmL^{\qq/(1+\pp)}}$, and thus the coercivity estimate
\eqref{eq:coerc.K} is established by recalling $\mu=\lambda{+}\kappa$. 
\end{proof}

\subsection{Existence and uniqueness of critical points of \texorpdfstring{$\calL$}{L}}
\label{su:exist-minimizers}
 
Having established the coercivity of $\calK$ on $\FF_\calK$, 
we are in a position to study the minimization problem for $\calK$. 
 
	\begin{proposition}[Existence, uniqueness, and regularity of minimizers of $\calK$]
		\label{prop:k-minimizer}$ $\\
		Suppose that Hypotheses \ref{hypo:entropy} 
        and the assumptions of Proposition \ref{prop:Coerc.calK} are in place. 
		Then, there exists a unique minimizer $(\eta^\ast, 
		\kappa^\ast,\lambda^\ast)$ of $\mathcal K$ 
		on $\FF_\calK$. It has the regularity $\lambda^\ast\in \rmC(\ol\Omega)$. 
	\end{proposition}
	\begin{proof}
First, we note that \[\calK(\eta,\kappa,\lambda)=\int_\Omega 
    \ol H^*\big(\kappa{+}\lambda, 
\eta\big) \dd x + \kappa \,Q_0 + \eta\,E_0 - \bbB(\lambda,\psiext) +
\frac1{2\eta} \calB(\lambda)\] 
is the sum of convex functionals and therefore convex itself.
We show the asserted properties in three steps.
The uniqueness of minimizers immediately follows from the strict convexity of $\calK$, 
which we establish in the first step. 

Step~1: \textit{Strict convexity of $\calK$.}\,
We assert that $\calK$ is strictly convex on its domain $\FF_\calK\setminus\{\calK=\infty\}$. Let $z_i=(\eta_i,\kappa_i,\lambda_i)\in \FF_\calK\setminus\{\calK=\infty\}$, $i=1,2$, with $z_1\not=z_2$, and let $\tau\in(0,1)$. We need to show the strict inequality 
\begin{align}\label{eq:strictConv}
    \calK(\tau z_1+(1{-}\tau)z_2)<\tau\calK(z_1)+(1{-}\tau)\calK(z_2).
\end{align}
If $\eta_1\not=\eta_2$, inequality~\eqref{eq:strictConv} follows from the strict joint convexity of $\ol{H^*}$. If $\eta_1=\eta_2\eqcol\eta$ and $\lambda_1\not=\lambda_2$, 
we deduce~\eqref{eq:strictConv} from the strict convexity of 
\[
\HH\ni\lambda\mapsto 
\frac{1}{2\eta} \calB(\lambda) = \int_{\Omega} \frac{\varepsilon}{2\eta}|\nabla \lambda|^2 \dd x
+ \int_{\Gamma_\rmR} \frac{\omega}{2\eta} \lambda^2 \dd a.
\]  
If $\calH^{d-1}(\Gamma_\rmD) > 0$ or if $\Gamma_\rmD = \emptyset$ and 
$\omega \equiv 0$, the strict convexity of this map follows from the properties of 
the space $\HH$. 
In the case $\Gamma_\rmD = \emptyset$ and $\omega \not\equiv 0$, 
strict convexity is ensured by the quadratic boundary integral. 
  
In the remaining case that $\eta_1=\eta_2$, $\lambda_1=\lambda_2$, $\kappa_1\not=\kappa_2$,
inequality~\eqref{eq:strictConv} follows by invoking the strict convexity of $\ol{H^*}$ with respect to its first argument.

Step~2: \textit{Existence.}\, The convex functional $\mathcal K:\FF_\calK\to\mathbb R\cup\{\infty\}$ is easily seen 
to be proper, bounded below, and lower semicontinuous. Thus, the coercivity property in
Proposition \ref{prop:Coerc.calK} guarantees that minimizing sequences 
$\{(\eta_j, \kappa_j, \lambda_j)\}_j\subset\FF_\calK$ for $\calK$ have subsequences that weakly converge in $\FF_\calK$ to a minimizer $(\eta^*, \kappa^*, \lambda^*)$ of $\calK$.

Step~3: \textit{Regularity.}\,
We first present the formal argument on how to deduce the continuity of 
$\lambda^*$, where $(\eta^*, \kappa^*, \lambda^*)\in \FF_\calK$ 
denotes the unique minimizer of $\calK$.
From the minimizing property, we formally deduce that 
$\rmD_\lambda\calK(\eta^*, \kappa^*, \lambda^*)=0$, meaning that the overall potential 
$\Psi^\ast \coleq \lambda^*/\eta^* \in \HH$ satisfies the elliptic equation
		\begin{gather}
        \label{eq:regularity-poisson}
			-\DIV(\eps\nabla \Psi^\ast) 
   +\rmD_1\ol{H^*} \big( \eta^\ast \Psi^\ast + \kappa^\ast, 
			\eta^\ast \big) 
   = D \quad \text{in}\ \Omega, \\ 
			\Psi^\ast = 0 \quad \text{on}\ \Gamma_\rmD, \qquad 
			\eps \nu \cdot \nabla \Psi^\ast + \omega \Psi^\ast = \grmR 
			\quad \text{on}\ \Gamma_\rmR. \nonumber
		\end{gather}
The idea is now to use monotonicity and elliptic regularity in order to
infer the boundedness and continuity of $\Psi^\ast$ and thus of $\lambda^*$. 
To this end, we aim to apply \cite[Thm.~4.8]{Trol10OCPD} which relies on our standing assumption \ref{hp:data}.  This theory guarantees the existence of a unique weak solution 
$\Phi \in \rmH^1(\Omega) \cap \rmC(\ol \Omega)$ to 
\begin{align}
 \label{eq:regularity-troeltzsch}
			-\DIV(\eps \nabla \Phi) + f(\Phi) &= D \quad \text{in}\ \Omega, \\ 
			\eps \nu \cdot \nabla \Phi + \omega \Phi &= \grmR \quad \text{on}\ \partial \Omega 
            \nonumber
\end{align}
provided $f : \mathbb R \rightarrow \mathbb R$ is continuous and monotonically non-decreasing. 
We note that the statement of \cite[Thm.~4.8]{Trol10OCPD} 
		(and also all preceding ones used in the proof therein) carry over one-to-one 
		to our situation with mixed boundary data. In particular, 
        the necessary coercivity of the bilinear form 
        $a : \HH \times \HH \rightarrow \bbR$, 
        \begin{align*}
            a[\Phi, \Psi] \coleq \int_\Omega \nabla \Phi \cdot \nabla \Psi \dd x 
            + \int_{\Gamma_\rmD} \omega \Phi \Psi \dd a
        \end{align*}
        follows in our situation directly from the properties of the function space 
        $\HH$ without an additional source term 
        in $\Omega$; see also Lemma \ref{lemma:poisson.gen} and estimate \eqref{eq:Lcoerc}. 
		Equation \eqref{eq:regularity-poisson} obviously fits to the setting 
        of equation \eqref{eq:regularity-troeltzsch} by letting 
        $f : \mathbb R \rightarrow \mathbb R$, 
		\begin{align*}
			f(\Phi) \coleq \rmD_1\ol{H^*} \big( \eta^\ast \Phi + \kappa^\ast, 
			\eta^\ast \big)
		\end{align*}
and noting that the monotone increase of $f$ is a consequence of the convexity of $\ol{H^*}$. 

In the above reasoning, we tacitly assumed the minimizer $z^*\coleq (\eta^*, \kappa^*, \lambda^*)$ 
to be sufficiently regular to guarantee the differentiability of $\calK$ at $z^*$ 
and the applicability of \cite[Thm.~4.8]{Trol10OCPD}.
To make our reasoning rigorous, we first apply the elliptic regularity argument to 
minimizers of an equicoercive sequence of approximate functionals $\calK_\delta$ that 
are Fr\'echet differentiable. The sequence $\{\calK_\delta\}_\delta$ 
is constructed in such a way that it converges to $\calK$ in the sense of 
$\Gamma$-convergence in $\FF_\calK$ with respect to the weak topology.  
Consequently, by classical $\Gamma$-convergence theory (see, e.g., 
\cite[Thm.~1.21]{Brai02GCB}), the sequence of minimizers 
$z_\delta=(\eta_\delta,\kappa_\delta,\lambda_\delta)$ of $\calK_\delta$ 
converges to the unique minimizer $z^*$ of $\calK$, weakly in $\FF_\calK$, so 
that any $\delta$-independent uniform bound for $\lambda_\delta$ is inherited by $\lambda^*$.
A convenient regularization can be obtained by inf-convolution. Let 
\[
\ol{H^*_\delta}(\mu,\eta) \coleq \inf_{\nu\in\mathbb R}
\Big(\ol{H^*}(\mu{-}\nu,\eta)+\frac{1}{2\delta} |\nu|^2\Big),
\quad\delta\in(0,1].
\]
The definition implies that $\ol{H^*_\delta}$ is convex. Furthermore,
we assert that the lower bound \eqref{eq:existence-bound-mfs*} implies 
$\ol{H^*_\delta}(\mu,\eta)\ge \tilde c_{1}\min\{(1+|\mu|)^\qq 
\eta^{-\pp},|\mu|^2+\eta^{-\pp}\}$ 
for a constant $\tilde c_1>0$ that can be chosen to be independent of $0<\delta\ll1$.
This can be seen by a case distinction: 
\begin{alignat}{3}&
\big( 1 + |\mu{-}\nu| \big)^\qq 
\eta^{-\pp}
+\frac{1}{2\delta}|\nu|^2\ge c_{2} \big(1 + |\mu| \big)^\qq 
			\eta^{-\pp}\qquad &&\text{if }|\nu|\le\frac{|\mu|}{2}\text{ or }|\nu|\ge 2|\mu|, \nonumber
   \\
   &\big( 1 + |\mu{-}\nu| \big)^\qq \eta^{-\pp}
+\frac{1}{2\delta}|\nu|^2\ge c_{3}\big(|\mu|^2+\eta^{-\pp}\big)\qquad&&\text{if }\frac{|\mu|}{2}<|\nu|<2|\mu|,\label{eq:case.mu=nu}
\end{alignat}
where we estimated $|\mu-\nu|\gtrsim|\mu|$ in the first case. In the second case, we used $|\nu|\gtrsim|\mu|$ in the second term on the left-hand side of~\eqref{eq:case.mu=nu}.
We then let $\calK_\delta:\FF_\calK\to\mathbb R$ denote the functional obtained by substituting $\ol{H^*_\delta}(\mu,\eta)$ for $\ol{H^*}(\mu,\eta)$ in the definition~\eqref{eq:k-functional} of $\calK$. 
   Adjusting the proof of Proposition~\ref{prop:Coerc.calK}, it is not difficult to see that $\calK_\delta$ enjoys the coercivity property in Proposition~\ref{prop:Coerc.calK} uniformly in $0<\delta\ll1$. 
   Furthermore, the properties of $\ol{H^*_\delta}(\mu,\eta)$ imply that $\calK_\delta$ is convex and lower-semicontinuous.
   Consequently, $\calK_\delta$ possesses  at least one minimizer $z_\delta\in \FF_\calK$. The regularized functional $\calK_\delta$ is Fr\'echet differentiable and $\rmD_\lambda \calK_\delta(\eta,\kappa,\cdot)$ is globally Lipschitz continuous (see, e.g., \cite{BC_2017}). For this reason the nonlinearity in the elliptic equation $\rmD_\lambda \calK_\delta(\eta_\delta,\kappa_\delta,\lambda)=0$ has (at most) linear growth in $\lambda$ as $|\lambda|\to\infty$. We may therefore follow the argument of \cite[Thm.~4.5]{Trol10OCPD} to deduce the estimate $\|\lambda_\delta\|_{\rmL^\infty(\Omega)}\le R$, where $R$ is independent of $0<\delta\ll1$.
Using the compact embedding $\HH\hookrightarrow \rmL^2(\Omega)$ and the fact that $\ol{H^*}\ge\ol{H^*_\delta}$, one easily verifies that 
$\calK_\delta$ converges to $\calK$ in the sense of $\Gamma$-convergence weakly in $\FF_\calK$. 
Hence, recalling the uniqueness of the minimizer $z^*$ of $\calK$, we infer that 
$z_\delta\rightharpoonup z^*$ in $\FF_\calK$ (with $\lambda_\delta\overset{*}{\rightharpoonup} \lambda^*$ in $\rmL^\infty(\Omega)$)
and consequently $\|\lambda^*\|_{\rmL^\infty(\Omega)}\le R$. The continuity $\lambda^*\in \rmC(\ol\Omega)$ now follows from~\cite[Thm.~4.8]{Trol10OCPD} by noting that the uniqueness of solutions to the elliptic equation already holds in $\rmL^\infty(\Omega) \cap \rmH^1(\Omega)$. 
\end{proof}
 
By the strict convexity of $\calK$, any critical point must be a minimizer. Thus, Proposition~\ref{prop:k-minimizer} immediately entails the uniqueness and regularity of critical points of $\calK$.
\begin{corollary}[Existence, uniqueness, and regularity of critical points of $\calK$]
    \label{cor:k-critical-point}$ $ \\ 
	  Let the hypotheses of Proposition \ref{prop:k-minimizer} be in place. 
		Then, there exists a unique critical point $(\eta^\ast, 
		\kappa^\ast,\lambda^\ast)$ of $\mathcal K$ 
		on $\FF_\calK$. It satisfies $\lambda^\ast\in \rmC(\ol\Omega)$.  
\end{corollary}
 
	The next result establishes the announced bijection between critical points of 
	the Lagrange functional $\calL$ and the strictly convex functional $\calK$. 
	
	\begin{proposition}[Critical points of $\calK$ and $\calL$]
		\label{prop:CharEquil}
		Consider the Lagrange functional $\calL$ from \eqref{eq:lagrange-functional} 
		and the convex functional $\calK$ from \eqref{eq:k-functional} together with 
		the corresponding domains $\FF_\calL$ and $\FF_\calK$ 
		defined in \eqref{eq:l-domain} and \eqref{eq:k-domain}, respectively. 
		Every critical point $(\eta,\kappa,\lambda) \in \FF_\calK$ of $\calK$ generates via
		\begin{align}
			\label{eq:d-s-ast}
			(\bfc(x),u(x))= \rmD H^*\big({-}(\lambda(x) {+} \kappa)\bfq, -\eta \big)
			\quad \text{and} \quad \psi(x) = \frac{\lambda(x)}{\eta} - \psiext(x)
		\end{align}
		a critical point $(\bfc, u, \psi, \eta, \kappa, \lambda) \in 
		\FF_\calL$ of $\calL$. 
		Every critical point $(\bfc, u, \psi, \eta, \kappa, \lambda) \in 
		\FF_\calL$ of $\calL$ gives rise to a critical point 
		$(\eta, \kappa, \lambda) \in \FF_\calK$ of $\calK$ 
		and \eqref{eq:d-s-ast} holds. In particular, there exists a bijection between 
		critical points of $\calL$ in $\FF_\calL$ and critical points 
		of $\calK$ in $\FF_\calK$. 
	\end{proposition}
	\begin{proof}
		Let $(\eta, \kappa, \lambda) \in \FF_\calK$ 
        be a critical point of $\calK$. 
        Thanks to the strict convexity of $\calK$, we know that 
        $(\eta, \kappa, \lambda)$ is the unique minimizer of $\calK$, 
        which by the arguments above fulfils $\lambda \in \rmC(\ol \Omega)$. 
        Hence, the relations in \eqref{eq:d-s-ast} give rise to continuous 
        and uniformly positive $(\bfc, u) \in \FF$, which 
        ensures that the subsequent derivatives of $\calK$ are well-defined. 
		By the definition of $\ol{H^*}$ from $H^*$ and \eqref{eq:d-s-ast}, 
		we have $\rmD \ol{H^*}(\lambda {+} \kappa,\eta) = ( {-} \bfq {\cdot} \bfc, -u)$. 
		Simple calculations and $\psi = \lambda/\eta - \psiext$ show that 
		\begin{align*}
			\rmD_\eta \calK(\eta,\kappa,\lambda) &= 
			-\frac12 \calB \Big( \frac{\lambda}{\eta} \Big) - \int_\Omega u
			\dd x + E_0, \\
			\rmD_\kappa \calK(\eta,\kappa,\lambda) &= 
            - \int_\Omega \bfq {\cdot} \bfc \, \dd x + Q_0, \\ 
			\rmD_\lambda \calK(\eta,\kappa,\lambda) h &= \int_\Omega \Big( 
            \varepsilon \nabla \Big( \frac{\lambda}{\eta} - \psiext \Big) \cdot \nabla h 
            - \bfq {\cdot} \bfc \, h \Big) \dd x 
            + \int_{\Gamma_\rmR} \omega \Big( \frac{\lambda}{\eta} - \psiext \Big) h \dd a 
		\end{align*}
		with $h \in \HH$. 
Furthermore, $\bfX \coleq (\bfc, u, \psi, \eta, \kappa, \lambda) \in \FF_\calL$ 
and the algebraic conditions \eqref{eq:lagrange-c-u}--\eqref{eq:lagrange-psi}, 
i.e., $\rmD_{(\bfc,u,\psi)} \calL(\bfX) = 0$, 
hold true due to the identities \eqref{eq:d-s-ast} and the 
well-known Fenchel equivalences. 
Finally, $\rmD \calK(\eta,\kappa,\lambda)=0$ implies 
\eqref{eq:lagrange-eta-kappa-lambda}, i.e., $\rmD_{(\eta, \kappa, \lambda)} \calL(\bfX) = 0$. 

Vice versa, given a critical point $(\bfc, u, \psi, \eta, \kappa, \lambda) 
\in \FF_\calL$ of $\calL$, 
the previous identities on $\rmD \calK (\eta, \kappa, \lambda) = 0$ 
and \eqref{eq:d-s-ast} follow from \eqref{eq:lagrange-eta-kappa-lambda}, 
\eqref{eq:lagrange-c-u}, and \eqref{eq:lagrange-psi}. 
\end{proof}

	We are now in a position to prove the existence of a unique critical point  
    $(\bfc^\ast, u^\ast) \in \FF$ of the total entropy $\calS$ 
    under the charge and energy constraints \eqref{eq:constraints}. 

\bigskip
\noindent
\begin{proof}[Proof of Theorem \ref{thm:existence-uniqueness}]
		By Proposition \ref{prop:CharEquil}, there is a one-to-one correspondence between 
		critical points of $\calL$ in $\FF_\calL$, 
        i.e., solutions to the Lagrange equations 
        \eqref{eq:lagrange-eta-kappa-lambda}--\eqref{eq:lagrange-psi}, 
		and critical points of $\calK$ in $\FF_\calK$. Hence, 
		by Corollary \ref{cor:k-critical-point}, there exists a unique solution 
        $(\bfc, u, \psi, \eta, \kappa, \lambda) \in \FF_\calL$
        to \eqref{eq:lagrange-eta-kappa-lambda}--\eqref{eq:lagrange-psi}. 
        Recalling the expressions for the derivatives of $\calE(\bfc, u)$ 
        and $\calQ(\bfc, u)$ from \eqref{eq:derivatives-energy-charge} and 
        demanding $\lambda = \eta (\psi {+} \psiext)$, we see 
        that \eqref{eq:lagrange-c-u} is equivalent to the 
        functional derivative 
        \begin{align*}
            \rmD \calS(\bfc, u) = \eta \rmD \calE(\bfc, u) + \kappa \rmD \calQ(\bfc, u),
        \end{align*}
        and we observe that \eqref{eq:lagrange-eta-kappa-lambda} is equivalent to 
        $\calE(\bfc, u) = E_0$, $\calQ(\bfc, u) = Q_0$, and 
        Poisson's equation \eqref{eq:poisson}--\eqref{eq:poisson-bc}. 
        This shows that $(\bfc^\ast, u^\ast) \in \FF$ 
        as defined in \eqref{eq:d-s-ast} is the unique 
        critical point of $\calS$ under the constraints \eqref{eq:constraints}. 
        The continuity of the associated potential $\Psi^\ast = \psi^\ast + 
        \psi_\mathrm{ext} = \lambda^\ast/\eta^\ast$ 
        follows from Corollary \ref{cor:k-critical-point}. 
        Finally, \eqref{eq:lagrange-c-u} entails that $1/\theta^\ast = 
        \rmD_u S(\bfc^\ast, u^\ast) = \eta^\ast > 0$ is constant. 
	\end{proof}

\section{Global optimization by the direct method}
\label{se:proofs-DirectMethod}
In this section, we establish the existence, uniqueness, and regularity of a global optimizer as asserted in Subsection~\ref{su:DirectMethod}.
Throughout this section, we assume Hypotheses~\ref{hp:h.pr.cx.lsc}--\ref{hp:S.infinity}.

\subsection{Lower semicontinuity of the extended negative entropy}\label{ssec:ext-meas}

Recall the extended negative entropy functional $\mathcal{\wt H}:U\to\mathbb{R}\cup\{\infty\}$ defined in \eqref{eq:negative-entropy-functional}. In the following, we will consider the Legendre transform of $H$ as a function defined on the whole space $H^*:\mathbb{R}^{I+1}\to\mathbb{R}\cup\{\infty\}$ via
	\begin{align*}
		H^*(\bfzeta)=\sup_{\bfz\in\mathbb{R}^{I+1}}\{\bfzeta\cdot \bfz-H(\bfz)\},\quad \bfzeta\in\mathbb R^{I+1}.
	\end{align*}
Property~\ref{hp:h.pr.cx.lsc} implies that the restriction of $H$ to $\text{dom}\, H$ is continuous.
Furthermore, $H^*$ is proper, convex, and lower semicontinuous, so that likewise $H^*:\text{dom}\,H^*\to\mathbb{R}$ is continuous.
Finally, note that due to the Fenchel--Moreau theorem and~\ref{hp:h.pr.cx.lsc}, it holds that $(H^*)^*=H$.

In the following lemma we show the crucial lower semicontinuity with respect to (weak, weak-star) convergence in $U$.
\begin{lemma}\label{l:usc-measure}
The extended negative entropy $\mathcal{\wt H}:U\to \mathbb R\cup\{\infty\}$ is lower semicontinuous with respect to (weak, weak-star) convergence in $\rmL^1(\Omega)^I\times M(\ol \Omega)$. More precisely, whenever $(\bfc_j,\nu_j),(\bfc,\nu)\subset U$ with $\bfc_j\rightharpoonup \bfc$ in $\rmL^1(\Omega)$
and $\nu_j\overset{*}{\rightharpoonup} \nu$ in $M_+(\ol\Omega)$, then
\begin{align}\label{eq:uscS}
    \mathcal{\wt H}(\bfc,\nu)\le \liminf_{j\to\infty} \mathcal{\wt H}(\bfc_j,\nu_j).
\end{align}
\end{lemma}
\begin{proof}
	The proof is based on ideas from convex duality as appearing in~\cite{DT_1984}. 
	First, we recall the well-known fact that the properties~\ref{hp:h.pr.cx.lsc} of $H$ allow us to express the functional $\mathcal{H}: \rmL^1_+(\Om)^{I+1}\to\mathbb{R}\cup\{\infty\}$ in terms of its convex conjugate (see, e.g., \cite{Brezis_1972,DT_1984}): for all $\bfz\in \rmL^1_+(\Om)^{I+1}$, 
	\begin{align}\label{eq:duality.dens}
		\mathcal{H}(\bfz) = \sup_{\bfzeta\in\mathcal{D}_H(X)}\bigg(\int_\Omega \bfz\cdot \bfzeta\dd x - \int_\Omega H^*(\bfzeta)\dd x\bigg),
	\end{align}
	where $\mathcal{D}_H(X) \coleq \{\bfzeta\in X:H^*\circ\bfzeta\in \rmL^1(\Omega)\}$ and $X \coleq \rmC(\ol\Om)^{I+1}$ for the rest of this proof, 
	which is composed of two steps. The main ingredient is an extension of formula~\eqref{eq:duality.dens} to the measure setting, which will be established in the first step.	
	
Step~1: \textit{Duality formula for measure}.
We assert that, for every $\nu=u\,\dn x+\nu^s\in M_+(\ol\Omega)$ and $\bfc\in \rmL^1_+(\Omega)^I$, 
\begin{align}\label{eq:duality.meas}
    \mathcal{\wt H}(\bfc,\nu) = \sup_{\bfzeta=(\bfzeta_I,\zeta_0)\in\mathcal{D}_H(X)}\bigg(\langle\nu,\zeta_0\rangle +\int_\Omega \bfc\cdot \bfzeta_I\dd x - \int_\Omega H^*(\bfzeta)\dd x\bigg),
\end{align}
where $\langle\nu,\zeta_0\rangle:=\int_{\ol\Om}\zeta_0\dd \nu$, which equals $\int_{\Om}\zeta_0u\dd x+\int_{\ol\Om}\zeta_0\dd \nu^s$. 

\noindent\textit{Proof of Step~1:} 
Since $ \mathcal{\wt H}(\bfc,\nu) =	\mathcal{H}(\bfc,u)$, it suffices to show that the right-hand side of the asserted identity agrees with that of formula~\eqref{eq:duality.dens} for $\bfz=(\bfc,u)$.

(a) `\textit{RHS\;\eqref{eq:duality.meas}$\le$RHS\;\eqref{eq:duality.dens}}':
Since $u\mapsto H(\bfc,u)$ is sublinear as $u\uparrow\infty$, it follows that $H^*(\bfzeta_I,\zeta_0)=+\infty$ whenever $\zeta_0>0$. Indeed,  the definition of the Legendre transform implies, for $(\bfzeta_I,\zeta_0)\in \mathbb R^I\times\mathbb R_{>0}$,
\[H^*(\bfzeta_I,\zeta_0)\ge \sup_{u\ge0}\big(\zeta_0u-H(0,u)\big)\ge\limsup_{u\to\infty}u(\zeta_0+o(1))=+\infty.\]
Therefore, $\zeta_0\le0$ in $\Om$ for all $\bfzeta=(\bfzeta_I,\zeta_0)\in \mathcal{D}_H(X)$. Hence,
$\la\nu^s,\zeta_0\ra\le0$ for all $\bfzeta=(\bfzeta_I,\zeta_0)\in\mathcal{D}_H(X)$, showing the inequality asserted in (a).

(b) `\textit{RHS\;\eqref{eq:duality.meas}$\ge$RHS\;\eqref{eq:duality.dens}}':
Let $\bfzeta=(\bfzeta_I,\zeta_0)\in\mathcal{D}_H(X)$ be given. Since $\nu^s\perp\mathcal{L}^d$, there exists for every $\delta>0$ a relatively open set $\widehat\Om_\delta\subset\ol\Om$ with $\mathcal{L}^d(\widehat\Om_\delta)\le \delta/2$ such that $\nu^s(\ol\Om\setminus\widehat\Om_\delta)=0$. Choose a relatively open set $\Om_\delta\subset\ol\Om$ with $\Om_\delta\supset\supset\widehat\Om_\delta$ and $\mathcal{L}^d(\Om_\delta)\le\delta$, and a function
$\eta_\delta\in \rmC(\ol\Om,[0,1{-}\ve])$ with $\eta_\delta=1{-}\ve$ on $\widehat\Om_\delta$ for suitable $\ve=\ve(\delta)\in(0,1)$ and with $\text{supp}\,\eta_\delta\subset\Om_\delta$. Then consider  $\hat\bfzeta=(\hat\bfzeta_I,\hat\zeta_0) \coleq (\bfzeta_I,(1-\eta_\delta)\min\{\zeta_0,-\chi_{\Om_\delta}\})\in \mathcal{D}_H(X)$ as a competitor in~\eqref{eq:duality.meas}. Since $\zeta_0\le0$,  it holds that $\hat\bfzeta\equiv\bfzeta$ on $\ol\Om\setminus\Om_\delta$ and hence
\begin{align*}
	\langle\nu,\hat\zeta_0\rangle +\int_\Omega \bfc\cdot \hat\bfzeta_I\dd x - \int_\Omega H^*(\hat\bfzeta)\dd x
	=\int_\Omega \bfz\cdot \bfzeta\dd x 
	- \int_\Omega H^*(\bfzeta)\dd x + R_\delta,
\end{align*}
where 
\begin{align*}
R_\delta \coleq  \int_{\Omega_\delta} u\cdot (\hat\zeta_0-\zeta_0)\dd x
+\int_{\Omega_\delta} (H^*(\bfzeta)-H^*(\hat\bfzeta))\dd x+\int_{\ol\Om}\hat\zeta_0\dd\nu^s.
\end{align*}
It remains to show that $\limsup_{\delta\to0}R_\delta\ge0$. Once established, this implies that for every $\bfzeta\in\mathcal{D}_H(X)$, there exists a sequence $\hat\bfzeta^{(N)}$, $N\to\infty$ as $\delta\downarrow0$, such that 
\begin{align*}
\limsup_{N\to\infty}\bigg(\langle\nu,\zeta_0^{(N)}\rangle +\int_\Omega \bfc\cdot \bfzeta_I^{(N)}\dd x - \int_\Omega H^*(\bfzeta^{(N)})\dd x\bigg)\ge 
\int_\Omega \bfz\cdot \bfzeta\dd x - \int_\Omega H^*(\bfzeta)\dd x,
\end{align*}
which implies the asserted inequality in (b) by means of a diagonal argument.

The first term in $R_\delta$ is easily handled, because
$|\int_{\Omega_\delta} u\cdot (\hat\zeta_0-\zeta_0)\dd x|\le \big( \|\hat\zeta_0\|_{\rmC(\ol\Om)}+\|\zeta_0\|_{\rmC(\ol\Om)} \big) \|u\|_{\rmL^1(\Om_\delta)}\overset{\delta\downarrow0}{\to}0$.
For the second term, we note that the increase of the function $(-\infty,0)\ni s\mapsto H^*(\bfzeta_I,s)\in\mathbb{R}$ implies $H^*(\bfzeta_I,{-}\|\hat\zeta_0\|_{\rmC(\ol\Om)}{-}1)\le H^*(\hat\bfzeta)\le H^*(\bfzeta_I,-\ve)$ in $\Om_\delta$, so that choosing the parameter $\ve=\ve(\delta)=o(1)$ to decrease sufficiently slowly as $\delta\downarrow0$, we can ensure that the sequence $\chi_\delta H^*(\hat\bfzeta)$ is uniformly integrable as $\delta\to0$, and therefore also the sequence $g_\delta= \chi_\delta(H^*(\bfzeta)-H^*(\hat\bfzeta))$. Thus, $\int_{\Om}g_\delta\dd x\to0$, since $g_\delta\to0$ pointwise a.e.\ in $\Om$. Observe that in this argument, $\ve$ can be chosen 
independently of $x\in \ol\Om$, since $\ol\Om$ is compact and $\bfzeta_I$ is continuous.
Finally, we note that thanks to $\lim_{\delta\downarrow0}\ve=0$ and the fact that $\hat\zeta_0\equiv-\ve$ on $\text{supp}\,\nu^s$, the third integral term in $R_\delta$ clearly 
converges to $0$ in the limit $\delta\to0$.
In conclusion, we infer $\lim_{\delta\downarrow0}R_\delta=0$ and thus (b).

Combining (a) and (b) yields the identity~\eqref{eq:duality.meas}.

Step~2: \textit{Proof of~\eqref{eq:uscS} by duality formula.}
Owing to~\eqref{eq:duality.meas}, there exists a sequence $\bfzeta^{(N)}=\big(\bfzeta_I^{(N)},\zeta_0^{(N)}\big) \in \mathcal{D}_H(X)$, $N\in \mathbb{N}$, such that 
\begin{align}\label{eq:lsc.pf-1}
\mathcal{\wt H}(\bfc,\nu) = \lim_{N\to\infty} \bigg(\big\langle\nu,\zeta_0^{(N)}\big\rangle +\int_\Omega \bfc\cdot \bfzeta_I^{(N)}\dd x - \int_\Omega H^*\big(\bfzeta^{(N)}\big)\dd x\bigg).
\end{align}
Furthermore, for every $N\in \mathbb{N}$, 
\begin{align*}
	\big\langle\nu,\zeta_0^{(N)}\big\rangle &+\int_\Omega \bfc\cdot \bfzeta_I^{(N)}\dd x - \int_\Omega H^*\big(\bfzeta^{(N)}\big)\varphi\dd x
	\\&=\lim_{j\to\infty}\bigg(\big\langle\nu_j,\zeta_0^{(N)}\big\rangle +\int_\Omega \bfc_j\cdot \bfzeta_I^{(N)}\dd x - \int_\Omega H^*\big(\bfzeta^{(N)}\big)\dd x\bigg)
	\le \liminf_{j\to\infty} \, \mathcal{\wt H}(\bfc_j,\nu_j),
\end{align*}
where in the second step, we used once more formula~\eqref{eq:duality.meas}.
The right-hand side of this inequality is independent of~$N$. Hence, sending $N\to\infty$ and using~\eqref{eq:lsc.pf-1} gives~\eqref{eq:uscS}.       
\end{proof}

In order to carry out the direct method, we need to ensure that the feasible set $M_{E_0,Q_0}$ resp.\ $\wt M_{E_0,Q_0}$ is non-empty. In the following, we provide  sufficient criteria by elaborating on Proposition~\ref{pr:MinElectroEner}.
For simplicity, we only consider the case $\min_iq_i<0<\max_j q_j$, which allows us to invoke Proposition~\ref{pr:MinElectroEner}.

\begin{lemma}[Non-emptiness of feasible set] \label{l:feasible.set}
        Assume that $\min_iq_i<0<\max_j q_j$, and let $E_0>\rmV(Q_0,\psie)$. 
        Then, $M_{E_0,Q_0}\not=\emptyset$ and 
		$\inf_{\wt M_{E_0,Q_0}}\mathcal{\wt H}<+\infty$.
	\end{lemma}
	\begin{proof}
		It follows from \ref{Fact1} in Lemma \ref{lemma:l1-density} 
        and the proof of Proposition~\ref{pr:MinElectroEner} 
        that the infimum in $\rmV(Q_0,\psie)$ can equivalently 
        be taken over $\rho\in \rmC(\ol\Om)$, i.e., 
	\[
	\rmV(Q_0,\psiext)= \inf\Bigset{ \rmE(\rho,\psiext)}{ 
		\rho \in \rmC(\ol\Omega),\ \int_\Omega \rho \dd x
		=Q_0 } .
	\]
(Note that we employ the identification $\rmC(\ol \Omega) \subset \HH^\ast$ in this context.) 
Thus, given $E_0>\rmV(Q_0,\psie)$, there exists $\hat\rho\in \rmC(\ol\Om)$ with 
$\int_\Omega \hat\rho \dd x=Q_0$ such that $\rmE(\hat\rho,\psiext)\le E_0$. By 
the compatibility of $(\bfq,Q_0)$, there exists 
$\bfc\in \rmC(\ol\Om,[0,\infty)^I)$ with $\bfq{\cdot}\bfc=\hat\rho$.
As $\rmE(\hat\rho,\psiext)\le E_0$, there exists $u\in \rmC(\ol\Om,[0,\infty))$ 
such that $\mathcal{E}(\bfc,u)=E_0$.
Then, $\bfz=(\bfc,u)\in M_{E_0,Q_0}$, and the boundedness of $\bfz$ guarantees 
that $\int_\Om H(\bfz)\dd x<\infty$. Consequently, $\inf_{\wt M_{E_0,Q_0}}\mathcal{\wt H}<\infty$.
\end{proof}

In the borderline case $E_0=\rmV(Q_0,\psie)$, there may be surface 
concentrations, and  the non-emptiness of $M_{E_0,Q_0}$ hinges upon 
the choice of the data, see Remark~\ref{rem:Attainment}.

\subsection{Existence and uniqueness of optimizer}
	
\begin{proof}[Proof of Theorem~\ref{thm:ex-uniq.direct}]
Let $s^* \coleq \sup_{\wt M_{E_0,Q_0}}\mathcal{\wt S}$.
We proceed in three steps:

Step~1: \textit{Absolute continuity of optimizer.}\,
Suppose that $(\bfc,\nu)\in \wt M_{E_0,Q_0} $ satisfies 
$\mathcal{\wt S}(\bfc,\nu)=s^*$, and let 
$\nu=u\dd x+\nu^s$ denote the Lebesgue decomposition of $\nu$. 
We argue by contradiction and assume that $\nu^s(\ol\Omega)>0$. 
Then $\hat u \coleq u+|\Omega|^{-1}\nu^s(\ol\Omega)>u$ a.e.\ in 
$\Omega$, and since $u\mapsto S(\bfc,u)$ is strictly increasing, we deduce
$\mathcal S(\bfc,\hat u)>\mathcal S(\bfc,u)=s^*$. This contradicts 
the definition of $s^*$, since by construction $(\bfc,\hat u\dd x)\in \wt M_{E_0,Q_0}$.

Step~2: \textit{Uniqueness.}\, As shown in Step~1, any solution 
$\bfmu\in \wt M_{E_0,Q_0}$ of the optimization problem must be 
absolutely continuous with respect to the Lebesgue measure. 
Thus, it suffices to establish uniqueness in $M_{E_0,Q_0}$.

We argue by contradiction. Suppose that there exist  $\bfz_i=(\bfc_i,u_i)\in M_{E_0,Q_0}$, $i=0,1$, $\bfz_0\not\equiv \bfz_1$,
with $\mathcal{S}(\bfz_i)=s^*$ for $i=0,1$.
Then consider $\bfz_\vartheta=(\bfc_\vartheta,u_\vartheta)=\vartheta \bfz_1+(1{-}\vartheta)\bfz_0$, 
so that $\psi_{\bfc_\vartheta}=\vartheta\psi_{\bfc_1} + (1{-}\vartheta) \psi_{\bfc_0}$.
The strict concavity of $\mathcal{S}$ implies that 
\begin{align*}
	\mathcal{S}(\bfz_\vartheta)>s^* \quad \text{for all } \vartheta\in(0,1).
\end{align*}
On the other hand, since $\bfz\mapsto\mathcal{E}(\bfz)$ is convex, it holds that 
\begin{align*}
	\mathcal{E}(\bfz_\vartheta)\le E_0.
\end{align*}
Fix any $\vartheta\in(0,1)$, e.g., $\vartheta=1/2$. Next, choose $ \hat u_\vartheta  \in \rmL^1_+ (\Omega) $ with $\hat u_\vartheta\ge u_\vartheta$ such that 
$\mathcal{E}(\bfc_\vartheta,\hat u_\vartheta)=E_0:$ for instance, 
take $\hat u_\vartheta=u_\vartheta+\frac{1}{|\Omega|}(E_0-\mathcal{E}(\bfz_\vartheta))$.
Then, $(\bfc_\vartheta,\hat u_\vartheta)\in \rmL^1_+(\Omega)^{I+1}$ 
with $\mathcal{E}(\bfc_\vartheta,\hat u_\vartheta)=E_0$ and 
$\mathcal{Q}(\bfc_\vartheta,\hat u_\vartheta)=Q_0$.
Since $\hat u_\vartheta\ge u_\vartheta$ pointwise a.e., we must have 
$\mathcal{S}(\bfc_\vartheta,\hat u_\vartheta)\ge\mathcal{S} 
(\bfc_\vartheta, u_\vartheta)>s^*$, which contradicts the definition of $s^*$.
    
Step~3: \textit{Convergence and existence.}\,
Let $(\bfmu_k)_k\subset \wt M_{E_0,Q_0}$ be such that 
$\mathcal{\wt S}(\bfmu_k)\to s^*$. Denote $\bfmu_k \eqcol (\bfc_k,\nu_k)$.
Using the coercivity~\eqref{eq:hp.coerc.S}  of $S$ and the 
definition~\eqref{eq:energy.ext} of $\mathcal{\wt E}$, it is easy to see 
that the sequence $(\bfc_k)_k$ is uniformly integrable (as $\gamma$ in 
\eqref{eq:hp.coerc.S} is superlinear) and that the sequence $(\nu_k(\ol\Omega))_k\subset\mathbb{R}$ is bounded. 
Hence, there exists $\bfmu_*=(\bfc_*,\nu_*)\in U$ with 
$\mathcal{\wt Q}(\bfmu_*)=Q_0$ such that, after passing to a subsequence,
$\bfc_k\rightharpoonup \bfc_*$ in $\rmL^1(\Omega)$, 
$\nu_k\overset{*}{\rightharpoonup}\nu_*$ in $M_+(\ol\Omega)$.
The (weak, weak-star) upper semicontinuity of $\mathcal{\wt S}$, 
see Lemma~\ref{l:usc-measure}, then yields $\mathcal{\wt S}(\bfmu_*)\ge s^*$. 
Since  $\mathcal{\wt E}(\bfmu_k)=E_0$ for all $k$, there 
exists $\tilde \psi\in \HH$ such that along another subsequence 
\begin{align*}
    \psi_{\bfc_k}\rightharpoonup\tilde\psi\quad \text{ in }\HH,
\end{align*}
and thus, in view of the Poisson equation~\eqref{eq:poisson-parts.1}, 
Lemma~\ref{lemma:poisson.gen}, and the weak continuity of the operator $L$, also 
\[\bfq{\cdot}\bfc_k=L\psi_{\bfc_k}
\rightharpoonup 
L\tilde\psi\quad\text{ in }\HH^*.\]
On the other hand, $\bfq{\cdot}\bfc_k\rightharpoonup \bfq{\cdot}\bfc_*$ in $\rmL^1(\Omega)$, 
which by the uniqueness of the weak limit (in the sense of distributions) implies that 
$\bfq{\cdot}\bfc_*=L\tilde\psi\in\HH^*\cap \rmL^1(\Omega)$.
Consequently, $\bfmu_*\in\text{dom}\,\mathcal{\wt E} \coleq \{\mathcal{\wt E}<\infty\}$, 
$\tilde\psi=\psi_{\bfc_*}$, and $\mathcal{\wt E}(\bfmu_*)\le E_0$.

We now assert that $\mathcal{\wt E}(\bfmu_*)= E_0$. 
To show this, we argue indirecty and assume that the energy gap is 
non-trivial, i.e., $\mathcal{\wt E}(\bfmu_*)<E_0$.  Then, as before, 
we can find a finite measure $\hat\nu_*\ge\nu_*$ such that 
$\mathcal{\wt E}(\bfc_*,\hat\nu_*)=E_0$ and $\mathcal{\wt S}(\bfc_*,\hat\nu_*)>s^*$, 
where the strict inequality can be achieved due to the strict monotonicity of $u\mapsto S(\bfc,u)$. 
This, however, contradicts the definition of $s^*$, and we conclude that $\mathcal{\wt E}(\bfmu_*)= E_0$. 
Hence, $\bfmu_*=(\bfc_*,\nu_*)\in \wt M_{E_0,Q_0}$ is an optimizer and 
we obtain the asserted (weak, weak-star) convergence of the optimizing 
sequence $\bfmu_k$ to $\bfmu_*$. The strong convergence $\psi_{\bfc_k}\to\psi_{\bfc_*}$ in $\HH$ follows from the weak convergence $
\psi_{\bfc_k}\wk\psi_{\bfc_*}$ in $\HH$ combined with the fact that 
$\mathcal{\wt E}(\bfmu_k)=E_0\to\mathcal{\wt E}(\bfmu_*)$, which implies 
the convergence of the norm $\|\psi_{\bfc_k}\|_{\HH}\to\|\psi_{\bfc_*}\|_{\HH}$. 
By the uniqueness of the optimizer (see Step~2), the convergence 
results established above are true for the entire sequence $(\bfmu_k)_k$.

This completes the proof of Theorem~\ref{thm:ex-uniq.direct}.
\end{proof}

\begin{remark}[Relaxed hypotheses for existence]\label{rem:non-strict.ex.opt}
The proofs of the weak-star lower semicontinuity in Lemma~\ref{l:usc-measure} 
and of Theorem~\ref{thm:ex-uniq.direct} above show that for the mere existence of 
an absolutely continuous optimizer the strict convexity hypothesis of $H=-S$ on 
$\ol\sfD$ can be relaxed to non-strict convexity, and the assumption of a 
strict decrease of $u\mapsto H(\bfc,u)$ can be weakened to $u\mapsto H(\bfc,u)$ 
being non-increasing.
\end{remark}

\section{Relation between the two solution approaches}
\label{se:completing}

The main purpose of this section is to close the gap between the results obtained by
the dual Lagrangian approach in Section~\ref{se:proofs-LagrangianApproach} and by 
the direct optimization under constraints in Section~\ref{se:proofs-DirectMethod}. 
We provide two different ways to close the gap. 
First, we provide the proof of Theorem \ref{thm:reg-globalopt.v1}, which states 
that the global maximizer $\bfz_*=(\bfc_*,u_*)$ of $\calS$ on $M_{E_0,Q_0}$ is continuous.
Secondly,  we establish Theorem \ref{th:GlobEquilibrium} showing that the critical 
point of the Lagrange function $\calL$ provides the global 
maximizer of the entropy $\calS$ under the given constraints.

\subsection{Regularity of global optimizers}\label{ssec:reg.opt}

To show the continuity of the global optimizer, we study a 
$\delta$-regularized problem  where the smoothness of the functional allows 
us to follow the classical Hilbert-space approach. 
We are then able to show that the optimizers $ \bfz_\delta $ 
converge to $ \bfz_* $ for $\delta \to 0$ in such a way that 
continuity and uniform positivity of $ \bfz_* $ can be deduced. 

\subsubsection{A regularized problem}
We introduce the $\delta$-Moreau envelope $H_\delta$ of $H$
\begin{align}\label{eq:defMoreau}
    H_\delta(\bfz) \coleq \inf_{\bfz'\in\mathbb R^{I+1}}\Big(H(\bfz{-}\bfz') + 
    \frac{1}{2\delta}|\bfz'|^2\Big), \qquad \bfz=(\bfc,u)\in \mathbb R^{I+1},
\end{align}
which is a convex and lower semicontinuous function from $\mathbb{R}^{I+1}$ 
to $\mathbb{R}$. Owing to~\ref{hp:h.pr.cx.lsc}, the Moreau envelope of $H$ 
is exact (cf.~\cite[Prop.~12.14]{BC_2017}), i.e., for all $\bfz\in\mathbb{R}^{I+1}$ 
the infimum on the right-hand side is attained at some 
$\bar \bfz\in\mathbb{R}^{I+1}$: \ $H(\bfz-\bar \bfz)+\frac{1}{2\delta}|\bar \bfz|^2=H_\delta(\bfz)$.
Using~\ref{hp:mon.h}, it is then easy to see that $u\mapsto H_\delta(\bfc,u)$ 
is strictly decreasing. Indeed, for $\bfc\in \mathbb R^I$ 
and for $u_1<u_2$, let $\bar \bfz^{(i)}\in \mathbb{R}^{I+1}$ be such that 
\begin{align*}
H\big( (\bfc,u_i){-}\bar \bfz^{(i)} \big)+\frac{1}{2\delta} \big|\bar \bfz^{(i)} \big|^2
=H_\delta(\bfc,u_i).
\end{align*}
Then, 
\begin{align*}
\infty>H_\delta(\bfc,u_1)
&= H\big( (\bfc,u_1){-}\bar \bfz^{(1)} \big)+\frac{1}{2\delta}|\bar \bfz^{(1)}|^2 \\ 
&> H\big( (\bfc,u_2){-}\bar \bfz^{(1)} \big)+\frac{1}{2\delta}|\bar \bfz^{(1)}|^2
\ge H_\delta(\bfc,u_2).
\end{align*}
Since $H_\delta$ is $\rmC^1$ and convex, we conclude that
\begin{align}\label{eq:strict-decr.del}
\pa_uH_\delta(\bfc,u)<0\text{ for all }(\bfc,u)\in\mathbb{R}^{I+1}.
\end{align} 
An equally elementary argument using the exactness of the infimal 
convolution~\eqref{eq:defMoreau} and the $1$-convexity of $\frac{1}{2}|\cdot|^2$ 
shows that the strict convexity of $H$ is inherited by $H_\delta$.

We next discuss the behavior of $H_\delta$ at infinity. Using~\eqref{eq:hp.coerc.S} 
and the identity $\iota_{\ol\sfD}(\bfc,u)=\iota_{[0,\infty)^I}(\bfc)+\iota_{[0,\infty)}(u)$,  
we estimate for $\bfz\in \mathbb{R}^{I+1}$, 
\begin{align}
H(\bfz)\ge H_\delta(\bfz)&\ge \inf_{\bfz'\in\mathbb{R}^{I+1}}
  \Big(\gamma(\bfc {-}\bfc') - K_1\sigma((u{-}u')_+) - K_0
  +\frac{1}{2\delta}|\bfz'|^2+\iota_{\ol\sfD}(\bfz{-}\bfz')\Big) \nonumber
\\&=\hat\gamma_\delta(\bfc)+\lambda_\delta(u) - K_0, \label{eq:h.del.infty}
\end{align}
where $\hat\gamma_\delta$, $\hat\lambda_\delta$ denote the Moreau envelopes of 
$\hat\gamma \coleq \gamma + \iota_{[0,\infty)^I}$ resp.\ of
$\lambda=-K_1\sigma+\iota_{[0,\infty)}$. 
The right-hand side can be estimated below by $\hat\gamma_\delta(\bfc)+\lambda(u) {-}\tilde K_0$,
showing the sublinearity at infinity of $H_\delta(\bfc,u)$ in $u$, i.e., 
$\lim_{t\to\infty}t^{-1}H_\delta(\bfc,tu)=0$.
From~\eqref{eq:h.del.infty} we can further read off the superlinearity of 
$H_\delta(\bfc,u)$ as $|\bfc|\to\infty$ resp.\ as $u\to-\infty$. 

Finally, let us note that the behavior of the infimal convolution under Legendre 
transformation (cf.~\cite[Prop.~13.21]{BC_2017}) allows us to express 
the Legendre transform of the Moreau envelope in terms of $H^*:$  
\begin{align}\label{eq:Morenv.LT}
(H_\delta)^*(\bfxi) = H^*(\bfxi) +\frac{\delta}{2}|\bfxi|^2,\quad\bfxi\in\mathbb{R}^{I+1}.
\end{align}

\paragraph{Existence of optimizer.}
We introduce the regularized functional $\mathcal{H}_\delta:\rmL^1(\Om)^{I+1} 
\to \mathbb{R}\cup\{\infty\}$, $\mathcal{H}_\delta(\bfz) \coleq \int_\Om H_\delta(\bfz)\dd x$.
Following Subsection~\ref{ssec:ext-meas}, we extend 
$\mathcal{H}_\delta$ to a convex functional on
\[	\widehat U \coleq \rmL^1(\Om)^{I}\times M, \]
where $M$ denotes the cone of signed Radon measures on 
$\ol\Om$ whose singular part with respect to the Lebesgue measure is non-negative, that is,
\begin{align}
M \coleq \{\nu\in M(\ol\Om): \nu=u\,\dn x+\nu^s\text{ with }u\in \rmL^1(\Om),\;\nu^s\in M_+(\ol\Om)\}.
\end{align}
The extension is then defined as
$\mathcal{\whh H}_\delta:\widehat U\to\mathbb{R}\cup\{\infty\}$, 
$(\bfc,u\,\dn x+\nu^s)\mapsto \mathcal{H}_\delta(\bfc,u)$. 
We also let $\mathcal{\whh H}:\widehat U\to\mathbb{R}\cup\{\infty\}$, 
$(\bfc,u\,\dn x+\nu^s)\mapsto \int_\Om H(\bfc,u)\dd x$. 
The extended energy functional $\mathcal{\whh E}$ is defined by the same 
formulas as in~\eqref{eq:energy.ext}, and the extension of the 
total charge functional is $\mathcal{\whh Q}(\bfc,\nu) \coleq 
\int_{\Om}\bfq{\cdot}\bfc\dd x$. Let us note that both $\mathcal{E}$ 
and $\mathcal{Q}$ can be defined on the entire space 
$\rmL^1(\Om)^{I+1}$, not only the positive cone $\rmL^1_+(\Om)^{I+1}$, 
and in the present section we understand 
$\mathcal{E}:\rmL^1(\Om)^{I+1}\to\mathbb{R}\cup\{\infty\}$, $\mathcal{Q}:\rmL^1(\Om)^{I+1}\to\mathbb{R}$.
Arguments similar to those in the proof of Lemma~\ref{l:usc-measure} 
(see also~\cite{FL_2007}) show that $\mathcal{\whh H}_\delta$ is 
lower semicontinuous with respect to weak/weak-star convergence. 
More precisely, if $\bfc_j\rightharpoonup \bfc$ in $\rmL^1(\Om)^I$ 
and $\hat\nu_j=u_j\,\dn x+\nu_j$ with $u_j\rightharpoonup u$ in 
$\rmL^1(\Om)$ and $\nu_{j}\overset{*}{\rightharpoonup} \nu$ in $M_+(\ol\Om)$, then 
for $\hat\nu \coleq u\,\dn x+\nu$, it holds that
\begin{align}\label{eq:lsc.weakstar.reg}
\mathcal{\whh H}_\delta(\bfc,\hat\nu)\le\liminf_{j\to\infty}\mathcal{\whh H}_\delta(\bfc_j,\hat\nu_j).
\end{align}

Define the constrained set $\whh M_{E_0,Q_0} \coleq \{(\bfc,\nu)\in \whh U: 
\mathcal{\whh E}(\bfc,\nu)=E_0,\;\mathcal{\whh Q}(\bfc,\nu)=Q_0\}$. 
It is easy to see that $\whh M_{E_0,Q_0}\not=\emptyset$ for all $(E_0,Q_0)
\in\mathbb{R}^2$, since $\whh U$ does not involve the positivity constraint.
Since $H_\delta:\mathbb{R}^{I+1}\to\mathbb{R}$ is continuous and thus 
locally bounded, we further observe that $\sup_{\whh M_{E_0,Q_0}}
\mathcal{\whh S}_\delta>-\infty$ for all $\delta\in(0,1]$ and any 
$(E_0,Q_0)\in\mathbb{R}^2$. We can thus solve the optimization problem for 
$\mathcal{\whh S}_\delta=-\mathcal{\whh H}_\delta$ by adapting the proof 
of Theorem~\ref{thm:ex-uniq.direct} and obtain an optimizer 
$\bfz_\delta=(\bfc_\delta,u_\delta)\in \rmL^1(\Om)^{I+1}$ so that $\mathcal{E}(\bfz_\delta)=E_0,\mathcal{Q}(\bfz_\delta)=Q_0$.
In order to deal with non-vanishing negative parts of the $u$-component, 
one needs to use the superlinearity of $H_\delta(\bfc,u)$ as $u\to-\infty$.

\paragraph{Lagrange multiplier rule.}
We first note that the fact that $H_\delta$ is $\rmC^1$ and convex with at most quadratic growth at infinity implies the bound 
\begin{align}\label{eq:DH.growth}
    |\rmD H_\delta(\bfz)|\le A_{1,\delta}+A_{2,\delta}|\bfz|\quad\text{ for all }\bfz\in\mathbb R^{I+1}.
\end{align}
Since $\bfz_\delta \coleq (\bfc_\delta,u_\delta)$ is an optimizer, 
any curve of the form 
$\bfz(t)=\bfz_\delta+t\tilde \bfz$ with $\tilde \bfz\in \rmC^\infty(\ol\Om)^{I+1}$ and 
\begin{align}\label{eq:dir.perp}
\tilde \bfz\perp^{\rmL^2}\text{span}\,\{\rmD \mathcal{E}(\bfz_\delta),\rmD \mathcal{Q}(\bfz_\delta)\}
\end{align}
must obey the stationarity condition $\frac{\dn}{\dn t}\big|_{t=0}\mathcal{H}_\delta(\bfz)=0$. 
Differentiating under the integral by virtue of~\eqref{eq:DH.growth} and the dominated convergence theorem then yield
\begin{align}\label{eq:LagrMul.1}
\int_\Om \rmD H_\delta(\bfz_\delta)\cdot \tilde \bfz\dd x=0.
\end{align}
Let $P_\TS:\rmL^2(\Om)^{I+1}\to \TS$ denote the orthogonal projection onto the closed subspace
$\TS \coleq \text{span}\,\{\rmD \mathcal{E}(\bfz_\delta),\rmD \mathcal{Q}(\bfz_\delta)\}^{\perp}\subset \rmL^2(\Om)^{I+1}$. 
Then, given $\bfvarphi\in \rmC^\infty(\ol\Om)^{I+1}$, $\tilde \bfz \coleq P_\TS\bfvarphi$ satisfies~\eqref{eq:dir.perp}. 
Insertion into~\eqref{eq:LagrMul.1} yields for $\bfzeta \coleq \rmD H_\delta(\bfz_\delta)$ 
the orthogonality condition $(P_\TS\bfzeta,\bfvarphi)_{\rmL^2}=(\bfzeta,P_\TS\bfvarphi)_{\rmL^2}=0$.
Since $\bfvarphi\in \rmC^\infty(\ol\Om)^{I+1}$ was arbitrary, we infer $P_\TS\bfzeta=0$, meaning that there exist $\eta_\delta,\kappa_\delta\in\mathbb R$ such that 
\begin{align}\label{eq:LagrMultR-delta}
	\rmD \mathcal{H}_\delta(\bfz_\delta)+\eta_\delta \rmD \mathcal{E}(\bfz_\delta)+\kappa_\delta \rmD \mathcal{Q}(\bfz_\delta)=0.
\end{align}
Notice that if $\bfq\equiv0$, then $\rmD \mathcal{Q}(\bfz_\delta)\equiv0$, and the choice 
of $\kappa_\delta$ can be arbitrary in this case. For definiteness, we set $\kappa_\delta=0$ in this case.
Further note that the $(I{+}1)^{\rm st}$ component of the vectorial identity~\eqref{eq:LagrMultR-delta} 
implies that $\eta_\delta=-\pa_uH_\delta>0$, where the second step follows from~\eqref{eq:strict-decr.del}. Let 
\begin{align*}
	\bfxi_\delta \coleq \rmD H_\delta(\bfz_\delta)
	=\big({-}(\eta_\delta\Psi_\delta 
	+\kappa_\delta)\bfq,-\eta_\delta\big)^T,\qquad\text{where }\Psi_{\delta} \coleq \psi_{\bfc_\delta}+\psi_{\text{ext}}.
\end{align*}
Then, by the Fenchel equivalences,
\begin{align}\label{eq:opt.reg.lagrmulr}
	\bfc_\delta=\rmD_IH_\delta^*(\bfxi_\delta),\qquad
	u_\delta=\rmD_{I+1}H_\delta^*(\bfxi_\delta)
\end{align}
where $H_\delta^*$ denotes the Legendre transform of $H_\delta$.

\paragraph{Uniform control.} We recall that $L^{-1}$ denotes the solution operator of 
the Poisson problem~\eqref{eq:poisson-parts.1}, see Lemma~\ref{lemma:poisson.gen}. 

\begin{lemma}\label{l:mon} 
The electrostatic potential $\psi_{\bfc_\delta}=L^{-1}(\bfq{\cdot}\bfc_\delta)$ of the optimizer $\bfz_\delta$ satisfies
$\psi_{\bfc_\delta}\in \rmC(\ol\Om)\cap \HH$, and 
there exists a continuous function $C=C(\eta,\kappa)$, $\eta\in(0,\infty),\kappa\in\mathbb{R}$ 
such that for all $\delta\in(0,1]$, 
\begin{align}\label{eq:psi-del.unif}
\|\psi_{\bfc_\delta}\|_{\rmL^\infty(\Om)}\le C(\eta_\delta,\kappa_\delta).
\end{align}
\end{lemma}
\begin{proof}
The idea of the proof is similar to that in Step~3 of the proof of Proposition~\ref{prop:k-minimizer}.
Recall from the proof of Theorem~\ref{thm:ex-uniq.direct} that $L\psi_{\bfc_\delta}-\bfq\cdot \bfc_\delta=0$ in $\HH^*\cap \rmL^1(\Om)$, and define the convex functional
\begin{align*}
	F_\delta(\psi)& \coleq \frac{1}{\eta_\delta}\mathcal{H}_\delta^*\big({-}\eta_\delta \bfq\, \psi{-}(\eta_\delta\psi_{\text{ext}}
	{+}\kappa_\delta)\bfq,-\eta_\delta\big)
	\\&\phantom{:}=\frac{1}{\eta_\delta}\mathcal{H}^*\big({-}\eta_\delta \bfq\, \psi{-}(\eta_\delta\psi_{\text{ext}}
	{+}\kappa_\delta)\bfq,-\eta_\delta\big)+\frac{1}{\eta_\delta}\frac\delta2\big(\|\eta_\delta \bfq\, \psi{+}(\eta_\delta\psi_{\text{ext}}
	{+}\kappa_\delta)\bfq\|_{\rmL^2}^2+\|\eta_\delta\|_{\rmL^2}^2\big).
\end{align*}
Owing to~\eqref{eq:opt.reg.lagrmulr}, 
$\psi \coleq \psi_{\bfc_\delta}$ satisfies the monotonic equation 
\begin{align}\label{eq:PoissonH1dual.del}
	L\psi+f_\delta(\cdot,\psi)=0\quad\text{in }\HH^*,
\end{align}
where 
\[	f_\delta(\cdot,\psi)=\frac{\delta}{\delta\psi}F_\delta.\]
Under the present hypotheses, classical monotonicity and elliptic regularity, 
cf.~\cite[Chap.~4]{Trol10OCPD} (in particular Theorem~4.10 therein), yield  
the existence of a (unique) continuous weak solution 
$\tilde\psi\in \rmC(\ol\Om)\cap \HH$ of equation~\eqref{eq:PoissonH1dual.del}. 
We assert that $\psi_{\bfc_\delta}$ coincides with $\tilde\psi$. This is a 
consequence of the strong monotonicity of the operator 
underlying~\eqref{eq:PoissonH1dual.del}: taking the difference of the 
equations for $\psi_{\bfc_\delta}$ and $\tilde \psi$, which both hold 
in $\HH^*$, and testing it with $\psi_{\bfc_\delta}-\tilde \psi\in \HH$, 
we find for some $\alpha_0>0$ (cf.~\eqref{eq:Lcoerc})
\begin{align*}
0=\la L(\psi_{\bfc_\delta}{-}\tilde\psi),\psi_{\bfc_\delta} {-} \tilde\psi\ra 
+\la f_\delta(\cdot,\psi_{\bfc_\delta}){-}f_\delta(\cdot,\tilde\psi),
  \psi_{\bfc_\delta}{-}\tilde\psi\ra&\ge \la L(\psi_{\bfc_\delta}{-}\tilde\psi),\psi_{\bfc_\delta}{-}\tilde\psi\ra
  \\&\ge\alpha_0\|\psi_{\bfc_\delta}{-}\tilde\psi\|_{\rmH^1}^2.
\end{align*} 
Hence, $\psi_{\bfc_\delta}\equiv\tilde\psi\in \rmC(\ol\Om)\cap \HH$. 
The bound~\eqref{eq:psi-del.unif} with 
\[C(\eta_\delta,\kappa_\delta)=\tilde C\|f_\delta(\cdot,0)\|_{\rmL^\infty}=\tilde C\|\bfq\cdot \rmD_I\mathcal{H}_\delta^*({-}(\eta_\delta\psi_{\text{ext}}
{+}\kappa_\delta)\bfq,-\eta_\delta)\|_{\rmL^\infty}\]
can be obtained as in~\cite[Thm.~4.8, 4.10]{Trol10OCPD}, see the proof of 
\cite[Thm.~7.6]{Trol10OCPD},  by invoking once more the strong monotonicity 
of the operator $\psi\mapsto L\psi+f_\delta(\cdot,\psi)$. We note that 
the constant $\tilde C\in[1,\infty)$ will in general depend on the 
coercivity constant $\alpha_0>0$ in Lemma~\ref{lemma:poisson.gen}.
\end{proof}

\subsubsection{Convergence to the original problem}

\begin{lemma}[Lim-inf estimate]\label{l:liminf-est} Let  
$\bfmu_\delta=(\bfc_\delta,\hat\nu_\delta),\bfmu=(\bfc,\nu)\in \widehat U$ and suppose that 
$\bfmu_\delta\wk\bfmu$ in $\whh U$ in the sense that $\bfc_\delta\wk \bfc$ in 
$\rmL^1(\Om)^{I}$ and $\nu_\delta\weakstar\nu$ in $M(\ol\Om)$. 
Then, 
\begin{align*}
\mathcal{\whh H}(\bfmu)\le\liminf_{\delta\downarrow0}\mathcal{\whh H}_\delta(\bfmu_\delta).
\end{align*}	
\end{lemma}
\begin{proof}
We first show for all $\bfmu\in\widehat U$ the pointwise convergence
	\begin{align}\label{eq:Moreau.ptw}
\lim_{\delta\downarrow0}\mathcal{\whh H}_\delta(\bfmu)=\mathcal{\whh H}(\bfmu).
	\end{align}
To this end note that $H_\delta(\bfz)\uparrow H(\bfz)$ for all $\bfz\in\mathbb{R}^{I+1}$.
We split $H_\delta(\bfz)=H_\delta^+(\bfz)+H_\delta^-(\bfz)$ into positive and negative part.
Then, for any $\bfz\in \rmL^1(\Om)^{I+1}$, the monotone convergence theorem implies that 
$\lim_{\delta\downarrow0}\int_{\Om}H_\delta^+(\bfz)\dd x= \int_{\Om}H^+(\bfz)\dd x$. 
The negative part is $\delta$-uniformly integrable due to the strict 
sublinearity of $H_\delta(\bfc,u)$ as $u\to\infty$ (cf.~\eqref{eq:h.del.infty}) 
ensuring the convergence $\lim_{\delta\downarrow0}\int_{\Om}H_\delta^-(\bfz)\dd x
= \int_{\Om}H^-(\bfz)\dd x$. In combination, $\lim_{\delta\downarrow0}
\mathcal{H}_\delta(\bfz)=\mathcal{H}(\bfz)$ for all  $\bfz\in \rmL^1(\Om)^{I+1}$, 
and recalling the definition of the extended negative entropies 
$\mathcal{\whh H}$, $\mathcal{\whh H}_\delta$, we deduce~\eqref{eq:Moreau.ptw}.

To show the asserted liminf inequality, let $0<\delta<\ve\le 1$. 
Then $\mathcal{\whh H}_\delta\ge\mathcal{\whh H}_\ve$, and hence
	\begin{align*}
	\liminf_{\delta\downarrow0}\mathcal{\whh H}_\delta(\bfmu_\delta)\ge 
			\limsup_{\ve\downarrow0}\liminf_{\delta\downarrow0}\mathcal{\whh H}_\ve(\bfmu_\delta)
            \ge \limsup_{\ve\downarrow0}\mathcal{\whh H}_\ve(\bfmu)=\mathcal{\whh H}(\bfmu),
	\end{align*}
	where the second inequality uses~\eqref{eq:lsc.weakstar.reg}, and the last step follows from~\eqref{eq:Moreau.ptw}.
\end{proof}

\begin{proposition}[Convergence of optimizers]\label{prop:conv-opt}
	Let $\bfz_\delta=(\bfc_\delta,u_\delta)\in \rmL^1(\Om)^{I+1}$ denote the optimizer of the regularized problem, and suppose that 
	$\inf_{\wt M_{E_0,Q_0}}\mathcal{\wt H}<+\infty$.
	Then, $\bfz_\delta\rightharpoonup \bfz_*$ in $\rmL^1(\Om)^{I+1}$, where $\bfz_*\in \rmL^1_+(\Om)^{I+1}$ denotes the unique optimizer constructed in Theorem~\ref{thm:ex-uniq.direct}.	
	Moreover, $\psi_{\bfc_\delta}\to\psi_{\bfc_*}$ in $\HH$ in the strong sense.
\end{proposition}
\begin{proof}
	Since $\mathcal{E}(\bfz_\delta)=E_0$, there exists $\tilde\psi\in\HH$ and a subsequence $\delta\downarrow0$ such that $\psi_{\bfc_\delta} \wk \tilde\psi$ in $\HH$.
	Furthermore, owing to~\eqref{eq:h.del.infty}, we can extract another subsequence such that, $\bfc_\delta\rightharpoonup \bfc_0$ in $\rmL^1$, 
	$u_{\delta,1}\rightharpoonup u_0$ in $\rmL^1$,
	$u_{\delta,2}\dn x\overset{*}{\rightharpoonup}\nu_0^s$ in $M_+(\ol\Om)$,
for suitable $\bfc_0\in \rmL^1(\Om)^I,u_0\in \rmL^1(\Om)$, and  $\nu_0^s\in M_+(\ol\Om)$, 
and certain $u_{\delta,1}\in \rmL^1$, $u_{\delta,1}\in \rmL^1_+$ with  $u_{\delta,1}+u_{\delta,2}=u_{\delta}$. Let $\nu_0 \coleq u_0\dn x+\nu_0^s$ and $\bfmu_0 \coleq (\bfc_0,\nu_0)$.
Then, $\mathcal{\whh Q}(\bfmu_0)=Q_0$ and $\mathcal{\whh E}(\bfmu_0)\le E_0$.
We want to show that $\bfmu_0\in \wt M_{E_0,Q_0}$ with $\mathcal{\wt H}(\bfmu_0)=\alpha_0$, where $\alpha_0 \coleq \inf_{\wt M_{E_0,Q_0}}\mathcal{\wt H}$. Note that $\alpha_0<\infty$ by hypothesis. 
We now let $\alpha_\delta \coleq \inf_{\whh M_{E_0,Q_0}}\mathcal{\whh H}_\delta =
\mathcal{\whh H}_\delta(\bfz_\delta)$. Since $H_\delta\le H$, it holds that 
$\alpha_\delta\le\alpha_0$. The convergence $(\bfc_\delta,u_\delta\,\dn x) \wk\bfmu_0$ 
in $\whh U$ and Lemma~\ref{l:liminf-est} thus imply that
\begin{align*}
\mathcal{\whh H}(\bfmu_0)\le\liminf_{\delta\downarrow0}\alpha_\delta\le\alpha_0.
\end{align*}
In particular, $\mathcal{\whh H}(\bfmu_0)<\infty$, which shows that 
$\bfmu_0\in \wt U$ and $\mathcal{\whh H}(\bfmu_0)=\mathcal{\wt H}(\bfmu_0)$.
As in the proof of Theorem~\ref{thm:ex-uniq.direct}, we can now use the strict 
decrease of $u\mapsto H(\bfc,u)$ to deduce that the energy gap is trivial, i.e., 
$\mathcal{\whh E}(\bfmu_0)=\mathcal{\wt E}(\bfmu_0)=E_0$. Therefore, $\bfmu_0\in 
\wt M_{E_0,Q_0}$ and $\mathcal{\wt H}(\bfmu_0)=\alpha_0$. Thus, $\bfmu_0$ 
coincides with the unique optimizer $\bfmu_*$ constructed in Theorem~\ref{thm:ex-uniq.direct}, 
and must hence be absolutely continuous, i.e., $\bfmu_0=(\bfc_*,u_*\dn x)$ 
for some $\bfz_* \coleq (\bfc_*,u_*)\in \rmL^1_+(\Om)^{I+1}$. 
The strong convergence $\psi_{\bfc_\delta}\to\psi_{\bfc_*}$ in $\HH$  follows as in the proof of Theorem~\ref{thm:ex-uniq.direct}.
\end{proof}

\subsubsection{Regularity of the global optimizer}

We are now in a position to establish the main result of this subsection on the regularity of the optimizer. 
The analysis below is based on the functional $\mathcal{K}$ (cf.~\eqref{eq:k-functional}) introduced in Section~\ref{su:Derive.whK}, and its generalizations to $\delta\in[0,1]$ given by 
$\mathcal{K}_\delta:(0,\infty) \times \R \times \HH \to\mathbb{R}\cup\{\infty\}$ with 
\begin{align}
	\label{eq:k-functional.delta}
	\calK_\delta(\eta, \kappa, \lambda) \coleq \int_\Omega H^*_\delta\big({-}(\kappa{+}\lambda)\bfq, -
	\eta\big) \dd x + \kappa \,Q_0 + \eta\,E_0 - \bbB(\lambda,\psiext) +
	\frac1{2\eta} \calB(\lambda),
\end{align}
where $H_\delta^*(\bfxi)=H^*(\bfxi)+\frac{\delta}{2}|\bfxi|^2$ for $\delta\in[0,1]$.
Notice that the functionals $\mathcal{K}_\delta$ are convex for all $\delta\in[0,1]$ (cf.\ Section~\ref{su:exist-minimizers} for details) and that $\mathcal{K}=\mathcal{K}_0$.

From Lemma~\ref{l:mon}, we recall that $\psi_{\bfc_\delta} \in \rmC(\ol\Om)$, and hence, 
$\lambda_\delta \coleq \eta_\delta(\psi_{\bfc_\delta}{+}\psie)\in \rmC(\ol\Om)$. The 
stationarity condition for any (regular) critical point of $\mathcal{K}_\delta$ is satisfied by $(\eta_\delta,\kappa_\delta,\lambda_\delta)$ due to~\eqref{eq:LagrMultR-delta} and the Poisson equation. Hence, $(\eta_\delta,\kappa_\delta,\lambda_\delta)$ is a minimizer. 
\medskip

We are now ready to finish the proof of Theorem \ref{thm:reg-globalopt.v1}, which states that the maximizer $\bfz_*=(\bfc_*,u_*)$ is continuous and uniformly positive.\smallskip 

\noindent 
\begin{proof}[Proof of Theorem~\ref{thm:reg-globalopt.v1}]
	Since $H_\delta^*\ge H^*$ and thus $\mathcal{K}_\delta\ge\mathcal{K}$, the coercivity~\eqref{eq:coerc.K} of $\mathcal{K}$
	  implies that of $\mathcal{K}_\delta$ uniformly in $\delta\in(0,1]$.
	Hence, the sequence  $(\eta_\delta,\kappa_\delta,\lambda_\delta)_\delta$ of minimizers of $\mathcal{K}_\delta$ is uniformly bounded in $(0,\infty) \times \R \times \HH$, and $(\eta_\delta)_\delta$ is bounded away from zero. 
	Thus, along a subsequence, $\eta_\delta\to\eta_0$ for some $\eta_0\in(0,\infty)$, $\kappa_\delta\to\kappa_0$ for some $\kappa_0\in \mathbb{R}$, and $\lambda_\delta\wk\lambda_0$ in  $\HH$ and $\lambda_\delta\to\lambda_0$ a.e.\ in $\Om$
	for some $\lambda_0\in \HH$.
	It follows that $\psi_{\bfc_\delta}\wk\psi_0$ in $\HH$, where $\psi_0 \coleq \eta_0^{-1}\lambda_0-\psie$. 
	Owing to Lemma~\ref{l:mon}, we infer the regularity $\psi_0,\lambda_0\in \rmL^\infty(\Om)$.
	Consequently, 
	  $$\bfxi_0 \coleq \big({-}\eta_0\;\!(\psi_0{+}\psie)\bfq-\kappa_0 \bfq,-\eta_0 \big)^T\in \rmL^\infty(\ol\Om,\mathsf C)$$ 
	for a closed convex set $\mathsf C \Subset \text{dom}(\rmD H^*)=\mathbb{R}^I\times(-\infty,0)$,
	and we have the pointwise convergence $\bfz_\delta= \rmD H_\delta^*(\bfxi_\delta)\to \rmD H^*(\bfxi_0)$. 
	Thanks to formula~\eqref{eq:LagrMultR-delta} and the boundedness of $(\eta_\delta,1/\eta_\delta)$, $(\kappa_\delta)$, and of $(\psi_\delta)\subset \rmC(\ol\Om)$, the sequence $(\bfz_\delta)\subset \rmC(\ol\Om)^{I+1}$ must be bounded.
	We therefore conclude that  $\bfz_0 \coleq  \rmD H^*(\bfxi_0)\in \rmL^\infty(\ol\Om; \rmD H^*(\mathsf C))$. Since $\rmD H^*(\mathsf C)\Subset(0,\infty)^{I+1}$, this implies the control of the values $\bfz_0(\ol\Om)\Subset(0,\infty)^{I+1}$. 
    Further note that the definition of $\bfz_0$ implies that
	\begin{align*}
	\rmD\mathcal{H}(\bfz_0)+\eta_0 \rmD\mathcal{E}(\bfz_0)+\kappa_0 \rmD\mathcal{Q}(\bfz_0)=0.
	\end{align*}
	The continuity of $\bfz_0$ is deduced by applying~\cite[Thm.~4.8, 4.9]{Trol10OCPD} to the Poisson equation for $\psi_0$.
	
	It is easy to see that $\mathcal{K}_\delta$ converges to $\mathcal{K}$ in the sense of $\Gamma$-convergence on $(0,\infty) \times \R \times \HH$, using 	
	the weak lower semicontinuity of $\mathcal{K}$, the inequality $\mathcal{K}_\delta\ge\mathcal{K}$, and the pointwise convergence $\mathcal{K}_\delta\to \mathcal{K}$ on $(0,\infty) \times \R \times \HH$.
	Hence, the limit $(\eta_0,\kappa_0,\lambda_0)$ is a minimizer of $\mathcal{K}$. 
	It follows that $\pa_\eta\mathcal{K}_{|(\eta_0,\kappa_0,\lambda_0)}=0$, meaning that $\mathcal{E}(\bfz_0)=E_0$. At the same time, $\bfz_0\in \rmL^1_+(\Om)^{I+1}$ and $\mathcal{Q}(\bfz_0)=Q_0$. Thus, $\bfz_0\in M_{E_0,Q_0}$. 
	Furthermore, the regularity of $\bfz_0$ implies that $\mathcal{S}(\bfz_0)=-\mathcal{H}(\bfz_0)>-\infty$. 
	Combining the observations above with 
    Proposition~\ref{prop:conv-opt}, we conclude that $\bfz_0=\bfz_*\in M_{E_0,Q_0}$ is the 
	unique solution of the optimization problem $\mathcal{S}(\bfz_*)=\sup_{M_{E_0,Q_0}}\mathcal{S}$.
\end{proof}

\subsection{Global optimality of critical points}

We first establish the equivalence between constrained critical 
points of $\calS$ and constrained local equilibrium solutions to 
\eqref{eq:general-system}. The proof of Theorem \ref{th:GlobEquilibrium} 
on the existence of a unique global equilibrium is given subsequently. 

\bigskip  
\noindent
\begin{proof}[Proof of Proposition~\ref{prop:crit=equil}]\label{page:proofProp}
 	The proof is based on an adaptation of~\cite[Chap.~43]{Zeidler_1985_variational}. 
    To simplify notation, we only present the argument in the case $\HH = \rmH^1(\Omega)$.
 	First, let us note that $\FF\subset \rmC(\ol\Omega)^{I+1}$ is open and that
 	$\calS,\calE,\calQ:\FF\subset \rmC(\ol\Omega)^{I+1}\to\mathbb R$ are continuously Fr\'echet differentiable.
 	The formulas~\eqref{eq:derivatives-energy-charge} show that the Fr\'echet derivatives $\rmD\calE(\bfz):\rmC(\ol\Omega)^{I+1}\to\mathbb R$, 
 	$\rmD\calQ(\bfz):\rmC(\ol\Omega)^{I+1}\to\mathbb R$ are surjective 
 	for all $\bfz\in\FF$ provided $\bfq\not\equiv0$ 
 	(which we assume throughout the proof). 
 	If $\bfq \equiv 0$, all arguments of the subsequent proof hold if the 
 	charge constraint $\calQ(\bfz) = Q_0$ is skipped. The statement of the proposition 
 	(including the charge constraint) then trivially follows.
 	In terms of 
 	$$\mathcal G(\bfz)\coleq \big(\calE(\bfz){-}E_0,\calQ(\bfz){-}Q_0\big):\FF\to\mathbb R^2,$$ 
 	the constraint becomes $\mathcal G=0$.
 	It is easy to verify that $\mathcal G$ is a submersion in the sense of~\cite[Def.~43.15]{Zeidler_1985_variational}. 
 	Therefore, the constraint set $M\coleq \{\bfz\in\FF:\mathcal G(\bfz)=0\}$ forms a $\rmC^1$ manifold in $\rmC(\ol\Omega)^{I+1}$, and the basic hypotheses of~\cite[Thm.~43.D]{Zeidler_1985_variational} are fulfilled. 
 	We now prove the two directions separately. The first is a direct application of~\cite[Thm.~43.D~(1)]{Zeidler_1985_variational}, while the second requires several adjustments due to discrepancies in the functional setting linked to the singularity of $\rmD S(\bfc,u)$ as $\min_ic_i\to0$.\smallskip
 	
 	Re~1: 
 	Let $\bfz_*\in\FF$ be a local equilibrium solution of $\calS$ with respect to the side condition $\mathcal G=0$. Then, invoking~\cite[Thm.~43.D~(1)]{Zeidler_1985_variational}, there exist $(\eta,\kappa)\in\mathbb R^2$ such that $\rmD\calS(\bfz_*)=\eta\rmD\calE(\bfz_*)+\kappa \rmD\calQ(\bfz_*)$. Hence, $\bfz_*$ is a constrained critical point in the sense of Definition~\ref{def:equilibrium}.
 	\smallskip
 	
 	Re~2: For this direction, we will use the fact that, under the present hypotheses on $\calS$, the functionals $\calS,\calE,\calQ:\FF\subset \rmC(\ol\Omega)^{I+1}\to\mathbb R$ are twice continuously Fr\'echet differentiable.
 	Let $\bfz_*$ be a critical point of $\calS$ with respect to the constraints $\calE=E_0,\calQ=Q_0$, i.e., a
 	solution to 
 	the equation $\rmD\calS(\bfz_*)=\eta \rmD\calE(\bfz_*)+\kappa \rmD\calQ(\bfz_*)$ for suitable $(\eta,\kappa)\in \mathbb R^2$. For later usage, we note that $\eta$ must be positive since the equation implies that $\eta = \rmD_u S (\bfz_\ast) >0$ as a consequence of~\eqref{eq:derivatives-energy-charge} and Hypotheses~\ref{hypo:entropy}. 
 	
 	We wish to show that $\calS(\bfz_*)> \calS(\bfz)$ for all $\bfz$, $\bfz\neq\bfz_*$, in a $\rmC^0$ neighborhood of $\bfz_*$ within $M$. 
 	Since $S\in \rmC^2((0,\infty)^{I+1})$, the Taylor theorem ensures that, as $\bfz\to \bfz_0\in (0,\infty)^{I+1}$, 
 	\[S(\bfz)-S(\bfz_0)=\rmD S(\bfz_0)(\bfz-\bfz_0)+\frac{1}{2}(\bfz-\bfz_0)^T\rmD^2S(\bfz_0)(\bfz-\bfz_0)+o_{|\bfz-\bfz_0|}(1)|\bfz-\bfz_0|^2.\]
 	For $\bfz\in \rmC(\ol\Omega,(0,\infty)^{I+1})$, we thus infer for all $x\in \Omega$,
 	\begin{align*}
    S(\bfz(x))-S(\bfz_*(x))=\rmD S(\bfz_*)(\bfz-\bfz_*) &+\frac{1}{2}(\bfz-\bfz_*)^T\rmD^2S(\bfz_*)(\bfz-\bfz_*) \\ &+o_{\|\bfz-\bfz_*\|_{\rmC(\ol\Omega)}}(1)|\bfz(x)-\bfz_*(x)|^2.
    \end{align*}
 	Here, $o$ can be chosen uniformly in $\bfz, \bfz_*$ if $\bfz(\ol\Omega), \bfz_*(\ol\Omega)\subset K$ for some fixed  $K\Subset(0,\infty)^{I+1}$.  Integration over $x\in\Omega$ thus yields
 	\begin{align}
 		\calS(\bfz)-\calS(\bfz_*)=\rmD\calS(\bfz_*)(\bfz-\bfz_*) &+\frac{1}{2}\big\langle \bfz-\bfz_*, \rmD^2\calS(\bfz_*)(\bfz-\bfz_*)\big\rangle \nonumber \\ 
        &+f(r)\|\bfz-\bfz_*\|^2_{\rmL^2(\Omega)} \label{eq:TaylorL2}
 	\end{align}
 	for all $\bfz\in\FF$ with $\bfz(\ol\Omega)\subset K$ and $\|\bfz-\bfz_*\|_{\rmC(\ol\Omega)}\le r$, where $f\in \rmC([0,\infty))$ satisfies $\lim_{r\to0} f(r)=0$. 
    
 	As in the proof of~\cite[Thm.~43.D~(2)]{Zeidler_1985_variational}, 
    we consider the functional
 	\[\mathcal F(\bfz) \coleq \calS(\bfz)-\eta\calE(\bfz)-\kappa\calQ(\bfz), \quad \bfz\in\FF.\]
 	By hypothesis, $\rmD\mathcal{F}(\bfz_*)\bfh=0$ for all $\bfh\in \rmC(\ol\Omega)^{I+1}$, where $M=\{\bfz\in\FF:\mathcal G(\bfz)=0\}$.
 	To proceed, we note that $\mathcal G$ can be extended naturally to a Fr\'echet differentiable map $\mathcal{\widetilde G}: \rmL^2(\Omega)^{I+1}\to\mathbb R^2$ (the solvability theory for the Poisson equation determining $\psi_\bfc$ does not require the components of $\bfc$ to be positive). Hence, the set $\widetilde M\coleq \{\bfz \in \rmL^2(\Omega)^{I+1}:\mathcal{\widetilde G}(\bfz)=0\}$ is a $\rmC^1$ manifold  in $\rmL^2(\Omega)^{I+1}$.
 	Thus, there exist open neighborhoods $\widetilde V\subset  T_{\bfz_*}\widetilde M$ of zero and 
 	$\widetilde N\subset \widetilde M$ of $\bfz_*$,
 	and a (homeomorphic) $\rmC^1$ parametrization $\bfvarphi:\widetilde V\subset T_{\bfz_*}\widetilde M\to \widetilde N$ satisfying   (cf.~\cite[Thm.~43.C]{Zeidler_1985_variational})
 	\[\bfvarphi(\bfh)=\bfz_*+\bfh+o(\|\bfh\|_{\rmL^2(\Omega)})\quad\text{ as }\bfh\to0\text{ in }\widetilde V.\]
 	Since the topology of $\rmL^2(\Omega)$ is coarser than that of $\rmC(\ol\Omega)$, there exists a $\rmC^0$ neighborhood $N\subset M$ of $\bfz_*$ satisfying $N\subset \widetilde N$. 
 	Upon extension of $\rmD\calS(\bfz_*)$ to a linear continuous functional on $\rmL^2(\Omega)^{I+1}$, the identity $\rmD\calS(\bfz_*)=\eta\rmD\calE(\bfz_*)+\kappa \rmD\calQ(\bfz_*)$ also holds as an equality in $(\rmL^2(\Omega)^{I+1})^*$. 
 	Furthermore, we note that the bilinear forms $\rmD^2\calS(\bfz_*),\rmD^2\calE(\bfz_*),\rmD^2\calQ(\bfz_*):\rmC(\ol\Omega)^{I+1}\times \rmC(\ol\Omega)^{I+1}\to\mathbb R$ can be continuously extended to $\rmL^2(\Omega)^{I+1}\times \rmL^2(\Omega)^{I+1}$. 
 	Using~\eqref{eq:TaylorL2} and the Taylor theorem for the smooth functional $\eta\calE+\kappa\calQ:\rmL^2(\Omega)^{I+1}\to\mathbb R$, 
 	we infer that for all $\bfh\in \bfvarphi^{-1}(N)\subset T_{z_*}\widetilde M$, 
 	\[\mathcal F(\bfvarphi(\bfh))-\mathcal F(\bfz_*)\le \big\langle \bfh,\rmD^2\mathcal F(\bfz_*)\bfh\big\rangle+o(1)\|\bfh\|_{\rmL^2}^2,\]
 	where $o(1)\to0$ as $\bfz\to \bfz_*$ in $\rmC(\ol\Omega)$.
 	Note that the left-hand side equals $\mathcal S(\bfvarphi(\bfh))-\mathcal S(\bfz_*)$ 
    since $\bfvarphi(\bfh), \bfz_* \in M$.
 	
 	To conclude, we assert that the positivity of $\eta>0$, the strong concavity of $S$, and the structure of the nonlinearity in $\calE$ imply the bound
 	$\langle \bfh,\rmD^2\mathcal F(\bfz_*)\bfh\rangle\le -c\|\bfh\|_{\rmL^2}^2$ for some $c>0$. 
 	It holds that $\rmD^2\mathcal F(\bfz_*)=\rmD^2\mathcal S(\bfz_*)-\eta\rmD^2\mathcal E(\bfz_*)$. Thus, in order to establish the asserted inequality, it suffices to show that $\langle \bfh,\rmD^2\mathcal E(\bfz_*)\bfh\rangle\ge0$ for all  $\bfh\in \rmL^2(\Omega)^{I+1}$.
But this is a consequence of the positivity of the quadratic form 
$\calB$. Indeed, for all
$\bfh=(\bfh',h_{I+1})\in \rmL^2(\Omega)^{I+1}$ and all $\bfz=(\bfc,u)\in \rmL^2(\Omega)^{I+1}$, it holds that
\[
	\big\langle \bfh,\rmD^2\calE(\bfc, u) \bfh \big\rangle = \big\langle \bfh',\rmD_{\bfc\bfc}\calE(\bfc, u) \bfh' \big\rangle = \frac12 \big\langle \bfq \cdot \bfh',L^{-1}(\bfq \cdot \bfh')\big\rangle_{\HH}=\frac12\calB(\psi_{\bfh'})\ge0.
\]
We conclude that $\calS(\bfz_*)> \calS(\bfz)$ for all $\bfz\in N'$ 
with $\bfz\neq\bfz_*$ for a sufficiently small $\rmC^0$ neighborhood $N'\subseteq N\subset M$ of $\bfz_*$.
This finishes the proof of Proposition~\ref{prop:crit=equil}.
\end{proof}

\bigskip
\noindent
\begin{proof}[Proof of Theorem \ref{th:GlobEquilibrium}] 
 		We argue by contradiction. Assume that $(\ol\bfc,\ol u)$
 		satisfies 
 		\[
 		\calS(\ol\bfc,\ol u)\geq \calS(\bfc^*,u^*), \quad  \calE(\ol\bfc,\ol
 		u)=\calE(\bfc^*,u^*)=E_0  \ \text{ and } \ \calQ(\ol\bfc, \ol u)=\calQ(\bfc^*, u^*)=Q_0. 
 		\]
 		For $\vartheta\in [0,1]$, we define a straight connection and a continuous curve 
 		\[
 		\big(\bfc_\vartheta,u_\vartheta\big) \coleq (1{-}\vartheta) \big(\bfc^*,u^*\big) + \vartheta 
 		\big(\ol\bfc,\ol u\big) \quad \text{and } \  
 		\big(\wt\bfc_\vartheta,\wt u_\vartheta\big) \coleq \big(\bfc_\vartheta,u_\vartheta\big) 
 		+ \big(0, \vartheta(1{-}\vartheta) \wt E), 
 		\]
 		where $\wt E$ is given by $\wt E(x) \coleq 
            \frac\eps2 |\nabla \psi_{\bfc^*}(x)  {-} \nabla \psi_{\ol\bfc}(x)|^2\geq 0$. 
 		Using the quadratic nature of $\calE$ and
 		the linearity of $\calQ$, we easily find 
 		\[
 		\calE(\wt\bfc_\vartheta,\wt u_\vartheta)=E_0 \ \text{ and } \ 
 		\calQ(\wt\bfc_\vartheta, \wt u_\vartheta)=Q_0 \quad
 		\text{for all } \vartheta\in [0,1],
 		\]
 		which means that $(\wt \bfc_\vartheta, \wt u_\vartheta)$ are admissible competitors. 
 		
 		By the assumption $\calS(\ol\bfc,\ol u)=\calS(\bfc_1,u_1)\geq
 		\calS(\bfc^*,u^*)=\calS(\bfc_0, u_0)$ and the concavity of $\calS$, we have 
 		$\calS(\bfc_\vartheta,u_\vartheta)\geq  \calS(\bfc^*,u^*)$ for all $\vartheta\in
 		[0,1]$. Now using $\wt E(x)\geq 0$ and $\rmD_u S(\bfc,u)>0$, we obtain 
 		\[
 		\calS(\wt\bfc_\vartheta,\wt u_\vartheta) \geq  \calS(\bfc^*,u^*) \quad\text{for all } 
 		\vartheta \in {(0,1)}.
 		\]
 		Since $\vartheta$ can be taken arbitrarily small, this contradicts the fact that
 		$(\bfc^*,u^*)$ is a strict local maximizer, which was established in
 		Proposition \ref{prop:crit=equil}. 
 		
 		Thus, the assumption $\calS(\ol\bfc,\ol u)\geq \calS(\bfc^*,u^*)$ produced a
 		contradiction, and we conclude that $  (\ol\bfc,\ol u)$ is the unique global
 		maximizer under the given constraint. 
 	\end{proof}

\section{Proofs from Section \ref{se:electro-energy}}
\label{se:proofs-electro-energy}

We first establish the properties \ref{Fact1} (density of $\rmC(\ol\Om)$) and \ref{Fact2} (unboundedness of $\rho\mapsto \int_\Omega \rho\dd x$ if $\calH(\Gamma_\rmD)>0$) for $\HH^*$ as stated in Lemma \ref{lemma:l1-density}.\medskip 

\noindent
\begin{proof}[Proof of Lemma \ref{lemma:l1-density}]
\textit{Proof of \ref{Fact1}:} 
We start with some preliminaries. We identify any function $\rho\in \rmC(\ol\Om)$ 
with an element in $\rmH^1(\Om)^*$ resp.\ in $\HH^*$ via 
$Y\ni\psi\mapsto\int_\Om\rho\psi\,\dd x$, where $Y=\rmH^1(\Om)$ resp.\ $Y=\HH$.
Furthermore, since $\HH\subset \rmH^1(\Om)$ is a closed subspace, 
there exists a unique orthogonal projection $P: \rmH^1(\Om)\to \HH$. 
Let $P^*:\HH^*\to \rmH^1(\Om)^*$ denote its adjoint. 
In order to prove that $\rmC(\ol\Om) \subset \HH^*$ is dense,
we use the fact that $\rmC(\ol\Om)\subset \rmH^1(\Om)^*$ is dense, 
which can be shown by classical mollification arguments thanks to 
the Lipschitz continuity of $\pa\Om$. Let $\rho\in \HH^*$. 	
Then, $f\coleq P^*\rho\in \rmH^1(\Om)^*$. Thus, there exist 
$(f_j)_j\subset \rmC(\ol\Om)\subset \rmH^1(\Om)^*$ with $f_j\to f$ in $\rmH^1(\Om)^*$. 
The restrictions $\rho_j\coleq {f_j}_{|\HH}\in \rmC(\ol\Om) \subset \HH^*$ 
satisfy $\rho_j\to\rho$ in $\HH^*$.

\textit{Proof of \ref{Fact2}:} We specify the construction indicated in
\eqref{eq:def-rhon-an}. As $\Omega$ has a \mbox{Lipschitz} boundary and
$\calH^{d-1}(\Gamma_\rmD)>0$, we can assume that -- after a suitable rotation
and translation of the coordinate system -- there are some $\delta>0$ and a
Lipschitz function $h: B_\delta(0)\subset \R^{d-1}\to \R$ with $h(0)=0$ such
that  
\[
\begin{aligned}
S_\delta &\coleq \bigset{(y,h(y))}{ y\in B_\delta(0)} \subset \pl\Omega, \qquad
\wt\gamma_\delta\coleq \calH^{d-1}( S_\delta \cap \Gamma_\rmD) >0, 
\\   
A_1 &\coleq \bigset{(y,h(y){+}r)}{ y\in B_\delta(0), \ r\in (0,\delta)} \subset
\Omega.
\end{aligned}
\]
We set $\Sigma_\delta \coleq \bigset{y \in B_\delta(0)}{(y,h(y))\in \Gamma_\rmD } 
\subset \R^{d-1}$ and find
\[
\gamma_\delta\coleq \calL^{d-1}( \Sigma_\delta) \geq \int_{\Sigma_\delta}
\frac{\sqrt{1{+}|\nabla h(y)|^2}}{1+\|\nabla h\|_{\rmL^\infty}} \dd y 
= \frac{\wt\gamma_\delta} {1+\|\nabla h\|_{\rmL^\infty}} \gneqq 0. 
\]
We now define the $\rmL^\infty$ functions 
\[
\rho_n \coleq \frac{n}{\delta\gamma_\delta} \,\bm1_{A_n} \quad \text{with }
A_n\coleq \Big\{x=(y,h(y){+}r) \in \Omega \; \Big| \; y\in \Sigma_\delta,\ r\in \Big( 0,\frac{\delta}{n} \Big) \Big\}. 
\] 
Clearly, we have $\calL^d(A_n)=\calL^{d-1}(\Sigma_\delta)\,\delta / n =
\gamma_\delta \delta/n$, which implies $\int_\Omega \rho_n \dd x =1$. 

To estimate $\langle \rho_n, \psi \rangle_{\HH^\ast \ti \HH}$, we use that $\psi
\in \HH$ satisfies $\psi(y,h(y))=0$ for
$\calL^{d-1}$-a.e.\ $y \in \Sigma_\delta$. Hence, for $r\in (0,\delta)$, we have 
\begin{equation}
  \label{eq:sqrt.H1}
  \big|\psi(y,h(y){+}r)\big|^2\leq \Big(\int_0^r \big|\pl_{y_d}\psi(y,h(y){+}s)\big| 
  \dd s\Big)^2 \leq r \int_0^\delta \big|\pl_{y_d}\psi(y,h(y){+}s)\big|^2\dd s  .
\end{equation}
With this, we can estimate as follows:
\begin{align*}
    \big| \langle \rho_n, \psi \rangle_{\HH} 
    \big|^2 
    &= \Big| \int_\Omega \rho_n \psi \dd x \Big|^2 = \Big|
    \frac{n}{\delta\gamma_\delta} \int_{A_n} \psi \dd x \Big|^2 \leq 
     \frac{n}{\delta\gamma_\delta} \int_{A_n} \psi(x)^2 \dd x
\\
& = \frac{n}{\delta\gamma_\delta} \int_{\Sigma_\delta} \int_0^{\delta/n}
\psi(y,h(y){+}r)^2 \dd r \dd y 
\\
&\!\overset{\text{\eqref{eq:sqrt.H1}}}\leq 
\frac{n}{\delta\gamma_\delta} \int_{\Sigma_\delta} \int_0^{\delta/n} r 
\int_0^\delta\big|\pl_{y_d}\psi(y,h(y){+}s)\big|^2\dd s  \dd r \dd y
\\
&    = \frac{\delta}{2\gamma_\delta\, n}  \int_{A_1} \big|\pl_{y_d}\psi(x) \big|^2 \dd
x  \leq \frac{\delta}{2\gamma_\delta\, n} \| \nabla \psi \|_{\rmL^2}^2 
\leq \frac{\delta}{2\gamma_\delta\, n} \| \psi\|_\HH^2 .
\end{align*}
This shows $\| \rho_n\|_{\HH^*} \leq C/\sqrt{n} \,\to \, 0$, 
and Lemma \ref{lemma:l1-density} is established.
\end{proof}
\medskip 

The final result provides the characterization of the lower energy bound $\mathrm V(Q_0,\psie)$.\medskip

\noindent
\begin{proof}[Proof of Proposition \ref{pr:MinElectroEner}]
We first treat the case (C) and discuss the necessary adjustments for (A) and
(B) afterwards.

In the case (C), we have $1_\Omega \in \HH = \rmH^1(\Omega)$ and
$\rho \mapsto \int_\Omega \rho \dd x =\langle \rho,1\rangle_\HH$ is well-defined
on all of $\HH^*$. Since $\rmE$ and the constraint are continuous in $\HH^*$, the
density in \ref{Fact1} allows us to replace the infimum for $\rmV(Q_0,\psiext)$ by
the infimum over all $\rho \in \HH^*$ with $\langle \rho,1\rangle_\HH=Q_0$.

Since $\psi_\rho$ is the unique minimizer of $\psi \mapsto \frac12\calB(\psi) -
\langle \rho, \psi \rangle_\HH$ with minimum value $-\rmE(\rho,0)$, we see
that 
\[
 \rmE(\rho,\psiext) = \frac12 \calB(\psiext) + \langle \rho, \psiext \rangle_\HH 
  + \sup\Big\{ \langle \rho,\psi\rangle_\HH - \frac12\calB(\psi) \; \Big| \; \psi \in \HH \Big\} .
\]

To include the total-charge constraint on $\rho$, we define the Lagrange
function $\calL:\HH^*\ti \HH\ti \R\to \R$ via 
\[
\calL(\rho,\psi,\kappa) \coleq \frac12 \calB(\psiext) - \frac12\calB(\psi) 
+ \langle \rho,  \psiext {+} \psi {-}\kappa \rangle_\HH  +  Q_0 \kappa,
\]
where the scalar $\kappa$ is the Lagrange parameter for the charge constraint.
Note that $\calL$ is a quadratic functional and we have
\[
  \rmV(Q_0,\psiext) = \inf_{\rho \text{\rule{0em}{1em}}\in \HH^* } \Big(
  \sup_{\psi\in \HH,\: \kappa\in \R}   \calL(\rho,\psi,\kappa) \Big).
\]
Clearly, $\rho \mapsto \calL(\rho,\psi,\kappa)$ is affine, and hence convex,
$(\psi,\kappa) \mapsto \calL(\rho,\psi,\kappa)$ is concave. Thus, we have a
classical saddle-point problem with a continuous quadratic Lagrange function.
By the theory in \cite[Chap.~VI]{EkeTem76CAVP}, it is enough to find a saddle
point, which can be done by interchanging the inf-sup into a sup-inf. By
linearity in $\rho$, we have 
\[
\inf_{\rho\in \HH^*}  \calL(\rho,\phi,\kappa)  = 
\begin{cases} 
\frac12 \calB(\psiext) - \frac12\calB(\phi) 
+  Q_0 \kappa & \text{for }\psiext+\phi \equiv \kappa ,\\
 -\infty  &\text{for } \psiext+\phi \not\equiv \kappa. 
\end{cases}
\]
The subsequent supremum over $ \phi, \kappa $ is then attained at $ \phi_*=
\kappa- \psiext$ and leads to 
\begin{align}
\sup_{\psi\in \HH} \inf_{\rho\in \HH^*}  \calL(\rho,\psi,\kappa) &= 
\frac12\calB(\psiext) - \frac12 \calB(\kappa{-}\psiext) + \kappa Q_0 \\
&= \kappa Q_0 + \kappa \bbB(1,\psiext) - \frac{\kappa^2}2 \calB(1) ,
\end{align}
where we used the bilinear form $\bbB$ for $\calB(\phi)=\bbB(\phi,\phi)$. Note
that $\calB(1)=\bbB(1,1)=\int_{\pl\Omega} \omega\dd a >0$ and
$\bbB(1,\psiext) = \int_{\pl\Omega} \omega \psiext \dd a$. 

Maximizing with respect to $\kappa$ gives the maximizer $\kappa_*= \big(Q_0 {+}
\bbB(1,\psiext) \big) /\calB(1)$, which provides the desired saddle point
$(\rho_*,\psi_*,\kappa_*)$ where $\psi_*=\kappa_*{-}\psiext$ and 
$\rho_*=L\psi_*=L(\kappa{-}\psiext)\in \HH^*$. Moreover, we have 
\[
\rmV(Q_0,\psiext)= \frac12 \calB( \psi_* {-} \psiext ) 
  = \frac12 \calB(\kappa_*)  = \frac{\kappa_*^2}2 \calB(1)
  = \frac{\kappa_*^2}2 \int_{\pl\Omega} \omega \dd a.
\]
In the case $\psiext=M(D,\psi_R)=L^{-1}N(D,\grmR)$, we use that
$L(\kappa_*\bm1_\Omega)  =N(0,\kappa_*\omega)$ to obtain
$\rho_*=L(\kappa_*{-}\psiext)= N({-}D,\kappa_*\omega{-}\grmR)$. 
With this, both assertions concerning (C) are established.\medskip

To show (B), we use
$ \HH_\rmN \coleq \bigset{\psi \in \rmH^1(\Omega)}{ \int_\Omega \psi\dd x =0}$ and
restrict to the case $Q_0 = 0$. 
For all $\psiext \in \HH_\rmN$, we easily see that $\psi_* \coleq -\psiext$ and 
$\rho_* \coleq L\psi_*= -L\psiext$ provide the unique minimizer of $\rmE$ 
defined in \eqref{eq:E.rho.psiext}, which leads to $\rmV(0,\psiext)=0$. 
In fact, since $\HH_\rmN = \mathrm{span}\{1\}^{\perp_{\rmH^1}}$, 
we have $\HH_\rmN^\ast \cong \rmH^1(\Omega)^\ast / \HH_\rmN^{\perp_{\rmH^1}} 
= \rmH^1(\Omega)^\ast / \mathrm{span}\{1\}$; see, e.g., \cite{Lang_1993_Functional}. 
Moreover, for every $\rho + \mathrm{span}\{1\} \in \HH_\rmN^\ast$ with 
$\rho \in \rmH^1(\Omega)^\ast$, there exists some $r \in \bbR$ such that 
$\langle \rho + r, 1 \rangle_{\rmH^1} = 0$. Hence, the charge constraint is 
not present at all on $\HH_\rmN^\ast$ and both parts of (B) are shown.

For the case (A), we can use \ref{Fact1} and \ref{Fact2}, which essentially means that we can adjust the
total charge $Q_0$ with arbitrarily small cost. This means that for taking the
infimum, we can drop the charge constraint completely and obtain a minimizer
with a wrong total charge $Q_*$. 
Thus, as in case (B), we minimize $\rmE(\cdot,\psiext)$ by finding the minimizer
$\psi_*$ of $\psi_\rho \mapsto \frac12\calB(\psi_\rho{+}\psiext)$ and then
letting $\rho_* \coleq L\psi_*$. The difference is that we now use $\HH_\rmD \coleq \bigset{
  \psi\in \rmH^1(\Omega) }{ \psi|_{\Gamma_\rmD} = 0}$. Clearly, we find $\psi_*=
- \psiext$ and $\rho_*=-L\psiext$ as well as the minimal energy
$\rmE(\rho_*,\psiext)=0$. Thus, the result for part (A) is established as well, and Proposition \ref{pr:MinElectroEner} is proven.
\end{proof}

\appendix
\section{An EERDS with reversible mass-action reactions}
\label{app:TimeDepModel}
	In this short appendix, we give an example for a class of (non-stationary) 
    electro--energy--reaction--diffusion systems with reversible mass-action reactions. 
    The discussion focuses on modeling aspects and the relation between critical points 
    of the total entropy and stationary states of the EERDS. We stress that the 
    rigorous analysis of the stationary problem presented in this article is completely 
    independent of the following considerations. 
	
	We start with reviewing equivalent ways of 
	expressing the evolution law \eqref{eq:general-system} following the 
	notation in \cite{Miel11GSRD} and \cite{MiPeSt21EDPC}. 	
	Without going too much into details, we recall that the dual entropy-production potential
	$
		\calP^\ast(\bfZ; \bfW) = \sup_{\bfV \in Y} 
		\big\{ \langle \bfW, \bfV \rangle - \calP(\bfZ; \bfV) \big\}
	$
	is the Legendre--Fenchel transformation of the non-negative and convex entropy-production potential 
    $\calP(\bfZ; \bfV)$ satisfying $\calP(\bfZ; 0) = 0$. 
	A fundamental relation in convex analysis are the Fenchel equivalences 
	\begin{align*}
		\bfW_0 \in \pl_\bfV \calP (\bfZ; \bfV_0) \ \Leftrightarrow \
		\calP(\bfZ; \bfV_0) + \calP^\ast(\bfZ; \bfW_0) = \langle \bfW_0, \bfV_0 \rangle 
		\ \Leftrightarrow \ \bfV_0 \in \pl_\bfW \calP^\ast (\bfZ; \bfW_0). 
	\end{align*}
	Here, we use the symbol $\pl$ to denote the possibly set-valued subdifferential 
	of a convex function. 
	The relations above allow us to express the dynamics of our EERDS in three 
	equivalent ways: 
	\begin{subequations}
		\label{eq:gradient-flow-forms}
		\begin{gather}
			0 \in \pl_\bfV \calP(\bfZ; \dot \bfZ) - \rmD \calS (\bfZ), 
			\label{eq:gradient-flow-form-force} \\ 
			\calP(\bfZ; \dot \bfZ) + \calP^\ast(\bfZ; \rmD \calS (\bfZ)) 
			= \langle \rmD \calS (\bfZ), \dot \bfZ \rangle, 
			\label{eq:gradient-flow-form-power} \\ 
			\dot \bfZ \in \pl_\bfW \calP^\ast (\bfZ; \rmD \calS (\bfZ)). 
			\label{eq:gradient-flow-form-rate}
		\end{gather}
    \end{subequations}
		The first relation is referred to as the \emph{force balance} as 
		it balances two ``vectors'' resembling the principle of actio and reactio. 
		The second identity is called the \emph{entropy balance} since it states that 
		the temporal derivative of the entropy $\calS(\bfZ)$ on the right-hand side 
		is the sum of the entropy-production potential and its dual potential 
		evaluated at $\dot \bfZ$ and $\rmD \calS(\bfZ)$, respectively. 
		And the third inclusion, which is the \emph{rate equation} 
		in case the subdifferential is single-valued, is the original 
		gradient-flow formulation from \eqref{eq:general-system}.

	We now specify the EERDS \eqref{eq:general-system} for the 
	concentrations $\bfc = (c_1, \dotsc, c_I)$ and the 
	internal energy $u$, 
	\begin{equation}
		\label{eq:EERDS.gfe}
		\binom{\dot\bfc}{\dot u} = \pl_{(\bfy,v)} \calP^*\big(\bfc,u; \rmD_\bfc
		\calS(\bfc,u), \rmD_u \calS(\bfc,u)\big), 
	\end{equation}
	by introducing a suitable choice for 
	the dual entropy-production potential $\calP^*(\bfc,u;\bfy,v)$. 
	This approach was already outlined in a similar setting in \cite{Miel11GSRD}, 
	for the case of additional bulk-interface interactions in \cite{Miel13TMER}, 
	and in the context of mass-action reaction systems in \cite{MiPeSt21EDPC}. 
	
	We prescribe the additive form $\calP^*=\calP^*_\mathrm{diff} +
	\calP^*_\mathrm{reac}$ with the diffusive dual entropy-production potential 
	\begin{align}
		\label{eq:entropy-prod-pot-diff}
		\calP^*_\mathrm{diff}(\bfc,u;\bfy,v) &\coleq \frac12 \int_\Omega 
		\binom{\nabla \bfy - v \bfq \oti \nabla \Psi}{\nabla v} : \bbM(\bfc,u) 
		\binom{\nabla \bfy - v \bfq \oti \nabla \Psi}{\nabla v} \dd x, 
	\end{align}
	where $\bbM (\bfc, u) \in \R^{(I+1)\ti (I+1)}$ is a symmetric and positive 
	semidefinite mobility matrix modeling diffusion and
	heat transfer. We additionally demand that $\bbM (\bfc, u)$ is positive  
	definite if $c_i > 0$ and $u > 0$ hold for all $i = 1, \dotsc, I$. 
	In the case of neutral species, i.e., $\bfq = 0$, this model 
	reduces to the situation studied in \cite{FHKM22GEAE,Hopf_2022}.

	We rely on the Marcelin--de Donder kinetics 
	(\cite{Marc15CECP}, \cite[Definition 3.3]{Fein72CKCC}) 
	for chemical reactions of mass-action type; 
	see	also \cite{MPPR17NETP} for a derivation via the theory of large deviations). 
	The reactive dual entropy-production potential thus reads 
	\begin{align}
		\calP^*_\mathrm{reac}(\bfc,u;\bfy,v)&\coleq \int_\Omega
		P^*_\mathrm{reac}\big(\bfc(x),u(x);\bfy(x),v(x)\big)\dd x \quad \text{with} 
		\nonumber \\
		P^*_\mathrm{reac}(\bfc,u;\bfy,v)&\coleq \sum_{r=1}^{R} k_r \prod_{i = 1}^{I} 
		\Big( \frac{c_i}{w_i(u)} \Big)^\frac{\alpha_i^r + \beta_i^r}{2}
		\sfC^*\big( \bfgamma^r  {\cdot}  \bfy\big), \label{eq:entropy-prod-pot-reac}
	\end{align}
	where $w_i:{[0,\infty)}\to {(0,\infty)}$ are increasing and concave 
	equilibrium concentrations and $\sfC^*(\xi) \coleq
	4\cosh(\xi/2) - 4= 2\big( \ee^{\xi/2} -2 + \ee^{-\xi/2}\big)$. In this context, 
	we consider $R$ pairs of forward and backward reactions of species $(C_i)_i$:
	\begin{align*}
		\sum_{i=1}^{I} \alpha_i^r C_i \ \rightleftharpoons \ \sum_{i=1}^{I} \beta_i^r C_i, 
        \qquad r = 1, \dotsc, R. 
	\end{align*}
	The chemical reactions are assumed to preserve the
	charge, i.e., the stoichiometric vectors $\bfalpha^r$ and $\bfbeta^r$ satisfy 
	$\bfgamma^r {\cdot} \bfq = 0$ for $\bfgamma^r \coleq \bfbeta^r- \bfalpha^r \in \Z^{I}$ 
	and $r=1,\ldots, R$. 
	In terms of the matrix 
	$
	\bfGamma \coleq ( \gamma_i^r )_{ri} \in \bbZ^{R \times I},
	$
	this means $\bfq \in \ker \bfGamma$. In general, there may exist 
	$N \coleq \dim \ker \bfGamma$ linearly independent conservation laws. 
	As mentioned in Section \ref{se:Intro}, we assume for simplicity that 
	the conservation of charge is the only conservation law of the reaction system, 
    i.e., $N = 1$ and $\ker \bfGamma = \operatorname{span} \bfq$. 
	
	The following remark shows that stationary points of \eqref{eq:EERDS.gfe} and 
    critical points of the total entropy $\calS$ satisfying the constraints 
    \eqref{eq:constraints} coincide in the present setting. 
	
	\begin{remark}[Critical points and stationary states]
		\label{rem:crit-stat}
		Let $(\bfc, u) \in \FF$ satisfy \eqref{eq:constraints} (with $\FF$ defined in 
        \eqref{eq:function-space}). If $(\bfc, u)$ is a critical point 
		of the total entropy $\calS$ subject to \eqref{eq:constraints}, i.e., 
		there exist constants $\eta, \kappa \in \bbR$ such that 
		\begin{align*}
			\rmD \calS(\bfc, u) = \eta \rmD \calE(\bfc, u) 
			+ \kappa \rmD \calQ(\bfc, u),
		\end{align*}
		then $(\bfc, u)$ is a stationary state of \eqref{eq:EERDS.gfe}, i.e., 
		\begin{align*}
			\partial_{(\bfy, v)} \calP^\ast \big(\bfc, u; 
			\rmD_\bfc \calS(\bfc, u), \rmD_u \calS(\bfc, u) \big) = 0. 
		\end{align*}
		Vice versa, if $(\bfc, u)$ is a stationary state of \eqref{eq:EERDS.gfe}, 
        then it is a critical point of the 
		total entropy $\calS$ under the constraints \eqref{eq:constraints}.  
		
		To see this, let $(\bfc, u) \in \FF$ be a critical 
		point of $\calS$ under the constraints \eqref{eq:constraints}. 
		It follows that $\rmD\calS$ is a linear combination of 
		$\rmD \calE$ and $\rmD \calQ$ at $(\bfc, u)$. 
		By the definition of $\calP^\ast_\mathrm{diff}$ and $\calP^\ast_\mathrm{reac}$, 
		we find together with \eqref{eq:derivatives-energy-charge} 
		that $\pl_{(\bfy,v)}\calP^* \big( \bfc, u; \eta \rmD\calE (\bfc, u) 
		+ \kappa \rmD\calQ (\bfc, u) \big) = 0$ holds 
		for all $\eta, \kappa \in \R$, hence, $(\bfc, u)$ is a stationary state 
		of \eqref{eq:EERDS.gfe}. 
		
		Let now $\bfZ = (\bfc, u) \in \FF$ be a stationary 
		state of \eqref{eq:EERDS.gfe}, i.e.\ 
		$
		\partial_\bfW \calP^\ast \big(\bfZ; \rmD \calS(\bfZ) \big) = 0. 
		$
		As $\dot \bfZ = 0$, we infer from 
		\eqref{eq:gradient-flow-form-power} and $\calP(\bfZ; 0) = 0$ that 
		\begin{align*}
			\calP_\mathrm{diff}^\ast \big(\bfZ; \rmD \calS(\bfZ) \big) = 
			\calP_\mathrm{reac}^\ast \big(\bfZ; \rmD \calS(\bfZ) \big) = 0. 
		\end{align*}
		Due to \eqref{eq:entropy-prod-pot-reac}, we obtain 
        $\bfy = -\rmD_\bfc S(\bfc, u) \in \ker \bfGamma = \operatorname{span} \bfq$ 
        for a.e.\ $x \in \Omega$. This yields $\bfy = k(x) \bfq$ with some $k(x) \in \bbR$, 
		$x \in \Omega$. 
		Invoking the hypothesis that the mobility matrix $\bbM(\bfc,u)$ is positive 
		definite for $(\bfc,u)\in(0, \infty)^{I+1}$, we deduce from 
		\eqref{eq:entropy-prod-pot-diff} that 
		\begin{align*}
			\binom{\nabla \bfy - v \bfq \oti \nabla \Psi}{\nabla v} = 0 \quad \text{a.e.\ in}\ \Omega.
		\end{align*}
		This shows that $v$ is constant in $\Omega$ and that 
		$y_i = v q_i \Psi + \kappa q_i$ with a constant $\kappa \in \bbR$ due to 
		$\bfy = k(x) \bfq$. All together, we arrive at 
		\begin{align*}
			-\rmD S(\bfc, u) = \binom{\bfy}{v} = v \binom{\bfq \Psi}{1} + \kappa \binom{\bfq}{0}, 
		\end{align*}
		which proves that $\rmD \calS(\bfc, u)$ is a linear combination of 
		$\rmD \calE(\bfc, u)$ and $\rmD \calQ(\bfc, u)$. Therefore, $(\bfc, u)$ 
		is a critical point of $\calS$ under the constraints \eqref{eq:constraints}. 
	\end{remark}

	\paragraph*{Acknowledgments.} 
    The research of M.K.\ was funded in whole 
	by the Austrian Science Fund (FWF) 10.55776/J4604. 
    The research of A.M.\ was partially supported by the
	DFG through the Berlin Mathematics Research Center MATH+ (EXC-2046/1, project
	ID: 390685689) subproject ``DistFell''. 
    For the purpose of Open Access, the authors have applied 
    a CC BY public copyright license to any 
	Author Accepted Manuscript (AAM) version arising from this submission. 
	
	\footnotesize
	
	\addcontentsline{toc}{section}{References}
	
	\bibliographystyle{alpha_AMs}
	\bibliography{EERDS_Equil}

\end{document}